\documentclass[11pt,a4paper]{amsart}

\usepackage[T1]{fontenc}
\usepackage{lmodern}
\usepackage{amsmath}
\usepackage{amssymb,amscd,amsxtra,calc}
\usepackage{mathtools}
\usepackage{mathrsfs}
\usepackage{appendix}
\usepackage{multirow,longtable}
\usepackage{enumitem}
\usepackage{comment}
\usepackage[all]{xy}
\usepackage{tikz-cd}
\usepackage{microtype}
\usepackage[colorinlistoftodos]{todonotes}
\usepackage[colorlinks,linkcolor=blue,anchorcolor=blue,citecolor=green,backref=none]{hyperref}

\theoremstyle{plain}
\newtheorem{theorem}{Theorem}[section]

\newtheorem*{conjecture}{Conjecture}
\newtheorem{corollary}[theorem]{Corollary}
\newtheorem{lemma}[theorem]{Lemma}
\newtheorem{proposition}[theorem]{Proposition}

\theoremstyle{definition}
\newtheorem{example}[theorem]{Example}
\newtheorem{definition}[theorem]{Definition}

\newtheorem*{notation*}{Notation and Terminology}
\newtheorem{remark}[theorem]{Remark}

\theoremstyle{remark}

\numberwithin{equation}{section}
\newcommand{\Z}{\mathbb{Z}}

\newcommand{\nc}{\newcommand}
\nc{\cH}{{\mathcal H}}
\nc{\cA}{{\mathcal A}}
\nc{\cG}{{\mathcal G}}
\nc{\cC}{{\mathcal C}}
\nc{\cO}{{\mathcal O}}
\nc{\cI}{{\mathcal I}}
\nc{\cB}{{\mathcal B}}
\nc{\cY}{{\mathcal Y}}
\nc{\cK}{{\mathcal K}}
\nc{\cX}{{\mathcal X}}
\nc{\cS}{{\mathcal S}}
\nc{\cE}{{\mathcal E}}
\nc{\cF}{{\mathcal F}}
\nc{\cZ}{{\mathcal Z}}
\nc{\cQ}{{\mathcal Q}}
\nc{\cN}{{\mathcal N}}
\nc{\cP}{{\mathcal P}}
\nc{\cL}{{\mathcal L}}
\nc{\cM}{{\mathcal M}}
\nc{\cT}{{\mathcal T}}
\nc{\cW}{{\mathcal W}}
\nc{\cU}{{\mathcal U}}
\nc{\cJ}{{\mathcal J}}
\nc{\cV}{{\mathcal V}}

\nc{\bH}{{\mathbb H}}
\nc{\bA}{{\mathbb A}}
\nc{\bG}{{\mathbb G}}
\nc{\bC}{{\mathbb C}}
\nc{\bO}{{\mathbb O}}
\nc{\bI}{{\mathbb I}}
\nc{\bB}{{\mathbb B}}
\nc{\bY}{{\mathbb Y}}
\nc{\bK}{{\mathbb K}}
\nc{\bX}{{\mathbb X}}
\nc{\bS}{{\mathbb S}}
\nc{\bE}{{\mathbb E}}
\nc{\bF}{{\mathbb F}}
\nc{\bZ}{{\mathbb Z}}
\nc{\bQ}{{\mathbb Q}}
\nc{\bN}{{\mathbb N}}
\nc{\bP}{{\mathbb P}}
\nc{\bL}{{\mathbb L}}
\nc{\bM}{{\mathbb M}}
\nc{\bT}{{\mathbb T}}
\nc{\bW}{{\mathbb W}}
\nc{\bU}{{\mathbb U}}
\nc{\bD}{{\mathbb D}}
\nc{\bJ}{{\mathbb J}}
\nc{\bV}{{\mathbb V}}
\nc{\bbZ}{{\mathbb Z}}
\nc{\bR}{{\mathbb R}}
\nc{\fr}{{\rightarrow}}
\nc{\co}{{\nabla}}

\nc{\cu}{{\overline{\nabla}}}

\title{Varieties Admitting Polarized Galois Endomorphisms}

\author{Zhiyuan Jiang}
\address{Department of Mathematics, Purdue University, West Lafayette, IN 47907, USA}
\email{jian1206@purdue.edu}

\author{Yujie Luo}

\address{School of Mathematical Sciences, University of Science and Technology of China, Hefei 230026, People's Republic of China}

\email{yujieluo96@gmail.com,\ lyj96@ustc.edu.cn}

\date{}


\begin{document}

\begin{abstract}
Let $X$ be an $n$-dimensional smooth complex projective variety admitting an int-amplified Galois endomorphism. We prove that its maximal rationally connected fibration is represented by a smooth toric fibration over a smooth $Q$-abelian variety. Consequently, if $X$ is rationally connected, then it is toric. If it also has Picard number one, then $X\cong \mathbf{P}^n$. As another application, we prove that such $X$ is of log Calabi-Yau type as conjectured by Gongyo.
\end{abstract}

\maketitle

\tableofcontents

\section{Introduction}

We work over the field of complex numbers $\mathbb{C}$.

\medskip

For an integer $q>1$, a surjective endomorphism $f:X\to X$ of a projective variety is \emph{$q$-polarized} if $f^*H\sim qH$ for some ample Cartier divisor $H$. It is \emph{polarized} if it is $q$-polarized for some integer $q>1$. Understanding the geometry of projective varieties admitting polarized endomorphisms is a fundamental problem at the intersection of algebraic geometry and algebraic dynamics.

One of the oldest problems in this direction is the following well-known folklore conjecture, dating back to the 1980s (see Hwang and Nakayama \cite[Conjecture~1.1]{HN11} for an equivalent statement).

\begin{conjecture}\label{conj: main}
Let $X$ be an $n$-dimensional Fano manifold of Picard number one. If $X$ admits a non-isomorphic surjective endomorphism, then $X\cong\mathbf{P}^n$.
\end{conjecture}

The conjecture is known for rational homogeneous manifolds \cite{PS89,HM99}. Hwang and Nakayama obtained a related toric criterion \cite[Theorem~2.1]{HN11}. Using Bott vanishing, Kawakami and Totaro proved the conjecture for Fano threefolds and for Fano fourfolds of index greater than one \cite[Theorem~A]{KT25}. Fakhruddin asked more generally whether every smooth projective rationally connected variety admitting a polarized endomorphism is toric \cite[Question~4.4]{Fakh03}. A surjective endomorphism $f:X\to X$ of a projective variety is \emph{int-amplified} if $f^*H-H$ is ample for some ample Cartier divisor $H$ (see Definition~\ref{def: int-amplified}). Every polarized endomorphism is int-amplified. Meng and Zhong later asked the corresponding question for int-amplified endomorphisms \cite[Question~1.2]{MZ23}.

\medskip

More generally, the existence of a polarized endomorphism imposes a strong restriction on Kodaira dimension. It is known that a smooth projective variety $X$ admitting a polarized endomorphism satisfies $\kappa(X)\leq 0$ (\cite[Theorem~4.2]{Fakh03}), a restriction later extended to normal projective varieties by Nakayama and Zhang \cite[Theorem~1.3]{NZ10}. Moreover, a smooth projective variety $X$ as above with $\kappa(X) = 0$ admits a finite quasi-\'{e}tale cover by an abelian variety, that is, $X$ is \emph{$Q$-abelian} (see Definition~\ref{def: quasi-etale}).

The \emph{maximal rationally connected} (MRC) fibration provides further information on the structure of projective varieties admitting polarized endomorphisms \cite{Camp92,KMM92b, Nak10,Zhang10}. Meng and Zhang proved that a suitable MRC fibration can be made equivariant with respect to a given polarized endomorphism and that its base is $Q$-abelian \cite[Proposition~1.6]{MZ18}. The Albanese morphism gives a complementary form of this structure. This aspect builds on Kawamata's work and on later results of Krieger, Reschke, Cascini, Meng and Zhang \cite{Kawa81,KR17,MZ18,CMZ20}. Meng extended this picture to int-amplified endomorphisms: the special MRC fibration has a $Q$-abelian base, is equidimensional, and has irreducible rationally connected fibers, while the Albanese morphism is surjective and equidimensional with irreducible fibers \cite[Theorems~1.7 and 1.8(1)]{Meng20}. Yoshikawa further proved that, after a finite quasi-\'{e}tale cover, the total space is of Fano type over its Albanese variety \cite[Theorem~1.3]{Yosh21b}. These results led Meng and Zhong to ask whether, for a $\mathbb{Q}$-factorial klt projective variety admitting an int-amplified endomorphism, one can pass to a finite quasi-\'{e}tale cover whose Albanese morphism is a \emph{toric fibration} (see Definition~\ref{def: toric fibration}) over an abelian variety \cite[Question~6.4]{MZ24}.

\medskip

We consider smooth projective varieties with int-amplified endomorphisms and impose the additional Galois condition, meaning that the induced extension of function fields is Galois (see Definition~\ref{def: galois}). Aprodu, Kebekus, and Peternell used this condition to study ramified coverings and endomorphisms of Fano manifolds \cite{AKP08}.

Under the Galois hypothesis, our main theorem strengthens the known MRC structure: the MRC fibration is represented by a smooth toric fibration over a smooth $Q$-abelian variety, and the int-amplified endomorphism descends to the base.

\begin{theorem}\label{thm: main}
Let $X$ be a smooth complex projective variety and $f:X\to X$ an int-amplified Galois endomorphism. Then the maximal rationally connected fibration of $X$ is represented by a toric fibration $h:X\to C$, where $C$ is a smooth $Q$-abelian variety.

Moreover, there is a unique int-amplified \'{e}tale Galois endomorphism $f_C:C\to C$ such that $h\circ f = f_C\circ h$, and $\mathrm{Gal}(f_C)$ is abelian. If $f$ is $q$-polarized, then $f_C$ is $q$-polarized.
\end{theorem}

More precisely, let $D$ be the relative toric boundary of $h$. Then all fibers of $(X,D)\to C$ are isomorphic to a fixed smooth projective toric pair $(F,\partial F)$, where $T:= F\setminus\partial F$ is the dense torus. There exist an abelian variety $A$, a connected finite \'{e}tale Galois cover $\pi:A\to C$, and a principal $T$-bundle $P\to A$ such that, over $A$,
$$(X\times_C A,D\times_C A)  \cong (P\times^T F, P\times^T\partial F).$$
The projection $X\times_C A\to A$ is the Albanese morphism of $X\times_C A$ and a \emph{split toric fibration} in the sense of Definition~\ref{def: toric fibration}. Thus Theorem~\ref{thm: main} answers \cite[Question~6.4]{MZ24} in the smooth int-amplified Galois setting.

The smoothness assertion underlying Theorem~\ref{thm: main} does not require the Galois hypothesis. Building on Meng's results \cite[Theorems~1.7 and 1.8(1)]{Meng20}, we give a direct proof of the following proposition.

\begin{proposition}\label{prop: smooth-mrc}
Let $X$ be a normal complex projective variety and let $f:X\to X$ be an int-amplified endomorphism. Then the Albanese morphism $\mathrm{alb}_X$ is surjective and flat, and every geometric fiber is normal and integral. 

Moreover, if $X$ is smooth, then the Albanese morphism is smooth, and the maximal rationally connected fibration of $X$ is represented by a smooth projective morphism $h:X\to C$ with rationally connected fibers, where $C$ is a smooth $Q$-abelian variety.
\end{proposition}

The same argument combined with \cite[Theorem~1.1]{Zhuang24} implies that if $X$ is of klt type, then the fibers of $\mathrm{alb}_X$ is of klt type (see Remark~\ref{rem: klt case}).

As a direct corollary of Theorem~\ref{thm: main}, we obtain the following result.

\begin{corollary}\label{cor: albanese toric cover}
Let $X$ be a smooth complex projective variety admitting an int-amplified Galois endomorphism. Then there is a finite \'{e}tale Galois cover $\tau: \widetilde X\to X$ from a smooth variety $\widetilde X$ such that the Albanese morphism of $\widetilde X$ is a split toric fibration.
\end{corollary}

Related smoothness and local-triviality results are known under positivity assumptions not imposed in Proposition~\ref{prop: smooth-mrc}. Cao proved local triviality of the Albanese morphism when the anticanonical divisor is nef \cite{Cao19}, while Matsumura and Wang obtained a corresponding structure theorem for maximal rationally connected fibrations of klt pairs with nef anti-log-canonical divisor \cite{MW25}. Under an additional assumption on the logarithmic tangent bundle, Druel and Lo Bianco obtained a smooth locally trivial family of toric pairs \cite[Theorem~1.1]{DL22}.

Other toric characterizations assume an invariant boundary, either for polarized endomorphisms \cite{MZ19} or for int-amplified endomorphisms \cite[Theorem~1.5]{MZ23}, while Yoshikawa used a splitting condition for direct images of line bundles \cite[Theorem~1.1]{Yosh21a}. Under a prescribed invariant reduced log Calabi--Yau boundary, Meng and Zhong obtained a splitting toric fibration after a quasi-\'{e}tale cover \cite[Theorem~1.5]{MZ24}. Finite-quotient toric descriptions of log Calabi--Yau pairs, both absolutely and over an abelian variety, were obtained in \cite[Theorem~5.1]{MYY25} and \cite[Theorem~1.1]{MYY24}. Theorem~\ref{thm: main} instead assumes that $X$ is smooth and that the int-amplified endomorphism is Galois. It requires neither nefness of the anticanonical divisor nor the existence of an invariant boundary.

\medskip

When $X$ is rationally connected, the base of its MRC fibration is a point. Theorem~\ref{thm: main} therefore gives the following result. Under the Galois hypothesis, this answers the questions posed by Fakhruddin \cite[Question~4.4]{Fakh03} and by Meng and Zhong \cite[Question~1.2]{MZ23}.

\begin{theorem}\label{thm: rationally connected}
Let $X$ be a smooth projective rationally connected variety admitting an int-amplified Galois endomorphism. Then $X$ is a toric variety. In particular, $\mathrm{Cox}(X)$ is a polynomial ring.
\end{theorem}

Every smooth Fano variety is rationally connected \cite{Camp92,KMM92a}, so Theorem~\ref{thm: rationally connected} includes the smooth Fano case. The toric conclusion of Theorem~\ref{thm: rationally connected} was previously known for surfaces \cite[Theorem~3]{Nak02}, smooth Fano threefolds \cite[Theorem~1.4]{MZZ22}, several classes of smooth Fano fourfolds \cite{JZ23,SZ24}, and smooth rationally connected threefolds of Picard number two \cite[Theorem~1.3]{CMZ26}.

In Picard number one, every non-isomorphic surjective endomorphism is int-amplified, and a smooth projective toric variety is projective space. We therefore obtain the folklore conjecture under the Galois hypothesis.

\begin{theorem}\label{thm: projective-space}
Let $X$ be an $n$-dimensional smooth Fano variety of Picard number $\rho(X) = 1$. If $X$ admits a Galois endomorphism of degree greater than one, then $X\cong \mathbf{P}^n$.
\end{theorem}

Once $X$ is identified with projective space, the endomorphism in Theorem~\ref{thm: projective-space} is described by the classification in \cite[Theorem~1.1]{LZ26}. The smoothness assumption in Theorem~\ref{thm: rationally connected} cannot be removed (see Example~\ref{ex: klt-counterexample}).

\medskip

Let $X$ be a normal variety. A log pair $(X,B)$ is called a \emph{log Calabi--Yau} pair if $(X,B)$ is \emph{log canonical} (cf. \cite{KM98}) and $K_X+B\sim_\mathbb{Q} 0$. As another application of our main theorem, we confirm a Conjecture of Gongyo (cf. \cite[Conjecture~1.2]{BG17}) for Galois self-covers on $X$. This extends a result of the second named author and Zhang \cite[Theorem ~1.4]{LZ26}.

\begin{theorem}\label{thm:log-cy}
Let $X$ be a smooth projective variety and $f:X\to X$ an int-amplified finite Galois cover. Then there exists a divisor $D$ such that $(X,D)$ is a log Calabi--Yau pair.
\end{theorem}

\medskip

\noindent\textbf{Outline of the proofs.} We first prove Proposition~\ref{prop: smooth-mrc}. Equivariance shows that the induced endomorphism of the Albanese variety maps the locus of smooth fibers into itself. The complement of this nonempty open set must then be empty by Meng's lemma (cf. \cite[Lemma~4.1]{Meng20}), so the Albanese morphism is smooth. We next pass to a suitable finite \'{e}tale Galois cover, whose Albanese morphism coincides with its MRC fibration. Descending this smooth fibration gives the required MRC fibration on $X$.

On this cover, we use the Galois hypothesis to study fibers over periodic points of the Albanese endomorphism. Inspired by the linear-lifting method of Luo and Zhang \cite[Lemma~3.1]{LZ26}, we use reflection-group invariant theory to express every Cox relation in terms of relations pulled back around a periodic orbit. After sufficiently many iterations, int-amplification makes pullback increase the least total degree of a monomial occurring in a nonzero relation, contradicting the existence of a relation for which this degree is minimal. The Cox rings are therefore polynomial, so these fibers are toric by \cite[Corollary~2.10]{HK00}. Since their polarized isomorphism types are finite in number and periodic points are dense, the family is \'{e}tale locally trivial over a nonempty open subset. Local triviality passes forward under the finite equivariant maps, so the locus where it fails contains its inverse image. As in the smoothness argument, Meng's lemma forces this locus to be empty. Thus the family is \'{e}tale locally trivial over the entire Albanese variety with a fixed toric fiber $F$.

It remains to make the toric structures compatible across the family. We first eliminate the finite monodromy on the fiberwise Picard groups by base change along an isogeny, making the Cox gradings consistent across the fibers. A further such base change removes the obstruction to gluing the Cox rings into a relative algebra. In the degrees of the Cox generators, the nontrivial eigenspaces of suitable reflections are one-dimensional. Since they are defined by the group action, these generator lines glue to line bundles. Generators fixed by the reflection group are recovered by pullback from smaller degrees with respect to the integer grading in Definition~\ref{def: relative degree}. Int-amplification makes the decrease strict, so induction supplies all generators. The diagonal transition functions of these line bundles descend through the Cox quotient and give the split presentation $P\times^T F$, where $T$ is the dense torus of $F$.

Finally, the Galois group of the finite \'{e}tale cover commutes with the Galois action along the fibers. We use this compatibility to choose the generators so that the covering group permutes their coordinate divisors. The union of these divisors is an invariant toric boundary, which descends to $X$ and makes its MRC fibration toric. We also show that the endomorphism descends to an \'{e}tale Galois endomorphism of the MRC base. Its Galois group is identified with the kernel of an isogeny and is therefore abelian.

\vspace{2mm}

{\bf Acknowledgements.}
Zhiyuan Jiang would like to thank James M\textsuperscript cKernan for many helpful discussions, and valuable suggestions. Zhiyuan Jiang is also grateful to Takumi Murayama for his support and very helpful suggestions on our main theorems. The authors would like to thank Qianli Fang, Fanjun Meng, Sheng Meng, Junyi Xie, De-Qi Zhang, Lei Zhang, Yiming Zhu for their interest and many helpful discussions and comments on this work.

\medskip

\section{Preliminaries}
\label{sec: preliminaries}

We adopt the standard notation and terminology of \cite{Har77,Laz04}. We write $\mathbb{C}(X)$ for the function field of a variety $X$.

\medskip

In this section, we recall and establish the preliminary results on finite covers, Albanese morphisms, int-amplified maps, toric fibrations, Cox rings, and complex reflection groups that will be used in the proofs below.

\medskip

\subsection{Covers, irregularity, and rational connectedness}

\begin{definition}\label{def: galois}
A finite surjective morphism $\pi:Y\to X$ between normal varieties is \emph{Galois} if the induced extension $\mathbb C(Y)/\pi^*\mathbb C(X)$ is Galois. In this case, we write $\mathrm{Gal}(\pi):= \mathrm{Gal}(\mathbb C(Y)/\pi^*\mathbb{C}(X))$. By normality, this group acts naturally on $Y$ over $X$, and the natural morphism $Y/\mathrm{Gal}(\pi)\to X$ is an isomorphism.
\end{definition}

\begin{definition}\label{def: quasi-etale}
A finite surjective morphism $\pi:Y\to X$ between normal varieties is \emph{quasi-\'{e}tale} if there is a closed subset $Z\subseteq X$ of codimension at least two such that $\pi^{-1}(X\setminus Z)\to X\setminus Z$ is \'{e}tale. A normal projective variety is \emph{$Q$-abelian} if it admits a finite quasi-\'{e}tale cover by an abelian variety.
\end{definition}

\begin{definition}[{\cite{NZ10}}]\label{def: augmented irregularity}
For a normal projective variety $X$, set $q(X)=h^1(X,\mathcal O_X)$. The \emph{augmented irregularity} of $X$ is
$$\widetilde q(X)= \sup\{q(Y):Y\to X \text{ is a connected finite quasi-\'{e}tale cover}\}.$$
\end{definition}

If $X$ is smooth, purity of the branch locus shows that every finite quasi-\'{e}tale cover of $X$ is \'{e}tale \cite[Expos\'{e}~X, Th\'{e}or\`{e}me~3.1]{SGA1}. In particular, a smooth projective variety is $Q$-abelian if and only if it admits a finite \'{e}tale cover by an abelian variety.

\medskip

For a projective variety $X$, let $N^1(X)$ denote the group of Cartier divisors modulo numerical equivalence, and put $N^1(X)_{\mathbb{R}}:= N^1(X)\otimes_{\mathbb{Z}}\mathbb{R}$. Its dimension $\rho(X):= \dim_{\mathbb{R}}N^1(X)_{\mathbb{R}}$ is the Picard number of $X$.

\begin{lemma}\label{lem: rc-basic}
Let $X$ be a smooth projective rationally connected variety. Then the associated complex analytic space $X^{\mathrm{an}}$ is simply connected, $H^j(X,\mathcal O_X)= 0$ for every $j>0$, and $\mathrm{Pic}(X)$ is free. Moreover,
$$\mathrm{Pic}(X)\xrightarrow{\sim} N^1(X).$$
\end{lemma}

\begin{proof}
Rational connectedness gives $H^0(X,\Omega_X^j)= 0$ for $j >0$ \cite[Corollary~IV.3.8]{Kol96} and $\pi_1(X^{\mathrm{an}})=1$ \cite[Proposition~2.3]{Koll00}. Hodge symmetry and Dolbeault theory give $H^j(X,\mathcal O_X)= 0$ for $j>0$. The exponential sequence therefore identifies $\mathrm{Pic}(X)$ with $H^2(X,\mathbb{Z})$ (cf. \cite[Lemma~4.13]{KM98}). Since $X^{\mathrm{an}}$ is simply connected, the universal coefficient theorem shows that $H^2(X,\mathbb{Z})$ is torsion-free. Since $H^1(X,\mathcal O_X)= 0$, the abelian variety $\mathrm{Pic}^0(X)$ has dimension zero and is therefore trivial. A numerically trivial line bundle has torsion image in $\mathrm{Pic}(X)/\mathrm{Pic}^0(X)$. Since $\mathrm{Pic}(X)$ is torsion-free, every numerically trivial line bundle is trivial. This proves the last assertion.
\end{proof}

We also need the following descent observation for rationally connected fibers.

\begin{lemma}\label{lem: free action on base}
Let $p: Y\to B$ be a surjective morphism whose fibers are smooth projective rationally connected varieties. Suppose that a finite group $G$ acts freely on $Y$ and acts on $B$ so that $p$ is $G$-equivariant. Then the action of $G$ on $B$ is free.
\end{lemma}

\begin{proof}
Suppose that a nonidentity element $\gamma\in G$ fixes a point $b\in B$. If $m$ is the order of $\gamma$, then the cyclic group $\langle\gamma\rangle$ preserves $Y_b$ and acts freely on it. Thus $Y_b/\langle\gamma\rangle$ is smooth, and the quotient map $u:Y_b\to Y_b/\langle\gamma\rangle$ is finite \'{e}tale of degree $m$. Rational connectedness gives $H^i(Y_b,\mathcal O_{Y_b})=0$ for $i > 0$, and hence $\chi(Y_b,\mathcal O_{Y_b})=1$. Since $u$ is \'{e}tale, $T_{Y_b}\cong u^*T_{Y_b/\langle\gamma\rangle}$. Hirzebruch--Riemann--Roch and the projection formula therefore give
$$1= \chi(Y_b,\mathcal O_{Y_b})= m\chi(Y_b/\langle\gamma\rangle,\mathcal O_{Y_b/\langle\gamma\rangle}).$$
This equality is impossible because $m>1$. Thus no nonidentity element of $G$ fixes a point of $B$.
\end{proof}

\medskip

We next discuss Albanese morphisms and special MRC fibrations for varieties admitting int-amplified endomorphisms. We begin with the definition. Recall that every surjective endomorphism of a projective variety is finite.

\begin{definition}\label{def: int-amplified}
A surjective endomorphism $f:X\to X$ of a projective variety is \emph{int-amplified} if $f^*H-H$ is ample for some ample Cartier divisor $H$.
\end{definition}

Every $q$-polarized endomorphism is int-amplified: if $f^*H\sim qH$ with $H$ ample and $q>1$, then $f^*H-H\sim(q-1)H$ is ample.

By \cite[Theorem~1.1]{Meng20}, a surjective endomorphism $f$ of a projective variety $X$ is int-amplified if and only if every eigenvalue of $f^*$ on $N^1(X)_{\mathbb R}$ has absolute value greater than one. In particular, $(f^*)^{-n}\to0$ for an int-amplified endomorphism. Int-amplification descends along equivariant surjective morphisms \cite[Lemma~3.4]{Meng20}.

By the special MRC fibration, we mean the Chow-reduction model defined in \cite[Definition~6.4]{Meng20}. It is unique up to an isomorphism of bases. Meng's special maximal rationally connected fibration for a variety admitting an int-amplified endomorphism has a $Q$-abelian base, is equidimensional, and has irreducible rationally connected fibers \cite[Theorem~1.7]{Meng20}. Together with the Albanese properties in \cite[Theorem~1.8]{Meng20}, this gives the following lemma.

\begin{lemma}\label{lem: maximal irregularity mrc}
Let $X$ be a smooth projective variety admitting an int-amplified endomorphism. If $q(X)= \widetilde q(X)$, then the Albanese morphism of $X$ is, up to an isomorphism of bases, its special MRC fibration. In particular, the Albanese morphism is equidimensional, and all its fibers are irreducible and rationally connected.
\end{lemma}

\begin{proof}
Let $\pi:X\to Y$ be the special MRC fibration. By \cite[Theorem~1.7(1)--(4)]{Meng20}, the variety $Y$ is $Q$-abelian, the morphism $\pi$ is equidimensional, and every fiber of $\pi$ is irreducible and rationally connected. Write $\mathrm{alb}_X:X\to\mathrm{Alb}(X)$ for the Albanese morphism.

The morphism $\mathrm{alb}_X$ is constant on every rationally connected fiber of $\pi$. Since $\pi$ has connected fibers and $Y$ is normal, the rigidity lemma \cite[Lemma~1.15]{Deba01} gives a morphism $h:Y\to\mathrm{Alb}(X)$ such that $\mathrm{alb}_X= h \circ\pi$. By \cite[Theorem~1.8(1)]{Meng20}, the morphism $\mathrm{alb}_X$ is surjective. Hence $h$ is surjective and $q(X)\leq\dim Y$.

Choose a finite quasi-\'{e}tale cover $\beta:B\to Y$ with $B$ an abelian variety, and let $X_B$ be the normalization of an irreducible component of $X\times_YB$ that dominates $X$. The induced morphism $X_B\to B$ is surjective. These morphisms form a commutative diagram:
$$
\begin{tikzcd}[column sep=large]
X_B \arrow[r] \arrow[d] & X \arrow[r, "\mathrm{alb}_X"] \arrow[d, "\pi"'] & \mathrm{Alb}(X)\\
B \arrow[r, "\beta"'] & Y \arrow[ur, "h"'] &
\end{tikzcd}
$$

The morphism $X_B\to X$ is quasi-\'{e}tale. Indeed, the branch locus $Z\subseteq Y$ of $\beta$ has codimension at least two in $Y$. The map $X_B\to X$ is \'{e}tale over $X\setminus\pi^{-1}(Z)$, and equidimensionality of $\pi$ gives $\mathrm{codim}_X(\pi^{-1}(Z))\geq2$. Purity of the branch locus therefore shows that $X_B\to X$ is \'{e}tale. Since $X_B\to B$ is surjective, we obtain
$$\dim Y=q(B)\leq q(X_B)\leq\widetilde q(X)=q(X)\leq\dim Y.$$
Here $X_B$ is smooth because it is \'{e}tale over $X$, and the first inequality follows because pullback along the surjective morphism $X_B\to B$ is injective on holomorphic one-forms. Thus $q(X)= \dim Y$.

The composite $h\circ\beta:B\to\mathrm{Alb}(X)$ is a surjective morphism between abelian varieties of the same dimension. After translating its target, it becomes a surjective homomorphism and hence an isogeny. In particular, it is finite. Since $\beta$ is surjective, every fiber of $h$ is finite. Thus $h$ is quasi-finite. Since it is projective, it is finite.

The fibers of $\mathrm{alb}_X$ are connected by \cite[Theorem~1.8(1)]{Meng20}. If $\deg h>1$, the fiber of $\mathrm{alb}_X$ over a general point would be the disjoint union of the nonempty $\pi$-fibers over the distinct points of an $h$-fiber. Hence $\deg h=1$. Since $h$ is finite and $\mathrm{Alb}(X)$ is normal, $h$ is an isomorphism.
\end{proof}

The following lemma constructs a finite \'{e}tale cover whose irregularity equals its augmented irregularity and to which the given endomorphism lifts.

\begin{lemma}\label{lem: invariant cover}
Let $X$ be a smooth complex projective variety and $f:X\to X$ an int-amplified endomorphism. Then the supremum defining $\widetilde q(X)$ is a maximum and satisfies $\widetilde q(X)\leq\dim X$. There is a connected finite \'{e}tale Galois cover $\tau: \widetilde X\to X$ satisfying
$$q(\widetilde X)= \widetilde q(\widetilde X)= \widetilde q(X),$$
such that $f$ lifts to an endomorphism $\widetilde f:\widetilde X\to\widetilde X$ making the following square Cartesian:
\begin{equation}\label{eq: cartesian-lift}
\begin{tikzcd}
\widetilde X \arrow[r, "\widetilde f"] \arrow[d, "\tau"'] & \widetilde X \arrow[d, "\tau"]\\
X \arrow[r, "f"'] & X
\end{tikzcd}
\end{equation}
The endomorphism $\widetilde f$ is int-amplified. If $f$ is $q$-polarized, then $\widetilde f$ is $q$-polarized. If $f$ is Galois, then $\widetilde f$ is Galois.
\end{lemma}

\begin{proof}
Choose an ample Cartier divisor $H$ such that $f^*H-H$ is ample. Choose also a base point of $X^{\mathrm{an}}$ and a path from it to its image under $f$, and put $\Pi= \pi_1(X^{\mathrm{an}})$. These choices define $f_*:\Pi\to\Pi$ up to an inner automorphism. This ambiguity is harmless because all subgroups considered below are normal. By the Riemann existence theorem, we identify finite \'{e}tale covers of $X$ with finite topological covers of $X^{\mathrm{an}}$.

\noindent\textbf{Step 1.} We prove that the supremum defining $\widetilde q(X)$ is finite and attained.

We begin with boundedness, without assuming that the supremum is attained. Let $Y\to X$ be a connected finite \'{e}tale cover. Passing to the connected Galois closure can only increase irregularity, because pullback on $H^1(\mathcal O)$ is injective for a finite surjective morphism by the trace. We may therefore replace $Y$ by this closure and assume that $Y\to X$ corresponds to a normal finite-index subgroup $N_0\subseteq\Pi$. Define
$$N_{i+1}=f_*^{-1}(N_i).$$
The map $\Pi/N_{i+1}\to\Pi/N_i$ induced by $f_*$ is injective, so the indices $[\Pi:N_i]$ are nonincreasing and eventually constant. Since $\Pi$ is finitely generated, it has only finitely many subgroups of any fixed finite index. The sequence is obtained by iterating the single operation $N\mapsto f_*^{-1}(N)$, so it is eventually periodic. Thus there are integers $m\geq 0$ and $s>0$ such that $N_{i+s}=N_i$ for every $i\geq m$. Put
$$N_\infty= \bigcap_{j=0}^{s-1}N_{m+j}.$$
Taking inverse images cyclically permutes these $s$ subgroups. Therefore $N_\infty$ is normal and $f_*^{-1}(N_\infty)=N_\infty$.

Let $X_i\to X$ and $X_\infty\to X$ be the covers corresponding to $N_i$ and $N_\infty$. Since $N_0$ corresponds to $Y$, we identify $X_0$ with $Y$. For every $i$, the cover $X_{i+1}$ is the connected component of $X\times_{f,X}X_i$ determined by the chosen base point. Projection to $X_i$ gives a finite surjective lift $X_{i+1}\to X_i$. Pullback on $H^1(\mathcal O)$ is injective by the trace, so $q(X_{i+1})\geq q(X_i)$ and hence $q(X_i)\geq q(Y)$ for every $i$. Since $X_\infty$ dominates $X_m,\ldots,X_{m+s-1}$, it follows that $q(X_\infty)\geq q(Y)$.

The relation $f_*^{-1}(N_\infty)=N_\infty$ lifts $f$ to an endomorphism of $X_\infty$. This lift is int-amplified because the pullbacks of $H$ and $f^*H-H$ are ample. Its Albanese morphism is therefore surjective by \cite[Theorem~1.8(1)]{Meng20}. Thus
$$q(Y)\leq q(X_\infty)= \dim\mathrm{Alb}(X_\infty)\leq\dim X.$$
This holds for every finite \'{e}tale cover $Y\to X$. The set of possible irregularities is a nonempty bounded set of integers, so its supremum is a maximum.

\medskip

\noindent\textbf{Step 2.} We construct a Galois cover realizing $\widetilde q(X)$ and a lift of $f$ for which the square \eqref{eq: cartesian-lift} is Cartesian.

Choose a cover realizing the maximum and replace it by its connected Galois closure. Apply the inverse-image construction above to the corresponding normal subgroup $N_0$. Irregularity cannot decrease along the resulting finite maps and cannot exceed $\widetilde q(X)$. Hence every cover in the tower, including the cover associated with $N_\infty$, has irregularity $\widetilde q(X)$. Denote the cover associated with $N_\infty$ by $\tau:\widetilde X\to X$.

Let $Z=X\times_{f,X,\tau}\widetilde X$ be the pullback of $\tau$ along $f$. The connected component of $Z$ determined by the chosen base point corresponds to $f_*^{-1}(N_\infty)=N_\infty$. Its degree over $X$ is therefore $[\Pi:N_\infty]$, which is also the total degree of $Z\to X$. Hence $Z$ is connected and is the cover associated with $N_\infty$. Identifying $Z$ with $\widetilde X$ gives the Cartesian square \eqref{eq: cartesian-lift}. The endomorphism $\widetilde f$ is a base change of $f$, so it is int-amplified. If $f$ is $q$-polarized, then $\widetilde f$ is $q$-polarized by the pullback of a polarizing divisor. If $f$ is Galois, then the connected base change $\widetilde f$ is Galois.

\medskip

\noindent\textbf{Step 3.} We prove that $\widetilde q(\widetilde X)=q(\widetilde X)$.

Every finite quasi-\'{e}tale cover of the smooth variety $\widetilde X$ is \'{e}tale by purity, and its composite with $\widetilde X\to X$ is a finite \'{e}tale cover of $X$. Thus $\widetilde q(\widetilde X)\leq\widetilde q(X)=q(\widetilde X)$, while the reverse inequality follows from the identity cover.
\end{proof}

\begin{lemma}\label{lem: cartesian galois actions}
Let $\tau:Y\to X$ be a finite \'{e}tale Galois cover of connected normal varieties, let $f:X\to X$ be a finite Galois endomorphism, and let $\widetilde f:Y\to Y$ be an endomorphism lifting $f$. Suppose that the following square is Cartesian:
$$
\begin{tikzcd}
Y \arrow[r, "\widetilde f"] \arrow[d, "\tau"'] & Y \arrow[d, "\tau"]\\
X \arrow[r, "f"'] & X
\end{tikzcd}
$$
Then $\widetilde f$ is finite and Galois. There is a unique automorphism
$$\alpha:\mathrm{Gal}(\tau)\to\mathrm{Gal}(\tau)$$
such that
$$\widetilde f\circ\gamma= \alpha(\gamma)\circ\widetilde f$$
for every $\gamma\in\mathrm{Gal}(\tau)$. Moreover, every element of $\mathrm{Gal}(\tau)$ commutes with every element of $\mathrm{Gal}(\widetilde f)$.
\end{lemma}

\begin{proof}
The Cartesian square gives the identification
$$Y\cong X\times_{f,X,\tau}Y,\qquad z\mapsto \bigl(\tau(z),\widetilde f(z)\bigr).$$
The first projection is $\tau$ and the second is $\widetilde f$. Since $\widetilde f$ is obtained from $f$ by base change and its source is connected, it is finite and Galois, with $\mathrm{Gal}(\widetilde f)\cong \mathrm{Gal}(f)$ acting on the first coordinate.

For each $\delta\in\mathrm{Gal}(\tau)$, the map
$$(x,y)\mapsto(x,\delta(y))$$
preserves the first projection. These maps are distinct because the second projection is surjective. Since their number is $\deg\tau$, they give all automorphisms of the first projection $\tau$. Hence each $\gamma\in\mathrm{Gal}(\tau)$ acting on the source $Y$ has a unique expression
$$(x,y)\mapsto(x,\alpha(\gamma)(y)),\qquad \alpha(\gamma)\in\mathrm{Gal}(\tau).$$
This correspondence is bijective and respects composition, so $\alpha$ is a group automorphism. Now we consider the second projection $\widetilde{f}\circ\gamma(z)$ of $\gamma(z)$ for $z\in Y$, we have $\widetilde{f}\circ\gamma(z)= \alpha(\gamma)\circ\widetilde f(z)$, and the surjectivity of $\widetilde f$ makes $\alpha$ uniquely determined by this relation.

The two coordinate actions commute, proving the last assertion.
\end{proof}

\medskip

\subsection{Isogenies}

We next record the facts about isogenies used below.

\begin{lemma}\label{lem: int-amplified-abelian-fixed-point}
Let $g:A\to A$ be an int-amplified endomorphism of an abelian variety. Then $g$ has a fixed point. With this point as the origin, $g$ is an isogeny.
\end{lemma}

\begin{proof}
We can write $g(x)=u(x)+c$, where $u:A\to A$ is a group homomorphism and $c\in A$. Since $g$ is surjective, $u$ is an isogeny. Choose an ample divisor $H$ such that $g^*H-H$ is ample. Translations act trivially on $N^1(A)_{\mathbb{R}}$, so $u^*H-H$ is also ample.

We claim that $\mathrm{id}_A-u$ is an isogeny. Otherwise, the identity component of $\ker(\mathrm{id}_A-u)$ is a positive-dimensional abelian subvariety $B$ on which $u$ is the identity. The restriction of $u^*H-H$ to $B$ is then numerically trivial, contradicting its ampleness. Thus $\mathrm{id}_A-u$ is surjective. Choose $x\in A$ with $(\mathrm{id}_A-u)(x)=c$. Then $g(x)=x$, and using $x$ as the origin identifies $g$ with the isogeny $u$.
\end{proof}

\begin{lemma}\label{lem: periodic points isogeny}
The periodic points of an isogeny $g:A\to A$ of a complex abelian variety are Zariski dense.
\end{lemma}

\begin{proof}
If $m$ is coprime to $\deg g$, then $g$ permutes the finite subgroup $A[m]$, so every point of $A[m]$ is periodic. In an analytic uniformization of $A$, the union of the subgroups $A[m]$ over all such $m$ is dense in the analytic topology and hence Zariski dense.
\end{proof}

\begin{lemma}\label{lem: no backward invariant set}
Let $g:A\to A$ be an int-amplified endomorphism of an abelian variety. If a proper closed subset $Z\subsetneq A$ satisfies $g^{-1}(Z)\subseteq Z$, then $Z$ is empty.
\end{lemma}

\begin{proof}
The descending chain
$$Z\supseteq g^{-1}(Z)\supseteq g^{-2}(Z)\supseteq\cdots$$
stabilizes. Thus $W=g^{-m}(Z)$ satisfies $g^{-1}(W)=W$ for some $m\geq 0$. If $Z$ were nonempty, the surjectivity of $g$ would imply that $W$ is also nonempty and proper, contrary to \cite[Lemma~4.1]{Meng20}.
\end{proof}

In the proof of Theorem~\ref{thm: main}, we will make base changes by multiplication isogenies to construct a toric boundary on a finite cover, and then descend this boundary to $X$. The following lemma ensures that composing a cover by an abelian variety with a multiplication isogeny preserves the Galois property and the compatibility with the endomorphism needed for this descent.

\begin{lemma}\label{lem: multiplication-galois-cover}
Let a finite group $H$ act freely on an abelian variety $A$, and let $\pi_0:A\to C:=A/H$ be the quotient. Then for every positive integer $m$, the composite $p= \pi_0\circ[m]:A\to C$ is a finite \'{e}tale Galois cover.

Suppose moreover that an isogeny $g:A\to A$ and an automorphism $\alpha$ of $H$ satisfy $g\circ\gamma= \alpha(\gamma)\circ g$ for every $\gamma\in H$. Then there is a group homomorphism $\beta:\mathrm{Gal}(p)\to\mathrm{Gal}(p)$ such that
$$g\circ\delta = \beta(\delta)\circ g$$
for every $\delta\in\mathrm{Gal}(p)$.
\end{lemma}

\begin{proof}
The morphism $p$ is finite and \'{e}tale, and its source is connected. To describe its automorphisms, write $\gamma(x)= \ell_{\gamma}(x)+c_\gamma$ for $\gamma\in H$, where $\ell_\gamma$ is a group automorphism and $c_\gamma= \gamma(0)$. Choose $b\in A$ with $[m]b=c_\gamma$ and put $\delta(x)= \ell_{\gamma}(x)+b$. Then $[m]\circ\delta= \gamma\circ[m]$, so $p\circ\delta=p$. Different pairs $(\gamma,b)$ give distinct automorphisms: the equality $[m]\circ\delta= \gamma\circ[m]$ determines $\gamma$, and $b= \delta(0)$. There are $|H|\cdot |A[m]|= \deg p$ such automorphisms. Thus $p$ is Galois, and we identify $\mathrm{Gal}(p)$ with these pairs.

The relation $g\circ\gamma= \alpha(\gamma)\circ g$ gives
$$g\circ\ell_\gamma= \ell_{\alpha(\gamma)}\circ g,\qquad g(c_\gamma)=c_{\alpha(\gamma)}.$$
Since $[m]g(b)=g(c_{\gamma})=c_{\alpha(\gamma)}$, the formula
$$\beta(\gamma,b)=(\alpha(\gamma),g(b))$$
defines a map from $\mathrm{Gal}(p)$ to itself. The first equality above gives $g\circ\delta= \beta(\delta)\circ g$. For $\delta_1,\delta_2\in\mathrm{Gal}(p)$, it follows that
$$\beta(\delta_1\delta_2)\circ g=g\circ\delta_1\circ\delta_2= \beta(\delta_1)\circ \beta(\delta_2)\circ g.$$
The surjectivity of $g$ now shows that $\beta$ is a homomorphism.
\end{proof}

\medskip

\subsection{Albanese morphisms and relative Galois structure}

We first strengthen the known structure of Albanese morphisms in the normal and smooth int-amplified settings.

\begin{proposition}\label{prop: smooth albanese}
    Let $X$ be a normal projective variety that admits an int-amplified endomorphism. Let $\mathrm{alb}_X:X\to A$ be its Albanese. Then $\mathrm{alb}_X$ is surjective and flat, and every geometric fiber is normal and integral. Moreover, if $X$ is smooth, then $\mathrm{alb}_X$ is a smooth morphism.
\end{proposition}

To prove this proposition,
we need some preparations from commutative algebra.

Let $\pi:X\to Y$ be a finite surjective morphism between schemes. We say $\pi$ \emph{splits} if there exists a $\mathcal O_Y$-linear morphism 
$$r: \pi_*\mathcal O_X\to \mathcal O_Y$$
such that $r\circ \pi^* = id_{\mathcal O_Y}$. We call such a morphism a \emph{retraction}.

\begin{lemma}\label{normal-split}
    Let $\pi:X\to Y$ be a finite morphism from a variety $X$ to a normal variety $Y$. Then $\pi$ splits.
\end{lemma}

\begin{proof}
    Since $Y$ is normal, the trace map
    $$
    \operatorname{tr}_{X/Y}: \pi_*\mathcal O_X\to \mathcal O_Y$$
is well-defined on regular functions. Then $r := \deg(\pi)^{-1}\operatorname{tr}_{X/Y}$ is a retraction as desired.
\end{proof}

The key point is that split morphisms behave well under (not necessarily flat) base changes.

\begin{lemma}\label{splitting}
    Let $\pi:X\to Y$ be a  surjective finite morphism between schemes of finite type over $\mathbb C$. Assume $\pi$ splits and let $r$ be a retraction. Then
    \begin{enumerate}
        \item Every base change of $\pi$ also splits;
        \item If $X$ is a normal variety, so is $Y$;
        \item Let $f: Y\to Z$ be a morphism between schemes of finite type over $\mathbb C$. If $\pi\circ f$ is flat, so is $f$.
    \end{enumerate}
\end{lemma}

\begin{proof}
    We may assume that $X,Y$ are affine schemes and $\pi$ is induced by a ring map $A\subseteq B$. 
    
    (1) If $r:B\to A$ is a retraction, then $r\otimes id_{A'}:B\otimes A'\to A\otimes A'$ is also a retraction.

    (2) Assume $B$ is integrally closed. For $a = x/y\in \operatorname{Frac}(A)$ which is integral over $A$, we have $x\in B$ because $B$ is integrally closed. Then we have
    \[
         x = r(x) = r(ya) = yr(a) 
    \]
    which implies that $r(a) = x/y = a$ and hence $a\in A$.

    (3) We assume $Z = \operatorname{Spec}(R)$. Since $A\subseteq B$ splits, we have $B = A\oplus C$ for some $C$ as $A$-modules, and hence as $R$-modules. Consequently, $R\to A$ is flat because $A$ is a direct summand of a flat algebra $R\to B$.
\end{proof}

Now, we are ready to prove the proposition.

\begin{proof}[Proof of Proposition \ref{prop: smooth albanese}]
    Let $f:X\to X$ be an int-amplified endomorphism and $g:A\to A$ be the endomorphism induced by $f$. By \cite[Theorem~1.8(1)]{Meng20}, the morphism $\mathrm{alb}_X$ is surjective and equidimensional, all its fibers are irreducible, and $g$ is int-amplified. A surjective endomorphism of an abelian variety is a translation followed by an isogeny, so $g$ is \'{e}tale. We could assume $\dim A>0$.

    Write $a = \mathrm{Alb}_X$. Then, there exists a commuative diagram
    \[
        \begin{tikzcd}
            X\arrow{r}{\pi}\arrow{dr}{a} & X_g\arrow{r}{g_X}\arrow{d}{a_g} & X\arrow{d}{a}\\
              & A\arrow{r}{g} & A
        \end{tikzcd}
    \]
    where the square is Cartesian and $\pi=(f,a)$ is induced by the universal property.

    We first show that $\pi$ splits.  In fact,  since $g$ is \'etale, $g_X$ is \'etale, and hence $X_g$ is normal. Since $a$ is a proper contraction and $g$ is flat, by \cite[III.9.3]{Har77}, we calculate that
   $$(a_g)_*\mathcal O_{X_g} \cong g^*a_*\mathcal O_X \cong \mathcal O_A,$$
    which implies that $a_g$ is a proper contraction as well. In particular,  $X_g$ is a normal variety.

    We also have $\pi$ is finite and surjective. In fact, since (the restriction of) $\pi$ decomposes as closed immersion followed by a finite morphism, i.e., $(f,a) = (f,id)\circ (id,a)$, it is finite. Since $\dim(X) = \dim(X_g)$, it follows that $\pi$ is surjective. As a result, by Lemma \ref{normal-split}, we have $\pi$ is split.

    Let $U\subseteq A$ be the largest open subset such that $X_U\to U$ is flat and has geometric integral normal fibers. Then by generic flatness and openness of normal fibers, we have $U\neq \emptyset$. We write $\Delta = A \setminus U$.

    Since $X_U\to U$ is flat, by applying Lemma \ref{splitting} to $a_g\circ\pi = a$, we conclude that $(X_g)_U\to U$ is flat. For a geometric point $b\to U$, since $\pi:X_g\to X$ is split, by Lemma \ref{splitting}, we have $X_{\bar b}\to (X_g)_{\bar b}$ is split. Since $X_{\bar b}$ is normal and integral, by Lemma \ref{splitting}, it follows that $(X_g)_{\bar b}$ is also normal and integral. We conclude that $X_g\to A$ is flat and has integral and geometric normal fibers over $U$. Since $g:A\to A$ is \'etale, we have $g(U)\subseteq A$ is open and $X_{g(U)}\to g(U)$ is flat and has integral and geometric normal fibers over $U$. By our construction of $U$, this implies $g^{-1}(\Delta)\subseteq \Delta$, and consequently, by  Lemma~\ref{lem: no backward invariant set}, $\Delta = \emptyset$.

    Next, we assume $X$ is smooth. The same strategy works with minor changes. Let $U$ be the largest open subset such that $X_U\to U$ is smooth and write $\Delta =  A\setminus U$.  By generic smoothness and properness, $\Delta$ is a proper closed subset. We claim that $g^{-1}(\Delta)\subseteq\Delta$. For $b\in U$, the finite morphism $f$ maps $X_b$ onto $X_{g(b)}$: its image is closed and, by the equidimensionality of $\mathrm{alb}_X$, has the same dimension as the irreducible fiber $X_{g(b)}$. Given $y\in X_{g(b)}$, choose $x\in X_b$ with $f(x)=y$. Differentiating $\mathrm{alb}_X\circ f = g\circ \mathrm{alb}_X$ gives
$$d(\mathrm{alb}_X)_y\circ d f_x=d g_b\circ d(\mathrm{alb}_X)_x.$$
The right-hand side is surjective because $\mathrm{alb}_X$ is smooth along $X_b$ and $g$ is \'{e}tale. Hence $d(\mathrm{alb}_X)_y$ is surjective. Since both $X$ and $A$ are smooth, the tangent-space criterion for smoothness shows that $\mathrm{alb}_X$ is smooth at $y$ \cite[Proposition~III.10.4]{Har77}. Thus $g(b)\notin\Delta$, proving the claim.  Lemma~\ref{lem: no backward invariant set} now gives $\Delta= \emptyset$.
\end{proof}

\begin{remark}\label{rem: klt case}
    By the same argument and \cite[Theorem~1.1]{Zhuang24}, one can prove the following:

    Let $X$ be a klt type variety that admits an int-amplified endomorphism. Then every geometric fiber of $\mathrm{alb}_X$ is of klt type.
\end{remark}

\medskip

We next record the relative Galois structure when $q(X)= \widetilde q(X)$.

\begin{proposition}\label{prop: albanese-structure}
Let $X$ be a smooth complex projective variety admitting an int-amplified Galois endomorphism $f:X\to X$, and suppose that $q(X)= \widetilde q(X)$. Let $\mathrm{alb}_X:X\to A:= \mathrm{Alb}(X)$ be an Albanese morphism. After translating $\mathrm{alb}_X$ if necessary, the following hold.
\begin{enumerate}
\item The endomorphism induced by $f$ is an int-amplified isogeny $g:A\to A$ satisfying $\mathrm{alb}_X\circ f=g\circ \mathrm{alb}_X$. The morphism $\mathrm{alb}_X$ is smooth with irreducible rationally connected fibers.

\item Set
$$K:= \{\sigma\in\mathrm{Gal}(f):\mathrm{alb}_X\circ\sigma= \mathrm{alb}_X\}.$$
The induced action on $A$ identifies $\mathrm{Gal}(f)/K$ with the group of translations by points of $\ker(g)$.

\item The morphism $(f,\mathrm{alb}_X)$ induces an isomorphism over $A$
\begin{equation}\label{eq: relative quotient}
X/K\cong X\times_{A,g}A.
\end{equation}
\end{enumerate}
\end{proposition}

\begin{proof}
By \cite[Theorem~1.8(1)]{Meng20}, the endomorphism $f$ induces an int-amplified endomorphism of $A$. By Lemma~\ref{lem: int-amplified-abelian-fixed-point}, this endomorphism has a fixed point. Translating $\mathrm{alb}_X$ and taking this point as the origin of $A$, we obtain an int-amplified isogeny $g:A\to A$ satisfying $\mathrm{alb}_X\circ f=g\circ \mathrm{alb}_X$.

Lemma~\ref{lem: maximal irregularity mrc} identifies $\mathrm{alb}_X$ with the special MRC fibration and shows that its fibers are connected, irreducible, and rationally connected. Proposition~\ref{prop: smooth albanese} shows that $\mathrm{alb}_X$ is smooth.

For $\sigma\in\mathrm{Gal}(f)$, we have $f\circ\sigma=f$. Since $\mathrm{alb}_X\circ f=g\circ \mathrm{alb}_X$, it follows that $g\circ \mathrm{alb}_X\circ\sigma=g\circ \mathrm{alb}_X$. Hence the morphism $x\mapsto \mathrm{alb}_X(\sigma(x))-\mathrm{alb}_X(x)$ takes values in the finite group $\ker(g)$ and is constant because $X$ is connected. Thus $\mathrm{Gal}(f)$ acts on $A$ through translations by points of $\ker(g)$, and the kernel of this action is $K$. In particular, $K$ is normal.

For a general point $b\in A$, the irreducible components of $f^{-1}(X_b)$ are the fibers $X_c$ with $c\in g^{-1}(b)$. Since $f$ is the quotient by $\mathrm{Gal}(f)$, the group $\mathrm{Gal}(f)$ acts transitively on these components. Consequently, the induced translation action is transitive on $g^{-1}(b)$. Since $g^{-1}(b)$ is a torsor under $\ker(g)$, every translation by a point of $\ker(g)$ occurs. This identifies $\mathrm{Gal}(f)/K$ with $\ker(g)$.

Put $Y=X/K$. Since $K$ preserves every Albanese fiber, the morphism $\mathrm{alb}_X$ descends to $Y$. The map $f$ factors through a finite morphism $\overline f:Y\to X$, with $\mathrm{Gal}(\overline f)= \mathrm{Gal}(f)/K\cong \ker(g)$. Every nonidentity element of this group induces a nonzero translation of $A$ and hence has no fixed point on $Y$. Thus the action of $\mathrm{Gal}(f)/K$ on $Y$ is free, and $\overline f$ is a finite \'{e}tale $\mathrm{Gal}(f)/K$-torsor. In particular, $Y$ is smooth.

Consider the induced morphism
$$Y\to X\times_{A,g}A.$$
Both the source and the target are finite \'{e}tale covers of $X$ of degree $|\ker(g)|$. If two closed points $[x_1], [x_2]\in Y$ have the same image, then $f(x_1)=f(x_2)$. Since $f$ is the quotient by $\mathrm{Gal}(f)$, there is some $\sigma\in\mathrm{Gal}(f)$ such that $x_2= \sigma(x_1)$. The translation induced by $\sigma$ fixes their common image in $A$, so it is the identity. Hence $\sigma\in K$ and $[x_1]= [x_2]$. The induced map is therefore injective on every fiber. Since the two finite \'{e}tale covers have the same degree, it is an isomorphism. This proves \eqref{eq: relative quotient}.
\end{proof}

\medskip

\subsection{Toric fibrations, Cox rings, and reflection groups}

\begin{definition}\label{def: toric fibration}
Let $B$ be a connected smooth complex projective variety. A \emph{torus over $B$} is a group scheme $\mathcal T\to B$ that is \'{e}tale locally isomorphic to $\mathbf{G}_{m,B}^r$ for some $r\geq 0$. A smooth projective morphism $h:Y\to B$ is a \emph{toric fibration} if $Y$ admits an action of a torus $\mathcal T$ over $B$ such that, \'{e}tale locally on $B$, there are equivariant isomorphisms
$$\mathcal T|_U\cong\mathbf{G}_{m,U}^r,\qquad Y\times_B U\cong X(\Sigma)\times U$$
for a smooth complete fan $\Sigma$. The complement of the relative dense orbit is the \emph{relative toric boundary}.

The toric fibration is \emph{split} if $\mathcal T\cong T\times B$ for an algebraic torus $T$. In this case, there are a smooth projective toric variety $F$ with dense torus $T$, a right principal $T$-bundle $P\to B$, and an isomorphism $Y\cong P\times^T F$ over $B$ \cite[Lemma~2.1.2]{AWZ21}. Here $P\times^T F=(P\times F)/T$, where $t\in T$ acts by $t\cdot(p,z)=(pt^{-1},tz)$.
\end{definition}

Definition~\ref{def: toric fibration} is the smooth projective case of \cite[Definition~2.1.1]{AWZ21}. If $D$ denotes the relative toric boundary, then $(Y,D)\to B$ is a toric fibration in the sense of Meng and Zhong \cite[Definition~2.8]{MZ24}. In the smooth projective setting, the converse follows from \cite[Proposition~4.3.1]{AWZ21}.

If $h$ is split, its boundary components are $P\times^T E$, where $E$ runs through the torus-invariant prime divisors of $F$. Each component restricts to an irreducible divisor on every fiber. Thus $(Y,D)\to B$ is a splitting toric fibration in the sense of Meng and Zhong.

\medskip

We next present the Cox-ring and reflection-group facts needed to recognize toric fibers, starting with varieties of Fano type.

\begin{definition}\label{def: fano-type}
A normal projective variety $X$ is \emph{of Fano type} if there is an effective $\mathbb{Q}$-divisor $\Delta$ such that $(X,\Delta)$ is klt and $-(K_X+\Delta)$ is ample.
\end{definition}

Every projective toric variety is of Fano type \cite[Example~11.4.26]{CLS11}. Every smooth projective variety of Fano type is rationally connected and hence has free Picard group by Lemma~\ref{lem: rc-basic}, and its Cox ring is finitely generated by \cite[Corollary~1.3.2]{BCHM10}.

\medskip

We now recall the Cox construction and fix the conventions used below. Let $X$ be a smooth connected projective variety whose Picard group $\Lambda:= \mathrm{Pic}(X)$ is free of finite rank $\rho$. Choose a basis $\lambda_1,\ldots,\lambda_\rho$ of $\Lambda$ and line bundles $\mathcal L_1,\ldots, \mathcal L_\rho$ representing these classes. For $\lambda= \sum_{i} m_i\lambda_i$, set $\mathcal L_\lambda= \bigotimes_i\mathcal L_i^{\otimes m_i}$. These choices make $\bigoplus_{\lambda\in\Lambda}\mathcal L_\lambda$ a $\Lambda$-graded $\mathcal O_X$-algebra. Its algebra of global sections is the Cox ring
$$\mathrm{Cox}(X):= \bigoplus_{\lambda\in\Lambda}H^0(X,\mathcal L_\lambda).$$
A nonzero section in the summand indexed by $\lambda$ is said to have Cox degree $\lambda$.

The characteristic torus is $S_X:= \mathrm{Hom}(\Lambda,\mathbf{G}_m)$, whose character lattice is $\Lambda$. The characteristic space is
\begin{equation}\label{eq: characteristic space}
U_X:= \mathrm{Spec}_X\left(\bigoplus_{\lambda\in\Lambda}\mathcal L_\lambda\right).
\end{equation}
On an open subset $W\subseteq X$ where $\mathcal L_1,\ldots,\mathcal L_\rho$ are trivial, the defining algebra is a Laurent polynomial algebra over $\mathcal O_W$ in $\rho$ variables. Hence $U_X|_W\cong W\times S_X$.

We use the right $S_X$-action characterized by
$$u(xs)= \lambda(s)^{-1}u(x),$$
where $u\in H^0(W,\mathcal L_\lambda)$ is viewed as a regular function on $U_X|_W$, $x\in U_X|_W$, and $s\in S_X$. With this action, $U_X\to X$ is a right $S_X$-torsor. For $\lambda\in\Lambda$, let $\mathbf A^1_\lambda$ denote the one-dimensional representation on which $s\in S_X$ acts by multiplication by $\lambda(s)$. With the associated-bundle convention of Definition~\ref{def: toric fibration}, the transformation rule above gives
$$U_X\times^{S_X}\mathbf A^1_\lambda\cong\mathcal L_\lambda.$$
The Cox ring is independent of the chosen basis and representatives up to noncanonical graded isomorphism, and the $S_X$-torsor $U_X\to X$ is independent up to noncanonical $S_X$-equivariant isomorphism.

If $\mathrm{Cox}(X)$ is finitely generated, let $\mathfrak m_X$ be the homogeneous maximal ideal generated by the pieces $\mathrm{Cox}(X)_\lambda$ with $\lambda\neq 0$. The finite-dimensional graded quotient $\mathfrak m_X/\mathfrak m_X^2$ records homogeneous elements modulo sums of products of elements of nonzero Cox degree. We call it the indecomposable quotient of the Cox ring. By the graded Nakayama lemma, any homogeneous basis of $\mathfrak m_X/\mathfrak m_X^2$ can be lifted to a minimal set of homogeneous algebra generators of $\mathrm{Cox}(X)$ \cite[Proposition~2.10.1 and the subsequent discussion]{CW11}.

\begin{lemma}\label{lem: cox-characteristic-space}
Let $X$ be a smooth projective variety of Fano type. Then $\mathrm{Cox}(X)$ is a factorial domain whose only units are the nonzero constants. The standard Cox construction gives an $S_X$-equivariant open immersion $U_X\hookrightarrow\mathrm{Spec}\mathrm{Cox}(X)$ whose complement has codimension at least two. Moreover, $\Gamma(U_X,\mathcal O_{U_X})= \mathrm{Cox}(X)$ and $U_X^{\mathrm{an}}$ is simply connected.
\end{lemma}

\begin{proof}
Factoriality and the assertion about units follow from \cite[Theorem~1.1 and Corollary~2.2]{Arzh09}. The open immersion, its complement, and the equality of global functions are given by \cite[Constructions~I.6.1.3 and I.6.3.1, Proposition~I.6.1.6, and Corollary~I.6.2.4]{ADHL15}.

It remains to prove simple connectivity. Since $X$ is of Fano type, it is rationally connected, and Lemma~\ref{lem: rc-basic} shows that $X^{\mathrm{an}}$ is simply connected. The base and the torus fiber are connected, so $U_X^{\mathrm{an}}$ is connected. The homotopy sequence of the principal torus bundle contains
$$\pi_2(X^{\mathrm{an}})\xrightarrow{\partial}\pi_1(S_X^{\mathrm{an}})\to\pi_1(U_X^{\mathrm{an}})\to\pi_1(X^{\mathrm{an}})=1.$$
The Hurewicz theorem identifies $\pi_2(X^{\mathrm{an}})$ with $H_2(X,\mathbb{Z})$. There is also a natural identification
$$\pi_1(S_X^{\mathrm{an}})\cong\mathrm{Hom}(\Lambda,\mathbb{Z}).$$
Under these identifications, the boundary map is the evaluation map
$$\partial(\beta)(\lambda)= \langle c_1(\mathcal L_\lambda),\beta\rangle.$$
By Lemma~\ref{lem: rc-basic} and the universal coefficient theorem, evaluation gives a perfect pairing between $\Lambda= \mathrm{Pic}(X)$ and $H_2(X,\mathbb{Z})/H_2(X,\mathbb{Z})_{\mathrm{tors}}$. Hence $\partial$ is surjective. Exactness and $\pi_1(X^{\mathrm{an}})=1$ now give
$$\pi_1(U_X^{\mathrm{an}})\cong\pi_1(S_X^{\mathrm{an}})/\mathrm{im}(\partial)=1.$$
\end{proof}

\medskip

We next recall the Chevalley--Shephard--Todd theorem and the basic facts about finite complex reflection groups.

\begin{definition}[{\cite[p.~274]{ST54}}]\label{def: pseudoreflection}
Let $V$ be a finite-dimensional complex vector space. A nontrivial element $\gamma\in\mathrm{GL}(V)$ of finite order is a \emph{pseudoreflection} if $\mathrm{rank}(\gamma-\mathrm{id}) = 1$, or equivalently, if $\gamma$ fixes a hyperplane pointwise. The one-dimensional space $\mathrm{im}(\gamma-\mathrm{id})$ is its \emph{normal line}. A finite subgroup of $\mathrm{GL}(V)$ generated by pseudoreflections is a \emph{complex reflection group}.
\end{definition}

For a representation $V$ of a finite group $\Gamma, \mathrm{Sym}(V)$ denotes the symmetric algebra generated by $V$, with its induced group action. Equivalently, it is the coordinate ring of the dual representation $V^*$. We use the natural grading $\mathrm{Sym}(V)= \bigoplus_{d\geq 0} \mathrm{Sym}^d(V)$. When $\mathrm{Sym}(V)^\Gamma$ is a polynomial ring, its homogeneous algebraically independent generators are called basic invariants, and their degrees in this grading are called the basic degrees.

\begin{theorem}[Chevalley--Shephard--Todd]\label{thm: reflection-invariants}
Let $\Gamma\subseteq\mathrm{GL}(V)$ be finite and let $e= \dim V$. The invariant ring $\mathrm{Sym}(V)^\Gamma$ is a polynomial ring if and only if $\Gamma$ is generated by pseudoreflections. When these conditions hold, there are homogeneous basic invariants of degrees $d_1,\ldots,d_e$, and
$$\prod_{j=1}^e d_j=|\Gamma|.$$

The local form is as follows. If a finite group $\Gamma$ acts faithfully on a smooth complex variety $W$, then $W/\Gamma$ is smooth at the image of $x\in W$ if and only if the stabilizer $\Gamma_x$ is generated by pseudoreflections on $T_xW$.
\end{theorem}

\begin{proof}
The polynomial characterization is due to Chevalley and Shephard--Todd \cite{Chev55,ST54}. The formula for the product of the basic degrees follows from \cite[Theorem~4.14(i)]{LT09}.

We prove the local assertion. Put $G= \Gamma_x$, let $\overline x$ be the image of $x$, and set $R= \mathrm{Sym}(T_x^*W)$. Completion of the finite quotient at the orbit of $x$ gives
$$\widehat{\mathcal O}_{W/\Gamma,\overline x}\cong\widehat{\mathcal O}_{W,x}^{G}.$$
Since $W$ is smooth, averaging a linear section of $\mathfrak m_x\to\mathfrak m_x/\mathfrak m_x^2=T_x^*W$ over $G$ gives a $G$-equivariant regular system of parameters. Hence
$$\widehat{\mathcal O}_{W,x}\cong\widehat R$$
as complete $G$-algebras, where $R$ is completed at its ideal of positive-degree elements. Since $R$ is finite over $R^G$ and $G$ is linearly reductive, taking invariants commutes with these completions. Therefore
$$\widehat{\mathcal O}_{W/\Gamma,\overline x}\cong\widehat{R^G}.$$

The quotient $W/\Gamma$ is smooth at $\overline x$ if and only if $\widehat{R^G}$ is regular. Since $R^G$ is a connected positively graded domain, this is equivalent to $R^G$ being a polynomial ring. Indeed, if the local ring of $R^G$ at its homogeneous maximal ideal is regular, a minimal set of homogeneous algebra generators has cardinality $\dim R^G$. The resulting surjection from a polynomial ring of that dimension has a height-zero prime kernel and is therefore an isomorphism. The polynomial characterization now shows that $R^G$ is polynomial if and only if $G$ is generated by pseudoreflections on $T_x^*W$. Since an element is a pseudoreflection on $T_x^*W$ if and only if it is one on $T_xW$, the local assertion follows.
\end{proof}

\begin{lemma}[{cf. \cite[Theorem~1.27 and its proof]{LT09}}]\label{lem: reflection-blocks}
Let a finite complex reflection group $\Gamma$ act faithfully on $V$, and suppose that it preserves every summand of a decomposition $V= \bigoplus_{i=1}^r V_i$. For each $i\in\{1,\ldots,r\}$, let $\Gamma_i$ be the subgroup of $\Gamma$ generated by the pseudoreflections $\gamma\in\Gamma$ whose normal lines lie in $V_i$. Then multiplication induces an isomorphism
$$\prod_{i=1}^r\Gamma_i\cong\Gamma,$$
and the invariant ring decomposes as
$$\mathrm{Sym}(V)^\Gamma= \bigotimes_{i=1}^r\mathrm{Sym}(V_i)^{\Gamma_i}.$$
\end{lemma}

\begin{proof}
For every pseudoreflection $\gamma\in\Gamma$, preservation of the decomposition gives $\mathrm{rank}(\gamma-\mathrm{id})= \sum_{i=1}^r\mathrm{rank}((\gamma-\mathrm{id})|_{V_i})=1$. Thus $\gamma$ acts nontrivially on exactly one summand and fixes all the others, and its normal line lies in that summand. Pseudoreflections supported on different summands commute. Since the pseudoreflections generate $\Gamma$, the multiplication map $\prod_{i=1}^r\Gamma_i\to\Gamma$ is surjective.

For each $i\in  \{1,\ldots,r\}$, let $\gamma_i\in\Gamma_i$, and suppose that $\prod_{i=1}^r\gamma_i$ acts trivially on $V$. Restricting to $V_i$ shows that $\gamma_i$ acts trivially on $V_i$, while by construction it fixes every other summand. Faithfulness of the action on $V$ gives $\gamma_i = 1$ for every $i$. Thus the multiplication map is injective. The product decomposition of $\Gamma$ gives the stated tensor-product decomposition of invariant rings.
\end{proof}

Recall that a finite complex reflection group on $V$ has a canonical decomposition \cite[Theorem~1.27 and its proof]{LT09}
$$V=V^\Gamma\oplus W_1\oplus\cdots\oplus W_s,$$
where the subgroup generated by the pseudoreflections with normal lines in $W_i$ acts irreducibly on $W_i$ and fixes the other summands.

To descend a toric boundary through a finite quotient, we will need the group to permute the coordinate hyperplanes. The following lemma gives the required linear-algebra statement.

\begin{lemma}\label{lem: reflection-centralizer}
Let $\Gamma\subseteq\mathrm{GL}(V)$ be a finite complex reflection group with $V^\Gamma=0$. Let $J\subseteq\mathrm{GL}(V)$ be a subgroup whose image centralizes the image of $\Gamma$ in $\mathrm{PGL}(V)$. Then $V$ is a direct sum of lines permuted by $J$. In particular, choosing a nonzero vector in each of these lines gives polynomial coordinates on $\mathrm{Spec}\mathrm{Sym}(V)$ whose coordinate hyperplanes are permuted by $J$.
\end{lemma}

\begin{proof}
The assertion is immediate if $\dim V\leq 1$. For $j\in J$ and a pseudoreflection $\gamma\in\Gamma$, their images commute in $\mathrm{PGL}(V)$. Hence $j\gamma j^{-1}=c \gamma$ for some $c\in\mathbb C^*$.

Suppose that $\dim V=2$, and fix a pseudoreflection $\gamma\in\Gamma$. Since $\gamma$ has finite order, it is diagonalizable, with eigenvalues $1$ and $\zeta\neq 1$. Let $L_1,L_2$ be the corresponding eigenlines, so $V=L_1\oplus L_2$. For $j\in J$, write $j\gamma j^{-1}=c\gamma$ as above. The operator $c\gamma$ acts on $L_1$ and $L_2$ by the distinct scalars $c$ and $c\zeta$, so its two eigenlines are still $L_1$ and $L_2$. On the other hand, if $\gamma v= \lambda v$, then
$$(j\gamma j^{-1})(jv)=j(\gamma v)= \lambda jv.$$
Thus the two eigenlines of $j\gamma j^{-1}$ are $j(L_1)$ and $j(L_2)$. Since $j\gamma j^{-1}=c\gamma$, we obtain
$$\{j(L_1),j(L_2)\}= \{L_1,L_2\}.$$
Therefore every element of $J$ either preserves both lines or exchanges them.

Suppose that $\dim V\geq 3$. The unique eigenvalue of multiplicity $\dim V-1$ is $1$ for $j\gamma j^{-1}$ and $c$ for $c\gamma$, so $c=1$. Thus $J$ commutes with every pseudoreflection and preserves each normal line. Choose a $\Gamma$-invariant positive definite Hermitian form on $V$. The orthogonal complement of the span of these normal lines is the common fixed space of all pseudoreflections, which is $V^\Gamma=0$. Hence the normal lines span $V$. Choosing a basis of $V$ consisting of vectors on these lines gives a direct sum decomposition into lines, each preserved by $J$.

The corresponding linear generators of $\mathrm{Sym}(V)$ are permuted up to nonzero scalar multiples. The hyperplanes defined by these linear generators are therefore permuted by $J$, proving the last assertion.
\end{proof}

We pass from these invariant-theoretic statements to toric geometry through the following Cox-ring characterization, due to Hu and Keel.

\begin{proposition}[{\cite[Corollary~2.10]{HK00}}]\label{prop: polynomial-cox-toric}
Let $X$ be a smooth projective variety such that $\mathrm{Pic}(X)$ is free of finite rank and $\mathrm{Pic}(X)\otimes_{\mathbb{Z}}\mathbb{R}=N^1(X)_{\mathbb{R}}$. Then $X$ is toric if and only if $\mathrm{Cox}(X)$ is a polynomial ring.
\end{proposition}

\section{Toric fibers and isotriviality}
\label{sec: toric-fibers}

\subsection{Galois cycles and toric fibers}

In this subsection, we prove that if a cycle of Galois covers between smooth projective rationally connected varieties has an int-amplified composite, then every variety in the cycle has a polynomial Cox ring and is therefore toric. Our approach is motivated by the work of Arzhantsev and Gaifullin on Cox realizations of finite quotients \cite[Theorem~3.1]{AG10}. It also draws inspiration from Luo and Zhang's use of a linear lift and reflection-group invariant theory to study Galois self-covers of projective space \cite{LZ26}.

\begin{lemma}\label{lem: fiber-cycle}
Let $V_0,\ldots,V_{s-1}$ be smooth projective rationally connected varieties and set $V_s=V_0$. For $0\leq i < s$, let $\varphi_i:V_{i+1}\to V_i$ be a finite Galois morphism. These maps form a cycle
$$V_0\xleftarrow{\varphi_0}V_1\leftarrow\cdots\xleftarrow{\varphi_{s-1}}V_s=V_0.$$ Put $\Psi_0:= \varphi_0\circ\cdots\circ\varphi_{s-1}:V_0\to V_0$. If $\Psi_0$ is int-amplified, then every $V_i$ has a polynomial Cox ring and is toric.
\end{lemma}

\begin{proof}
Put $r= \dim V_0$. We may assume that $r>0$.

We divide the proof into five steps.

\noindent\textbf{Step 1.} We transport int-amplification around the cycle and record the Picard and Cox-ring properties that will be used below.

Since every $\varphi_i$ is finite, $\dim V_i =r$ for every $i$. For $0<i<s$, define the cyclic return map $\Psi_i:V_i\to V_i$ by
$$\Psi_i= \varphi_i\circ\cdots\circ\varphi_{s-1}\circ\varphi_0\circ\cdots\circ\varphi_{i-1}.$$
The composite $\varphi_0\circ\cdots\circ\varphi_{i-1}$ satisfies the following:
$$(\varphi_0\circ\cdots\circ\varphi_{i-1})\circ\Psi_i= \Psi_0\circ(\varphi_0\circ\cdots\circ\varphi_{i-1}).$$
Choose an ample divisor $H_0$ such that $\Psi_0^*H_0-H_0$ is ample, and, for $i>0$, set
$$H_i=(\varphi_0\circ\cdots\circ\varphi_{i-1})^*H_0.$$
Then $H_i$ and $$\Psi_i^*H_i-H_i=(\varphi_0\circ\cdots\circ\varphi_{i-1})^*(\Psi_0^*H_0-H_0)$$ are ample. Thus every $\Psi_i$ is int-amplified.

Pullback by a finite morphism is injective on the Picard group modulo torsion. Applying this around the cycle shows that all Picard ranks are equal. Denote their common value by $\rho$. Put $\Lambda_i= \mathrm{Pic}(V_i)$ for $0\leq i < s$. Since $V_i$ is rationally connected, Lemma~\ref{lem: rc-basic} shows that $\Lambda_i$ is a free lattice of rank $\rho$ and that the natural map
$$\Lambda_i\xrightarrow{\sim}N^1(V_i)$$
is an isomorphism.
We take the indices of the $V_i$ and their associated data modulo $s$, using exactly the same choices at indices $0$ and $s$. In particular, put
$$c_i=[\Lambda_{i+1}:\varphi_i^*\Lambda_i]$$
for $0\leq i<s$.

Since $\Psi_i$ is int-amplified, $V_i$ is of Fano type by \cite[Corollary~1.4]{Yosh21b}. The preceding discussion and Lemma~\ref{lem: cox-characteristic-space} show that $\mathrm{Cox}(V_i)$ is a finitely generated factorial domain whose only units are the nonzero constants.

\medskip

\noindent\textbf{Step 2.} We lift the morphisms to Cox characteristic spaces and identify the resulting finite quotients.

For each $0\leq i<s$, choose line bundles representing a basis of $\Lambda_i$ and form $\mathcal L_{i,\lambda}$ for $\lambda\in\Lambda_i$ by tensor products, as in the construction preceding \eqref{eq: characteristic space}. Set
$$S_i:= \mathrm{Hom}(\Lambda_i,\mathbf{G}_m),\qquad U_i:= \mathrm{Spec}_{V_i}\left(\bigoplus_{\lambda\in\Lambda_i}\mathcal L_{i,\lambda}\right).$$
With the right action specified after \eqref{eq: characteristic space}, the projection $U_i\to V_i$ is an $S_i$-torsor. Lemma~\ref{lem: cox-characteristic-space} gives
$$H^0(U_i,\mathcal O_{U_i})= \mathrm{Cox}(V_i),$$
and $U_i^{\mathrm{an}}$ is simply connected.

Let $\alpha_i:S_{i+1}\to S_i$ be the homomorphism whose map on character lattices is $\varphi_i^*:\Lambda_i\to\Lambda_{i+1}$. Since $\varphi_i^*$ is injective with cokernel of order $c_i$, the map $\alpha_i$ is an isogeny with $|\ker(\alpha_i)|=c_i$. For every $\lambda\in\Lambda_i$, the line bundles associated with $\lambda$ by the pushout of $U_{i+1}$ along $\alpha_i$ and by the pullback $\varphi_i^*U_i$ both have class $\varphi_i^*\lambda$. The two $S_i$-torsors over $V_{i+1}$ are therefore isomorphic. Choose such an isomorphism. The pushout is $U_{i+1}/\ker(\alpha_i)$, so the composite
$$U_{i+1}\to U_{i+1}/\ker(\alpha_i)\cong\varphi_i^*U_i\to U_i$$
defines an $\alpha_i$-equivariant lift $\Phi_i$ of $\varphi_i$.
In this composition, the first map is the quotient by $\ker(\alpha_i)$ and has degree $c_i$, while the last map is the base change of $\varphi_i$. Hence $\Phi_i$ is finite and surjective, with degree $(\deg\varphi_i) c_i$.

For $\gamma\in\mathrm{Gal}(\varphi_i)$, the equality $\gamma^*\varphi_i^*= \varphi_i^*$ shows that $\gamma^*$ fixes the finite-index sublattice $\varphi_i^*\Lambda_i\subseteq\Lambda_{i+1}$. Since $\Lambda_{i+1}$ is torsion-free, $\gamma^*$ acts trivially on $\Lambda_{i+1}$. It follows from the description of split-torus torsors above that $\gamma$ admits an $S_{i+1}$-equivariant lift to $U_{i+1}$.

Let $\Gamma_i$ be the group of all $S_{i+1}$-equivariant automorphisms $\widetilde\gamma$ of $U_{i+1}$ that lift elements of $\mathrm{Gal}(\varphi_i)$ and satisfy
$$\Phi_i\circ\widetilde\gamma= \Phi_i.$$
Every element of $\mathrm{Gal}(\varphi_i)$ has such a lift. Indeed, start with any $S_{i+1}$-equivariant lift $\widetilde\gamma$. The maps $\Phi_i\circ\widetilde\gamma$ and $\Phi_i$ cover the same morphism to $V_i$ and are both $\alpha_i$-equivariant. Their ratio descends to an $S_i$-valued morphism on $V_{i+1}$, which is constant because $V_{i+1}$ is projective and connected. Thus there is a $t\in S_i$ such that
$$\Phi_i(\widetilde\gamma(u))= \Phi_i(u)t$$
for every $u\in U_{i+1}$. Choose $a\in S_{i+1}$ with $\alpha_i(a)=t^{-1}$. Then $u\mapsto\widetilde\gamma(u)a$ is a lift of $\gamma$ that belongs to $\Gamma_i$.

The kernel of the natural homomorphism $\Gamma_i\to\mathrm{Gal}(\varphi_i)$ consists of right translations by elements of $\ker(\alpha_i)$. Indeed, an $S_{i+1}$-equivariant automorphism $\widetilde\gamma$ over $V_{i+1}$ is right translation by an $S_{i+1}$-valued morphism on $V_{i+1}$. This morphism is constant, and the condition $\Phi_i\circ\widetilde\gamma= \Phi_i$ forces its value to lie in $\ker(\alpha_i)$. Hence there is a central exact sequence
$$1\to\ker(\alpha_i)\to\Gamma_i\to\mathrm{Gal}(\varphi_i)\to1.$$
Centrality follows because every element of $\Gamma_i$ is $S_{i+1}$-equivariant and therefore commutes with right translations. In particular, $\Gamma_i$ is finite and acts faithfully on $U_{i+1}$. Moreover,
\begin{equation}\label{eq: cycle-degree}
|\Gamma_i|=(\deg\varphi_i)c_i= \deg\Phi_i.
\end{equation}

The quasi-affine variety $U_{i+1}$ is quasi-projective, so its quotient by $\Gamma_i$ exists. The morphism $\Phi_i$ induces a finite morphism
$$U_{i+1}/\Gamma_i\to U_i.$$
Indeed, fix an affine open subset $W\subseteq U_i$. Since $\Phi_i$ is finite, $\Phi_i^{-1}(W)$ is affine and its coordinate ring is finite over $\mathcal O_{U_i}(W)$. The quotient $\Phi_i^{-1}(W)/\Gamma_i$ is defined by the invariant subalgebra, which is also finite over $\mathcal O_{U_i}(W)$. Since the action is faithful, \eqref{eq: cycle-degree} gives
$$[\mathbb{C}(U_{i+1}):\mathbb{C}(U_{i+1})^{\Gamma_i}]=|\Gamma_i|= \deg\Phi_i.$$
Hence $\mathbb{C}(U_{i+1})^{\Gamma_i}= \Phi_i^*\mathbb{C}(U_i)$, so the induced morphism is birational. It is therefore an isomorphism because $U_i$ is normal. Taking global functions gives
\begin{equation}\label{eq: cycle-invariants}
\mathrm{Cox}(V_{i+1})^{\Gamma_i}= \Phi_i^*\mathrm{Cox}(V_i)\cong \mathrm{Cox}(V_i).
\end{equation}

\medskip

\noindent\textbf{Step 3.} We show that the groups $\Gamma_i$ are reflection groups and compare the numbers of Cox-ring generators.

For each $i$, let $\mathfrak m_i$ be the homogeneous maximal ideal of $\mathrm{Cox}(V_i)$ and put
$$E_i= \mathfrak m_i/\mathfrak m_i^2.$$
Fix $i$. For each prime divisor in the ramification locus of $\Phi_i$, its divisorial inertia group is the subgroup of its stabilizer in $\Gamma_i$ that acts trivially on the residue field of the divisor. Let $I$ be the normal subgroup of $\Gamma_i$ generated by these inertia groups.

The connected normal variety $U_{i+1}/I$ admits a finite morphism
$$U_{i+1}/I\to U_{i+1}/\Gamma_i\cong U_i.$$
Every divisorial inertia group of $U_{i+1}\to U_i$ lies in $I$, so its image in $\Gamma_i/I$ is trivial. Thus the intermediate morphism is unramified in codimension one. Since $U_i$ is smooth, Zariski--Nagata purity \cite[Expos\'{e}~X, Th\'{e}or\`{e}me~3.1]{SGA1} shows that it is \'{e}tale. By the Riemann existence theorem \cite[Expos\'{e}~XII, Th\'{e}or\`{e}me~5.1]{SGA1}, its analytification is a connected finite topological cover of $U_i^{\mathrm{an}}$. This cover has degree one because $U_i^{\mathrm{an}}$ is simply connected. Hence $I= \Gamma_i$.

It remains to prove that every nonidentity element of a divisorial inertia group is a pseudoreflection. Fix such an element $\gamma$ in the inertia group of a ramification prime divisor. The action of $S_{i+1}$ preserves the ramification locus and hence permutes its finitely many prime divisors. Since $S_{i+1}$ is connected, it preserves the chosen divisor. Its closure in $\mathrm{Spec}\mathrm{Cox}(V_{i+1})$ is a prime divisor. Since $\mathrm{Cox}(V_{i+1})$ is factorial, this closure has an equation $z$. The $S_{i+1}$-invariance of this divisor and the fact that every unit of the Cox ring is constant allow us to choose $z$ homogeneous. As $\gamma$ acts trivially on the residue field of the divisor,
$$(\gamma-\mathrm{id})\mathrm{Cox}(V_{i+1})\subseteq(z).$$

We use $H_{i+1}$ to define a positive integer grading: an element of Cox degree $\lambda\in\Lambda_{i+1}$ is assigned the integer degree $\lambda\cdot H_{i+1}^{r-1}$. Every nonzero element of nonzero Cox degree has positive degree, since its section defines a nonzero effective divisor. This gives a connected positive grading whose homogeneous maximal ideal is $\mathfrak m_{i+1}$.

Fix a $\Gamma_i$-equivariant section of $\mathfrak m_{i+1}\to E_{i+1}$ preserving the Cox grading, by averaging. Its image generates $\mathrm{Cox}(V_{i+1})$ by the graded Nakayama lemma. Any element of $\Gamma_i$ acting trivially on $E_{i+1}$ therefore acts trivially on the Cox ring and hence on $U_{i+1}$. Since the action on $U_{i+1}$ is faithful, the action on $E_{i+1}$ is faithful.

The homogeneous equation $z$ is nonconstant, so $z\in\mathfrak m_{i+1}$. The image of $(z)$ in $E_{i+1}$ is contained in the line spanned by $[z]$. We therefore have
$$\mathrm{im}((\gamma-\mathrm{id})|_{E_{i+1}})\subseteq\mathbb{C}\cdot [z].$$
The faithfulness just proved and the assumption $\gamma\neq 1$ show that this image is nonzero. Thus $\gamma$ is a pseudoreflection on $E_{i+1}$. Since $I= \Gamma_i$, the group $\Gamma_i$ is generated by pseudoreflections.

For each $i$, the chosen section gives a $\Gamma_i$-equivariant surjection
$$\mathrm{Sym}(E_{i+1})\twoheadrightarrow\mathrm{Cox}(V_{i+1}).$$
Taking invariants is exact in characteristic zero, so \eqref{eq: cycle-invariants} gives a surjection
$$\mathrm{Sym}(E_{i+1})^{\Gamma_i}\twoheadrightarrow\mathrm{Cox}(V_i).$$
By Theorem~\ref{thm: reflection-invariants}, the source is a polynomial ring of dimension $\dim E_{i+1}$. Its homogeneous maximal ideal modulo its square therefore also has dimension $\dim E_{i+1}$. The last surjection induces a surjection from this quotient onto $E_i$, and hence
$$\dim E_i\leq\dim E_{i+1}.$$
Around the cycle all these dimensions are equal. Denote their common value by $e$. Consequently, the induced map on indecomposable quotients is an isomorphism for every $i$.

\medskip

\noindent\textbf{Step 4.} We compare the relation ideals in compatible polynomial presentations.

The sections chosen in Step~3 give presentations
$$P_i:= \mathrm{Sym}(E_i)\twoheadrightarrow\mathrm{Cox}(V_i),\qquad I_i:= \ker(P_i\to\mathrm{Cox}(V_i)).$$
The surjection $P_{i+1}^{\Gamma_i}\twoheadrightarrow\mathrm{Cox}(V_i)$ induces an isomorphism on indecomposable quotients. Lift the images of the homogeneous variables of $P_i$ homogeneously to $P_{i+1}^{\Gamma_i}$. Their classes form a basis of its indecomposable quotient, so they generate this invariant ring by graded Nakayama. The resulting surjection $P_i\to P_{i+1}^{\Gamma_i}$ is an isomorphism, since both are polynomial rings of dimension $e$. We obtain compatible inclusions
$$\theta_i:P_i\xrightarrow{\sim}P_{i+1}^{\Gamma_i}\hookrightarrow P_{i+1},$$
whose degree maps are $\varphi_i^*:\Lambda_i\to\Lambda_{i+1}$.

The ring $P_{i+1}$ is free of rank $|\Gamma_i|$ over $P_{i+1}^{\Gamma_i}$. Indeed, the homogeneous basic invariants, for the ordinary polynomial grading, form a homogeneous system of parameters of $P_{i+1}$ and hence a regular sequence. Lifting a homogeneous $\mathbb C$-basis of the quotient of $P_{i+1}$ by the ideal generated by these basic invariants gives a free module over $P_{i+1}^{\Gamma_i}$ surjecting onto $P_{i+1}$ by graded Nakayama. The Hilbert series formula for the quotient by this homogeneous regular sequence shows that the source and target of this surjection have the same Hilbert series. Hence the surjection is an isomorphism, and the rank is the product of the basic degrees, namely $|\Gamma_i|$.

Consequently, the natural surjection
$$P_{i+1}/\theta_i(I_i)P_{i+1}\twoheadrightarrow\mathrm{Cox}(V_{i+1})$$
has source free of rank $|\Gamma_i|$ over the domain $\mathrm{Cox}(V_i)$. Its target has the same rank by \eqref{eq: cycle-invariants} and the faithfulness of the $\Gamma_i$-action. Its kernel is therefore a torsion submodule of a free module over a domain, and is zero. Thus
\begin{equation}\label{eq: cycle-relation-pullback}
I_{i+1}= \theta_i(I_i)P_{i+1}.
\end{equation}
Composing around the cycle gives an endomorphism $\Theta:= \theta_{s-1}\circ\cdots\circ\theta_0:P_0\to P_0$ whose degree map is $\Psi_0^*$ and such that
\begin{equation}\label{eq: cycle-return-ideal}
I_0= \Theta(I_0)P_0.
\end{equation}

\medskip

\noindent\textbf{Step 5.} We use int-amplification to show that the relation ideals vanish.

Let $x_1,\ldots,x_e$ be the homogeneous variables of $P_0$, with Cox degrees $D_1,\ldots,D_e$. These classes are nonzero because $\mathrm{Cox}(V_0)_0= \mathbb C$, and Step~1 identifies $\Lambda_0\otimes_{\mathbb Z}\mathbb R$ with $N^1(V_0)_{\mathbb R}$. Since $\Psi_0$ is int-amplified, \cite[Theorem~1.1]{Meng20} gives $(\Psi_0^*)^{-N}\to 0$. Hence, for all sufficiently large $N$,
$$(\Psi_0^*)^ND_j\notin\{D_1,\ldots,D_e\}\qquad(1\leq j\leq e).$$
Otherwise, since there are only finitely many pairs $(j,k)$, there would be fixed indices $j,k$ such that $(\Psi_0^*)^ND_j=D_k$ for arbitrarily large $N$. Along these values of $N$, we would have $D_j=(\Psi_0^*)^{-N}D_k\to0$, a contradiction. Homogeneity and $(\Psi_0^*)^ND_j\neq 0$ now show that $\Theta^N(x_j)$ has neither constant nor linear terms.

Every element of $I_0$ has zero constant term. If $I_0\neq 0$, let $d\geq 1$ be the smallest total degree of a monomial with nonzero coefficient in an element of $I_0$. After substituting $\Theta^N(x_j)$ for each variable $x_j$ and expanding, every monomial occurring in the image of a relation has total degree at least $2d$. Thus every element of $\Theta^N(I_0)P_0$ has no terms of degree less than $2d$. This contradicts the identity $I_0= \Theta^N(I_0)P_0$, obtained by iterating \eqref{eq: cycle-return-ideal}. Hence $I_0=0$, and \eqref{eq: cycle-relation-pullback} gives $I_i=0$ for every $i$. Thus every $\mathrm{Cox}(V_i)$ is a polynomial ring, and Proposition~\ref{prop: polynomial-cox-toric} shows that every $V_i$ is toric.
\end{proof}

\medskip

\subsection{Extending local triviality}

We first extend local triviality from a nonempty open subset to the entire base, and then give a criterion in terms of fibers over a dense subset.

\begin{lemma}\label{lem: extending-local-triviality}
Let $B$ be a smooth irreducible complex variety and let $g:B\to B$ be a finite surjective morphism such that
$$\text{if}\ \ Z\subsetneq B\text{ is closed and}\ g^{-1}(Z)\subseteq Z,\ \text{then} \ Z=\emptyset.$$
Let $p:X\to B$ be a smooth projective surjective morphism with connected fibers. Suppose that there is a finite surjective morphism $f:X\to X$ satisfying $p\circ f=g\circ p$. If $p$ is \'{e}tale locally trivial over a nonempty Zariski open subset of $B$, then it is \'{e}tale locally trivial over $B$, and all its fibers are isomorphic.
\end{lemma}

\begin{proof}
\noindent\textbf{Step 1.} We prove that a smooth projective family admitting a proper surjective morphism from a product family is \'{e}tale locally trivial.

Let $V$ be a smooth connected complex variety, let $r:Y\to V$ be a smooth projective morphism with connected fibers, and suppose that there is a proper surjective $V$-morphism $\varphi:F\times V\to Y$, where $F$ is a smooth connected projective variety. Write $F\times V\to Y'\xrightarrow{\nu} Y$ for its Stein factorization, and put $e= \deg \nu$. Since $Y$ is normal, the normalized field trace of $\nu$ restricts to an $\mathcal O_Y$-linear map
$$\tau:\varphi_*\mathcal O_{F\times V}= \nu_*\mathcal O_{Y'}\to\mathcal O_Y,\qquad \tau(1)=1.$$
Indeed, a local section of the finite algebra $\nu_*\mathcal O_{Y'}$ is integral over $\mathcal O_Y$, so its field trace is integral over $\mathcal O_Y$ and belongs to $\mathcal O_Y$ by normality. Here the normalization is division by $e$, which is nonzero over $\mathbb C$.

Fix $v_0\in V$. After shrinking $V$, choose \'{e}tale coordinates $t_1,\ldots,t_n$ centered at $v_0$. On $F\times V$, let $\partial_i$ be the derivation differentiating in the $t_i$-direction and fixing the $F$-coordinate. For a local regular function $u$ on $Y$, define
$$\xi_i(u)= \tau\bigl(\partial_i(\varphi^*u)\bigr).$$
The $\mathcal O_Y$-linearity of $\tau$ and the Leibniz rule for $\partial_i$ give the Leibniz rule for $\xi_i$. These derivations commute with restriction and therefore define algebraic vector fields on $Y$. Since $\tau(1)=1$, they satisfy $\xi_i(r^*t_j)= \delta_{ij}$.

Write $t_j$ also for $r^*t_j$, and put $Y_0=Y_{v_0}$, with inclusion $\iota:Y_0\hookrightarrow Y$. For $m\geq0$, let $A_m= \mathbb C[T_1,\ldots,T_n]/(T_1,\ldots,T_n)^{m+1}$. The chosen \'{e}tale coordinates identify $\mathcal O_{V,v_0}/\mathfrak m_{v_0}^{m+1}$ with $A_m$ by sending $t_j$ to $T_j$. Set $Y_m=Y\times_V\operatorname{Spec}A_m$.

On $Y_0$, consider the sheaf $\mathcal B_m=(\iota^{-1}\mathcal O_Y)\otimes A_m$. Here the $t_j$ lie in the first factor and the auxiliary parameters $T_j$ in the second. Extend each $\xi_j$ $A_m$-linearly to $\mathcal B_m$, so that it fixes the $T_k$, and put
$$D_m= \sum_{j=1}^nT_j\xi_j,\qquad E_m= \exp(D_m)= \sum_{k=0}^{m}\frac{D_m^k}{k!}.$$
The derivation $D_m$ raises the $(T_1,\ldots,T_n)$-adic filtration, so $D_m^{m+1}=0$. The Leibniz rule shows that $E_m$ is an $A_m$-algebra automorphism with inverse $\exp(-D_m)$. Since $D_m(t_j)=T_j$ and $D_m(T_j)=0$, we have $E_m(t_j-T_j)=t_j$. Thus $E_m$ induces an isomorphism of $A_m$-algebras
$$\mathcal B_m/(t_1-T_1,\ldots,t_n-T_n)\xrightarrow{\sim}\mathcal B_m/(t_1,\ldots,t_n).$$
On the left, substituting $T_j=t_j$ identifies the quotient with $\iota^{-1}\mathcal O_Y/(t_1,\ldots,t_n)^{m+1}= \mathcal O_{Y_m}$. On the right, the quotient is $\mathcal O_{Y_0}\otimes A_m$. We therefore obtain an isomorphism $Y_m\cong Y_0\times_{\mathbb C}\operatorname{Spec}A_m$ over $\operatorname{Spec}A_m$.

The exponentials $E_m$ commute with truncation along $A_{m+1}\to A_m$, so these isomorphisms are compatible as $m$ varies. Thus $Y\to V$ is formally trivial at $v_0$. By \cite[Proposition~6.3]{Pocz25}, it is \'{e}tale locally trivial there. Since $v_0$ was arbitrary, this proves the claim.

\medskip
\noindent\textbf{Step 2.} We extend local triviality to the entire base.

Let $U\subseteq B$ be the locus where $p$ admits an \'{e}tale local product presentation, without fixing the fiber in advance. This is open because \'{e}tale morphisms are open, and it is nonempty by assumption. We prove that $g(U)\subseteq U$.

The relative morphism
$$(f,p):X\to X\times_{B,g}B$$
is proper because $p$ is proper, and its fibers are finite because $f$ is finite. Hence it is finite. The fibers of $p$ are smooth and connected, so they are irreducible, and they have a common dimension. Thus the finite morphism $f|_{X_b}:X_b\to X_{g(b)}$ is surjective for every $b\in B$: its image is closed and has the same dimension as $X_{g(b)}$. It follows that $(f,p)$ is surjective.

Fix $b\in U$. Choose a connected \'{e}tale neighborhood $v:V\to B$ with a point $v_0$ over $b$ and a product presentation $X\times_B V\cong F\times V$. Pulling back $(f,p)$ gives a finite surjective $V$-morphism
$$F\times V\to X\times_{B,g\circ v}V.$$
After replacing $V$ by an \'{e}tale neighborhood of $v_0$, Step~1 gives an isomorphism
$$X\times_{B,g\circ v}V\cong F'\times V$$
for a smooth projective variety $F'$.

Since $g$ need not be \'{e}tale, it remains to prove \'{e}tale local triviality near $g(b)$. The finite surjective morphism $g$ is flat because $B$ is smooth. Consequently, $W=(g\circ v)(V)$ is open, and $g\circ v:V\to W$ is an fppf cover. The relative Isom scheme
$$\mathcal I= \mathrm{Isom}_W(F'\times W,X\times_B W)$$
exists by \cite[Proposition~2.5]{Pocz25}, and its pullback to $V$ is $\mathrm{Aut}(F')\times V$. The group scheme $\mathrm{Aut}(F')$ is smooth in characteristic zero. Smoothness and surjectivity descend under fppf covers, so $\mathcal I\to W$ is smooth and surjective. It therefore has \'{e}tale local sections \cite[Chapter~I, Proposition~3.26]{Miln80}. These sections give \'{e}tale local product presentations of $p$ over $W$. In particular, $g(b)\in U$.

We have proved $g(U)\subseteq U$, or equivalently $g^{-1}(B\setminus U)\subseteq B\setminus U$. The hypothesis on $g$ gives $U=B$. Finally, local triviality makes the fiber isomorphism class locally constant, and connectedness of $B$ shows that all fibers are isomorphic.
\end{proof}

\begin{remark}
By Lemma~\ref{lem: no backward invariant set}, the condition on closed subsets in Lemma~\ref{lem: extending-local-triviality} holds whenever $B$ is an abelian variety and $g:B\to B$ is an int-amplified isogeny.
\end{remark}

We next give a criterion for local triviality in terms of fibers over a Zariski dense subset.

\begin{proposition}\label{prop: local-triviality-criterion}
Let $B$ be a smooth irreducible complex variety and let $g:B\to B$ be a finite surjective morphism such that no nonempty proper closed subset $Z\subseteq B$ satisfies $g^{-1}(Z)\subseteq Z$. Let $p:X\to B$ be a smooth projective surjective morphism with connected fibers. Suppose that a finite surjective morphism $f:X\to X$ satisfies $p\circ f=g\circ p$. Fix a relatively ample line bundle $\mathcal H$ and a Zariski dense subset of closed points $S\subseteq B$. For $b\in B$, put $X_b=p^{-1}(b)$ and $\mathcal H_b= \mathcal H|_{X_b}$. Then the following conditions are equivalent:
\begin{enumerate}
\item The morphism $p$ is \'{e}tale locally trivial over $B$.
\item There are finitely many smooth connected projective varieties $F_i$ with ample line bundles $\mathcal L_i$ such that, for every $b\in S$, there is an isomorphism $u_b:F_i\to X_b$ for some $i$ satisfying $u_b^*\mathcal H_b\equiv\mathcal L_i$.
\item There are a smooth connected projective variety $F$ of dimension $d\ge1$, an ample line bundle $\mathcal L$ on $F$, and a constant $M>0$ such that, for every $b\in S$, there is a surjective morphism $u_b:F\to X_b$ satisfying
$$(u_b^*\mathcal H_b\cdot\mathcal L^{d-1})\le M.$$
\end{enumerate}
When these conditions hold, all fibers of $p$ are isomorphic.
\end{proposition}

\begin{proof}
If the fibers have dimension zero, then $p$ is an isomorphism, and all three conditions hold; in (3), take $F= \mathbf P^1$ and $\mathcal L= \mathcal O_{\mathbf P^1}(1)$. We therefore assume that the fibers have positive dimension.

\medskip

\noindent\textbf{Step 1.} We prove $(1)\Rightarrow(2)\Rightarrow(3)$.

Assume (1). Since $B$ is quasi-compact, we can choose finitely many \'{e}tale morphisms $V_i\to B$ from connected varieties whose images cover $B$, together with trivializations $X\times_B V_i\cong F_i\times V_i$ over $V_i$. The pullback of $\mathcal H$ to $F_i\times V_i$ has a constant numerical class on the fibers. Let $\mathcal L_i$ be its restriction to one fiber. These finitely many pairs satisfy (2), in fact for every $b\in B$.

Now assume (2), and let $F= \prod_iF_i$ with projections $\mathrm{pr}_i$. Put $\mathcal L= \bigotimes_i\mathrm{pr}_i^*\mathcal L_i$, which is ample. Composing $\mathrm{pr}_i$ with the isomorphism in (2) gives a surjective morphism from $F$ to $X_b$, whose pullback of $\mathcal H_b$ is numerically equivalent to $\mathrm{pr}_i^*\mathcal L_i$. Thus the relevant intersection numbers belong to the finite set
$$\bigl\{(\mathrm{pr}_i^*\mathcal L_i\cdot\mathcal L^{\dim F-1})\bigr\}_i.$$
Taking their maximum as $M$ proves (3).

\medskip

\noindent\textbf{Step 2.} We use the bound in (3) to obtain finitely many Hilbert polynomials for the graphs of the morphisms $u_b$.

Assume (3). The classes $u_b^*\mathcal H_b$ are nef. The linear functional $D\mapsto(D\cdot\mathcal L^{d-1})$ is strictly positive on every nonzero class in $\mathrm{Nef}(F)$. For $d=1$ this is immediate. For $d\geq 2$, if a nonzero nef class $D$ had $(D\cdot\mathcal L^{d-1})=0$, the Hodge index theorem would give $(D^2\cdot\mathcal L^{d-2})<0$, contradicting nefness. Consequently,
$$\{D\in\mathrm{Nef}(F):(D\cdot\mathcal L^{d-1})\leq M\}$$
is compact. The integral numerical classes form a lattice in $N^1(F)_{\mathbb R}$, so only finitely many numerical classes of $u_b^*\mathcal H_b$ occur.

The graph of $u_b$ is isomorphic to $F$, and its Hilbert polynomial with respect to $\mathrm{pr}_F^*\mathcal L\otimes\mathrm{pr}_X^*\mathcal H$ is
$$k\mapsto\chi\bigl(F,(\mathcal L\otimes u_b^*\mathcal H_b)^{\otimes k}\bigr).$$
By Riemann--Roch, this polynomial depends only on the numerical class of $u_b^*\mathcal H_b$. Hence only finitely many such polynomials occur. Taking a fixed sufficiently large power of the relatively ample line bundle $\mathrm{pr}_F^*\mathcal L\otimes\mathrm{pr}_X^*\mathcal H$ on $F\times X\to B$ gives a relatively very ample line bundle, with respect to which the Hilbert polynomials of the graphs still belong to a finite set.

\medskip

\noindent\textbf{Step 3.} We obtain local triviality on a nonempty open subset and then extend it to $B$.

For each of these Hilbert polynomials, graphs of morphisms form an open subscheme of the corresponding relative Hilbert scheme \cite[Proposition~2.5]{Pocz25}. The disjoint union of these open subschemes is a scheme of finite type over $B$ parametrizing morphisms $F\to X_b$ and containing all the chosen $u_b$. The locus parametrizing surjective morphisms is constructible. Indeed, the universal morphism is proper, so its image is closed in the target family. The complement has constructible image in the parameter space by Chevalley's theorem. This image consists precisely of the parameters for which the morphism is not surjective.

The image of the surjective locus contains $S$ and is therefore dense in $B$. Decomposing this locus into finitely many locally closed integral strata, we obtain one stratum $T$ whose morphism to $B$ is dominant. By generic smoothness, a nonempty open subset of $T$ is smooth over $B$. Its image is a nonempty open subset of $B$, and the resulting smooth surjection has \'{e}tale local sections \cite[Chapter~I, Proposition~3.26]{Miln80}. Pulling back the universal morphism along one such section gives an \'{e}tale morphism $V\to B$ with nonempty image and a proper surjective $V$-morphism
$$F\times V\to X\times_B V.$$
We may take $V$ connected. Step~1 of the proof of Lemma~\ref{lem: extending-local-triviality} shows that $X\times_B V\to V$ is \'{e}tale locally trivial. Composing the \'{e}tale neighborhoods with $V\to B$ shows that $p$ is \'{e}tale locally trivial over the nonempty open image of $V\to B$. Lemma~\ref{lem: extending-local-triviality} now gives (1) and shows that all fibers are isomorphic.
\end{proof}

We apply this criterion to periodic fibers over an abelian variety.

\begin{corollary}\label{cor: periodic-fiber-local-triviality}
Let $A$ be a complex abelian variety, let $g:A\to A$ be an int-amplified isogeny, and let $p:X\to A$ be a smooth projective surjective morphism with connected fibers. Suppose that a finite surjective morphism $f:X\to X$ satisfies $p\circ f=g\circ p$, and fix a relatively ample line bundle $\mathcal H$. Write $X_b=p^{-1}(b)$ for $b\in A$.

Then $p$ is \'{e}tale locally trivial over $A$ if and only if there are finitely many smooth connected projective varieties $F_i$ with ample line bundles $\mathcal L_i$ such that, for every periodic point $b$ of $g$, there is an isomorphism $u_b:F_i\to X_b$ for some $i$ satisfying
$$u_b^*(\mathcal H|_{X_b})\equiv\mathcal L_i.$$
In this case, all fibers of $p$ are isomorphic.
\end{corollary}

\begin{proof}
By Lemma~\ref{lem: no backward invariant set}, the isogeny $g$ satisfies the condition on closed subsets in Proposition~\ref{prop: local-triviality-criterion}. Its periodic points are Zariski dense by Lemma~\ref{lem: periodic points isogeny}. Applying the equivalence of (1) and (2) in that proposition to the set of periodic points proves the assertion.
\end{proof}

\medskip

\subsection{Isotriviality of the Albanese morphism}

We now apply the preceding results to the Albanese morphism.

Throughout this subsection, let $X$ be a smooth complex projective variety, let $f:X\to X$ be an int-amplified Galois endomorphism, and fix an ample Cartier divisor $H$ such that $f^*H-H$ is ample. Assume that $q(X)= \widetilde q(X)$, and fix the Albanese morphism $\mathrm{alb}_X:X\to A:= \mathrm{Alb}(X)$, the induced isogeny $g:A\to A$, and the subgroup $K\subseteq\mathrm{Gal}(f)$ preserving every fiber of $\mathrm{alb}_X$, as in Proposition~\ref{prop: albanese-structure}.

\begin{lemma}\label{lem: periodic-toric-fibers}
For every periodic point $b\in A$ of $g$, the fiber $X_b:= \mathrm{alb}_X^{-1}(b)$ is a smooth projective toric variety.
\end{lemma}

\begin{proof}
By Proposition~\ref{prop: albanese-structure}(1), all fibers of $\mathrm{alb}_X$ are smooth, projective, and rationally connected. For $c\in A$, Equation~\eqref{eq: relative quotient} identifies $X_c/K$ with $X_{g(c)}$, and under this identification the quotient map is $f|_{X_c}:X_c\to X_{g(c)}$. Thus this restriction is the quotient by the image of $K$ in $\mathrm{Aut}(X_c)$ and is Galois.

Choose $s>0$ such that $g^s(b)=b$. The restrictions of $f$ along the orbit of $b$ form a cycle of finite Galois morphisms. Since
$$(f^s)^*H-H= \sum_{j=0}^{s-1}(f^j)^*(f^*H-H)$$
is ample, its restriction to $X_b$ shows that the return map $f^s|_{X_b}$ is int-amplified. Lemma~\ref{lem: fiber-cycle} therefore shows that $X_b$ is toric.
\end{proof}

\begin{proposition}\label{prop: isotrivial-toric}
The smooth morphism $\mathrm{alb}_X:X\to A$ is \'{e}tale locally trivial with a fixed smooth projective toric fiber $F$.
\end{proposition}

\begin{proof}
We show that the periodic fibers have only finitely many polarized isomorphism types and apply Corollary~\ref{cor: periodic-fiber-local-triviality}.

Flatness of $\mathrm{alb}_X$ and relative ampleness of $H$ show that the Hilbert polynomial of $(X_b,H|_{X_b})$ is independent of $b$. In particular, the dimension of the fibers and the top self-intersection number of $H|_{X_b}$ are constant.

Let $b$ be periodic. The fiber $X_b$ is toric by Lemma~\ref{lem: periodic-toric-fibers}. Every line bundle on a toric variety admits a torus linearization, so $\mathcal O_{X_b}(H)$ determines a full-dimensional lattice polytope. Its normalized lattice volume is $(H|_{X_b})^{\dim X_b}$. The dimension and this volume are independent of $b$. By \cite[Theorem~2]{LZ91}, only finitely many such polytopes occur up to integral affine equivalence. Integral affine equivalence identifies the corresponding polarized toric varieties. Hence only finitely many isomorphism classes of polarized toric varieties occur among the periodic fibers.

Since $g$ is int-amplified, Corollary~\ref{cor: periodic-fiber-local-triviality}, applied with $\mathcal H= \mathcal O_X(H)$, shows that $\mathrm{alb}_X$ is \'{e}tale locally trivial over $A$ and that all its fibers are isomorphic. Their common fiber is toric by Lemma~\ref{lem: periodic-toric-fibers}.
\end{proof}

\medskip

\subsection{Proofs of Theorems~\ref{thm: rationally connected} and~\ref{thm: projective-space}}

In this subsection, we prove the folklore conjecture \cite[Conjecture~1.1]{HN11} under the Galois assumption.

\begin{proof}[Proof of Theorem~\ref{thm: rationally connected}]
If $\dim X=0$, then $X$ is a point and hence toric. Assume that $\dim X>0$. Lemma~\ref{lem: fiber-cycle} with $s=1$ shows that $X$ is toric. Its Cox ring is the polynomial ring with one variable for each torus-invariant prime divisor \cite[Section~1]{Cox95}.
\end{proof}

\begin{proof}[Proof of Theorem~\ref{thm: projective-space}]
Since $\rho(X)=1$ and the given endomorphism has degree greater than one, it is int-amplified. The variety $X$ is rationally connected \cite{Camp92,KMM92a}, so Theorem~\ref{thm: rationally connected} shows that $X$ is a smooth complete toric variety. Let $\Sigma$ be its fan. The toric divisor sequence \cite[Section~3.4]{Fult93} gives
$$\rho(X)=|\Sigma(1)|-n.$$
Thus $\Sigma$ has exactly $n+1$ rays. Choose a maximal cone with primitive ray generators $v_1,\ldots,v_n$. Smoothness makes these vectors a basis of the lattice, and there is a unique remaining primitive ray generator $v_0$. For every $i$, the facet generated by the $v_j$ with $j\neq i$ lies in exactly two maximal cones. The second one is generated by $v_0$ and these $n-1$ vectors. Write
$$v_0= \sum_{i=1}^n b_i v_i.$$
The two cones lie on opposite sides of their common facet, so $b_i<0$. Their smoothness gives $|b_i|=1$. Hence $v_0=-v_1-\cdots-v_n$. The chosen cone and its $n$ adjacent cones are all the possible maximal cones generated by $n$ of the $n+1$ rays. Thus $\Sigma$ is the fan of $\mathbf{P}^n$, and $X\cong\mathbf{P}^n$.
\end{proof}

\section{Relative Cox coordinates}
\label{sec: relative-cox}

Throughout this section, let $X$ be a smooth complex projective variety, let $f:X\to X$ be an int-amplified Galois endomorphism, and fix an ample Cartier divisor $H$ such that $f^*H-H$ is ample. Assume that $q(X)= \widetilde q(X)$. Let $\mathrm{alb}_X:X\to A:= \mathrm{Alb}(X)$ be the Albanese morphism, let $g:A\to A$ be the induced isogeny, and let $K\subseteq\mathrm{Gal}(f)$ be the subgroup that preserves every fiber of $\mathrm{alb}_X$. Thus
$$\mathrm{alb}_X\circ f = g\circ \mathrm{alb}_X\ \ \ \text{and}\ \ \  X/K\cong X\times_{A,g}A$$
by Proposition~\ref{prop: albanese-structure}. The projective morphism $\mathrm{alb}_X$ is smooth by Proposition~\ref{prop: smooth albanese}. By Proposition~\ref{prop: isotrivial-toric}, it is \'{e}tale locally trivial with a fixed smooth projective toric fiber $F$.

We show that the family becomes a split toric fibration after base change by an isogeny of $A$. We first construct the relative Cox algebra and use the reflection group action to find homogeneous line bundles that freely generate it. The resulting diagonal changes of Cox coordinates give the split toric structure. We then choose these line bundles compatibly with the finite group actions so that the toric boundary descends to the corresponding quotient. If $\dim F=0$, the Albanese morphism is an isomorphism, so we assume $\dim F>0$ for the remainder of this section.

\medskip

\subsection{The relative Cox algebra}

Every isogeny $A'\to A$ is dominated by a multiplication isogeny: there is an isogeny $A\to A'$ such that the composition $A\to A'\to A$ is $[m]$ for some $m>0$ \cite[p.~115]{Miln86}. Whenever a base change by an isogeny is needed below, we may instead pull back along a multiplication isogeny factoring through it. Multiplication isogenies commute with the induced isogeny $g:A\to A$. For a base change by $[m]:A\to A$, put $X'=X\times_{A,[m]}A$ and denote the projections by $r$ and $p'$, as in the Cartesian square
$$
\begin{tikzcd}
X' \arrow[r, "r"] \arrow[d, "p'"']   &   X \arrow[d, "\mathrm{alb}_X"]\\
A \arrow[r, "{[m]}"']   & A
\end{tikzcd}
$$
The projection $r$ is finite \'{e}tale, and $f$ lifts to $f':X'\to X'$ given by $f'(x,a)=(f(x),g(a))$. The morphism $p'$ remains smooth and \'{e}tale locally trivial with connected fibers, and flat base change gives $(p')_*\mathcal O_{X'}= \mathcal O_A$. The identity $r\circ f'=f\circ r$ shows that $f'$ is finite and surjective, and
$$
f'^*r^*H-r^*H=r^*(f^*H-H) 
$$
is ample. The relative quotient by $K$ also commutes with this flat base change, so $X'/K\cong X'\times_{A,g}A$. These are the properties used in the constructions below. We reuse the notation $X,f,H$ and the notation for the other pullback objects after each replacement. Only finitely many replacements occur. We keep track of the composite isogeny from the final base to the base at the start of this section and use it in the proof of Theorem~\ref{thm: main}.

Recall that $F$ denotes the fixed smooth projective toric fiber of $\mathrm{alb}_X:X\to A$. Put $\Lambda= \mathrm{Pic}(F)$ and $S= \mathrm{Hom}(\Lambda,\mathbf{G}_m)$. Choose line bundles representing a basis of $\Lambda$ and form $\mathcal L_{F,\lambda}$ by tensor products. Set
$$
U_F:= \mathrm{Spec}_F\left(\bigoplus_{\lambda\in\Lambda}\mathcal L_{F,\lambda}\right)\to F.
$$
With the right action described after \eqref{eq: characteristic space}, this projection is an $S$-torsor. Lemma~\ref{lem: cox-characteristic-space} gives $\Gamma(U_F,\mathcal O_{U_F})= \mathrm{Cox}(F)$.

\begin{proposition}\label{prop: relative-cox-construction}
After a base change by a multiplication isogeny, there are line bundles $\mathcal L_\lambda$ on $X$, indexed by $\lambda\in\Lambda$, whose restrictions identify $\Lambda$ with the Picard group of every fiber, together with compatible isomorphisms $\mathcal L_\lambda\otimes\mathcal L_{\mu}\cong \mathcal L_{\lambda + \mu}$. The relative Cox algebra
$$
\mathcal R:= \bigoplus_{\lambda\in\Lambda}(\mathrm{alb}_X)_*\mathcal L_\lambda 
$$
is \'{e}tale locally isomorphic to $\mathrm{Cox}(F)\otimes \mathcal O_A$ as a $\Lambda$-graded algebra. Moreover, the characteristic space
$$
\widehat{\mathcal X}:= \mathrm{Spec}_X\left(\bigoplus_{\lambda\in\Lambda}\mathcal L_\lambda\right) 
$$
is an $S$-torsor over $X$ and admits a natural open immersion $\widehat{\mathcal X}\hookrightarrow\mathrm{Spec}_A\mathcal R$. The same local trivializations identify this torsor and this open immersion with the products of $U_F\to F$ and $U_F\hookrightarrow\mathrm{Spec}\mathrm{Cox}(F)$ with the base.
\end{proposition}

\begin{proof}
Since $X\to A$ is \'{e}tale locally trivial with fiber $F$, its local trivializations form an $\mathrm{Aut}(F)$-torsor over $A$. Since $\mathrm{Aut}(F)/\mathrm{Aut}^0(F)$ is finite \cite[Corollary~4.7]{Cox95}, the quotient of this torsor by $\mathrm{Aut}^0(F)$ is a finite \'{e}tale $\mathrm{Aut}(F)/\mathrm{Aut}^0(F)$-torsor over $A$.

Choosing a point of this finite \'{e}tale torsor over $0\in A$ gives a monodromy homomorphism
$$
\rho:\pi_1^{\mathrm{\acute{e}t}}(A,0)\to\mathrm{Aut}(F)/\mathrm{Aut}^0(F).
$$
Its image is finite and abelian, since the target is finite and $\pi_1^{\mathrm{\acute{e}t}}(A,0)$ is abelian. Choose a positive integer $m$ divisible by the order of every element of $\mathrm{im}(\rho)$. Writing the fundamental group additively, the multiplication isogeny $[m]:A\to A$ induces $[m]_*(\gamma)=m\gamma$. The monodromy of the torsor pulled back along $[m]$ is the composite $\rho\circ[m]_*$. Hence, for every $\gamma\in\pi_1^{\mathrm{\acute{e}t}}(A,0)$,
$$
(\rho\circ[m]_*)(\gamma)= \rho(m\gamma)= \rho(\gamma)^m=1,
$$
where the last equality follows from the choice of $m$. Thus the pullback torsor has trivial monodromy and is therefore a trivial finite \'{e}tale torsor. Choosing a section of this quotient torsor reduces the structure group of the pullback $\mathrm{Aut}(F)$-torsor to $\mathrm{Aut}^0(F)$. After this base change, we can therefore choose local trivializations of $X\to A$ so that all transition maps take values in $\mathrm{Aut}^0(F)$.

If we denote this principal $\mathrm{Aut}^0(F)$-torsor by $\mathcal P\to A$, then
$$
X\cong\mathcal P\times^{\mathrm{Aut}^0(F)}F. 
$$
The connected group $\mathrm{Aut}^0(F)$ acts trivially on the discrete lattice $\Lambda$. Let $\widetilde{\mathrm{Aut}}^0(F)$ be the group of $\Lambda$-graded Cox automorphisms inducing elements of $\mathrm{Aut}^0(F)$ on $F$. Every element of $\widetilde{\mathrm{Aut}}^0(F)$ preserves the characteristic space $U_F$. Then Cox's construction \cite[Theorem~4.2 and Corollary~4.7]{Cox95}, with the correction in \cite{Cox14}, gives a central exact sequence
$$
1\to S\to\widetilde{\mathrm{Aut}}^0(F)\to\mathrm{Aut}^0(F)\to 1.
$$

To construct the relative Cox algebra, it suffices to lift $\mathcal P$ through this extension. Such a lift glues the local copies of $U_F\to F$ into an $S$-torsor over $X$. The line bundles $\mathcal L_\lambda$ associated with the characters of $S$ then carry compatible isomorphisms $\mathcal L_\lambda\otimes\mathcal L_\mu\cong\mathcal L_{\lambda+\mu}$.

Since $S$ is smooth, the quotient map $\widetilde{\mathrm{Aut}}^0(F)\to\mathrm{Aut}^0(F)$ is a smooth $S$-torsor and admits \'{e}tale local sections. The central extension above therefore defines an obstruction to lifting $\mathcal P$ to a principal $\widetilde{\mathrm{Aut}}^0(F)$-torsor:
$$
\partial(\mathcal P)\in H^2_{\mathrm{\acute{e}t}}(A,S).
$$

To describe this class, choose an \'{e}tale cover $\{U_i\to A\}$ trivializing $\mathcal P$, with transition functions $g_{ij}$ satisfying $g_{ij}g_{jk}=g_{ik}$. After further \'{e}tale localization on the overlaps, choose lifts $\widetilde g_{ij}$ to $\widetilde{\mathrm{Aut}}^0(F)$. On common refinements of the triple overlaps, put
$$ c_{ijk}:=\widetilde g_{ij}\widetilde g_{jk}\widetilde g_{ik}^{-1}. $$
The image of $c_{ijk}$ in $\mathrm{Aut}^0(F)$ is the identity, so $c_{ijk}$ lies in $S$. Since $S$ is central, the elements $c_{ijk}$ satisfy the $2$-cocycle condition on common refinements of the quadruple overlaps and represent $\partial(\mathcal P)$.

If $\partial(\mathcal P)=0$, then after refining the \'{e}tale cover we can choose $S$-valued functions $s_{ij}$ such that the lifts $h_{ij}:=s_{ij}\widetilde g_{ij}$ satisfy
$$h_{ij}h_{jk}=h_{ik}.$$
These transition functions define a principal $\widetilde{\mathrm{Aut}}^0(F)$-torsor lifting $\mathcal P$.

After choosing a basis of $\Lambda$, the obstruction is a list of classes in $H^2_{\mathrm{\acute{e}t}}(A,\mathbf{G}_m)$. By \cite[Lemma~8.1]{Olss25}, these finitely many classes all vanish after pullback along a suitable isogeny of $A$. Replacing this isogeny by a dominating multiplication isogeny as above, we may assume that the obstruction vanishes. The modified transition functions then define a principal $\widetilde{\mathrm{Aut}}^0(F)$-torsor $\widetilde{\mathcal P}\to A$ lifting $\mathcal P$. Put
$$
\widehat{\mathcal X}:= \widetilde{\mathcal P}\times^{\widetilde{\mathrm{Aut}}^0(F)}U_F.
$$
The natural map $\widehat{\mathcal X}\to\mathcal P\times^{\mathrm{Aut}^0(F)}F\cong X$ is an $S$-torsor. For each $\lambda\in\Lambda$, let $\mathcal L_\lambda$ be the line bundle on $X$ associated with this $S$-torsor and the character $\lambda$. Its restriction to every fiber represents $\lambda$. These line bundles carry compatible isomorphisms $\mathcal L_\lambda\otimes\mathcal L_\mu \cong \mathcal L_{\lambda + \mu}$, and identify the torsor with the relative spectrum in the statement.

An \'{e}tale local trivialization of $\widetilde{\mathcal P}$ identifies these line bundles and their multiplication with the pullbacks of those defining $\mathrm{Cox}(F)$. It therefore identifies $\mathcal R$ with $\mathrm{Cox}(F)\otimes \mathcal O_A$ and the torsor $\widehat{\mathcal X}\to X$ with the product of $U_F\to F$ and the base. The natural map $\widehat{\mathcal X}\to\mathrm{Spec}_A\mathcal R$ is then the product of the standard open immersion $U_F\hookrightarrow\mathrm{Spec}\mathrm{Cox}(F)$ and the base. This proves the proposition.
\end{proof}

\medskip

We fix this Cox construction for the remainder of the section. Its degree-zero piece is $\mathcal R_0= \mathcal O_A$. By the compatible local product presentations above and flat base change \cite[Proposition~III.9.3]{Har77}, each graded piece $\mathcal R_\lambda=(\mathrm{alb}_X)_*\mathcal L_\lambda$ is locally free, and taking this direct image commutes with base change. In particular, its fiber at a closed point $y\to A$ is naturally $H^0(X_y,\mathcal L_\lambda|_{X_y})$. Let $\mathfrak m\subseteq\mathcal R$ be the homogeneous ideal generated by the pieces $\mathcal R_\lambda$ with $\lambda\neq 0$, and put
$\mathcal V= \mathfrak m/\mathfrak m^2.$

Let $e$ be the number of torus-invariant prime divisors of $F$, which is also the number of Cox variables. Then $e = \dim F+ \rho(F)$. The bundle $\mathcal V$ has rank $e$, and its nonzero homogeneous pieces have precisely the degrees of the Cox variables of $F$. Its sections represent homogeneous generators modulo products. Our aim is to lift these classes to line subbundles of $\mathfrak m$ that freely generate $\mathcal R$.

Under the fiberwise Picard identifications above, the restrictions of $H$ determine a single ample class $H_F$ on $F$.

We next describe how $f$ pulls back the line bundles defining the relative Cox algebra. We first determine their classes on the fibers, and then compare the resulting line bundles on $X$.

Let $X_g=X\times_{A,g}A$, with projections $p:X_g\to X$ and $p_g:X_g\to A$. The morphism $u=(f,\mathrm{alb}_X):X\to X_g$ is the relative quotient by $K$, by Proposition~\ref{prop: albanese-structure}(3) and the preceding base change argument. On the fiber over $y\in A$, it restricts to the finite morphism
$$u_y:X_y\to X_{g(y)}.$$
The identity $p\circ u=f$ gives $u^*p^*=f^*$.

\begin{proposition}\label{prop: relative-cox-pullback}
Under the fixed identifications $\mathrm{Pic}(X_y)\cong\Lambda$, the pullbacks $u_y^*$ are independent of $y$. Denote their common value by $B:\Lambda\to\Lambda$. Thus, for every $\mu\in\Lambda$ and $y\in A$,
$$u_y^*(\mathcal L_\mu|_{X_{g(y)}})\cong\mathcal L_{B\mu}|_{X_y}.$$
The homomorphism $B$ is injective and has finite-index image.

Globally, the line bundles $f^*\mathcal L_\mu$ and $\mathcal L_{B\mu}$ differ by a line bundle pulled back from $A$: there is a line bundle $\mathcal L_{A,\mu}$ on $A$ such that
\begin{equation}\label{eq: relative-twist}
f^*\mathcal L_\mu\cong\mathcal L_{B\mu}\otimes\mathrm{alb}_X^*\mathcal L_{A,\mu}.
\end{equation}
These isomorphisms can be chosen compatibly with tensor products, with
$$\mathcal L_{A,\mu+\nu}\cong\mathcal L_{A,\mu}\otimes\mathcal L_{A,\nu}.$$
Moreover, every eigenvalue of $B$ on $\Lambda\otimes_{\mathbb Z}\mathbb R$ has absolute value greater than one.
\end{proposition}

\begin{proof}
Under the fixed fiberwise Picard identifications, the pullbacks $u_y^*$ vary locally constantly. Since $A$ is connected, they give a single homomorphism $B:\Lambda\to\Lambda$. Each $u_y$ is finite and surjective, so its pullback is injective on the torsion-free Picard group. Since the source and target lattices have the same rank, the image has finite index.

The fiberwise identity above shows that
$$f^*\mathcal L_\mu\otimes\mathcal L_{B\mu}^{-1}$$
is trivial on every fiber of $\mathrm{alb}_X$. We show that this line bundle is pulled back from $A$. Each fiber is rationally connected, so $h^0(\mathcal O)=1$ and $h^1(\mathcal O)=0$. The cohomology and base change theorem therefore shows that
$$\mathcal L_{A,\mu}:=(\mathrm{alb}_X)_*\bigl(f^*\mathcal L_\mu\otimes\mathcal L_{B\mu}^{-1}\bigr)$$
is a line bundle, and the evaluation map gives an isomorphism
$$\mathrm{alb}_X^*\mathcal L_{A,\mu}\cong f^*\mathcal L_\mu\otimes\mathcal L_{B\mu}^{-1}.$$
Choose these isomorphisms for a basis of the free lattice $\Lambda$, and extend them using tensor products and duals. This gives the compatible isomorphisms in \eqref{eq: relative-twist}.

Finally, identify both fibers with $F$ compatibly with the fixed Picard identifications. The morphism $u_y$ then becomes a finite endomorphism of $F$ whose Picard pullback is $B$. The restriction of $f^*H-H$ has class $BH_F-H_F$, which is ample. Thus this endomorphism is int-amplified, and the assertion about eigenvalues follows from \cite[Theorem~1.1]{Meng20}.
\end{proof}

\medskip

For the induction on Cox generators, we need an additive function $\deg_B:\Lambda\to\mathbb Z$ that is positive on nonzero effective classes and satisfies
$$\deg_B(B\lambda)>\deg_B(\lambda)$$
for every such class $\lambda$. We will induct on $\deg_B(\lambda)$ in Proposition~\ref{prop: relative-generators}. The displayed inequality ensures that $\deg_B(\mu)<\deg_B(\lambda)$ whenever $\lambda=B\mu$ and $\mu$ is nonzero and effective. The construction uses the fact that every eigenvalue of $B$ has absolute value greater than one.

Recall that, under the fixed identifications of the fiberwise Picard groups with $\Lambda$, the pullback $u_y^*$ is represented by $B$. The pushforward on divisor classes satisfies
$$(u_y)_*u_y^*=(\deg u_y)\mathrm{id}.$$
Since $B$ is invertible on $\Lambda\otimes_{\mathbb Z}\mathbb R$, this gives
$$B^{-1}=(\deg u_y)^{-1}(u_y)_*.$$
Hence $B^{-1}$ preserves effective classes. By Proposition~\ref{prop: relative-cox-pullback}, all its eigenvalues have absolute value less than one, so the geometric series converges and gives
$$(B-\mathrm{id})^{-1}= \sum_{j\geq 1}B^{-j}.$$

\begin{definition}\label{def: relative degree}
Fix a sufficiently divisible positive integer $N$ such that
$N(B-\mathrm{id})^{-1}\Lambda\subseteq\Lambda$. For $\lambda\in\Lambda$, define
$$\deg_B(\lambda):=N\bigl((B-\mathrm{id})^{-1}\lambda\cdot H_F^{\dim F-1}\bigr).$$
Assigning degree $\deg_B(\lambda)$ to each homogeneous section of Cox degree $\lambda$ defines a $\mathbb Z$-grading in addition to the Cox grading.
\end{definition}

\medskip

This function is additive and $\Z$-valued. The series expression and the ampleness of $H_F$ show that it is positive on every nonzero effective class. Moreover,
\begin{equation}\label{eq: relative-degree-increase}
\deg_B(B\lambda)-\deg_B(\lambda)=N(\lambda\cdot H_F^{\dim F-1})>0
\end{equation}
for every such class $\lambda$, as required.

\medskip

\subsection{The Galois lift and pseudoreflections}

Recall that $F$ is the fixed smooth projective toric fiber of the Albanese morphism $\mathrm{alb}_X:X\to A= \mathrm{Alb}(X)$. Its Picard lattice and characteristic torus are
$$\Lambda= \mathrm{Pic}(F),\qquad S= \mathrm{Hom}(\Lambda,\mathbf{G}_m).$$
The line bundles $\mathcal L_\lambda$ constructed in Proposition~\ref{prop: relative-cox-construction} represent the classes $\lambda\in\Lambda$ on every fiber. They define the relative Cox algebra
$$\mathcal R= \bigoplus_{\lambda\in\Lambda}\mathcal R_\lambda,\qquad \mathcal R_\lambda=(\mathrm{alb}_X)_*\mathcal L_\lambda,$$
and the characteristic $S$-torsor
$$\widehat{\mathcal X}= \mathrm{Spec}_X\left(\bigoplus_{\lambda\in\Lambda}\mathcal L_\lambda\right)\to X.$$

Recall also that $g:A\to A$ is the isogeny induced by $f$. The finite group $K$ acts along the Albanese fibers, and its quotient map is
$$u=(f,\mathrm{alb}_X):X\to X_g:=X\times_{A,g}A.$$
Write $X_y= \mathrm{alb}_X^{-1}(y)$. The restriction of $u$ over $y\in A$ is $u_y:X_y\to X_{g(y)}$. Under the fixed identifications of the fiberwise Picard groups with $\Lambda$, all pullbacks $u_y^*$ are represented by the same injective homomorphism $B:\Lambda\to\Lambda$, whose image has finite index. Globally, Proposition~\ref{prop: relative-cox-pullback} gives line bundles $\mathcal L_{A,\mu}$ on $A$ and multiplicatively compatible isomorphisms
$$f^*\mathcal L_\mu\cong\mathcal L_{B\mu}\otimes\mathrm{alb}_X^*\mathcal L_{A,\mu},\qquad \mu\in\Lambda.$$
Thus $B$ records the pullback of fiberwise divisor classes, while $\mathcal L_{A,\mu}$ records the remaining factor pulled back from the base.

In this subsection, we construct a corresponding $S$-torsor $\widehat{\mathcal X}_g\to X_g$, together with a finite Galois lift
$$\widehat{\mathcal X}\to\widehat{\mathcal X}_g$$
of $u$. Its Galois group acts on $\mathcal R$.

On each fiber, we show that the induced action on generator classes modulo products is faithful and generated by pseudoreflections. These reflections will be used in the next subsection to construct the generating line bundles.

\begin{lemma}\label{lem: relative-theta}
Let $\widehat B:S\to S$ be the isogeny dual to $B:\Lambda\to\Lambda$, and put
$$K_B:= \ker(\widehat B)\cong\mathrm{Hom}(\Lambda/B\Lambda,\mathbf{G}_m).$$
We also write $K_B$ for the finite abelian group $\mathrm{Hom}(\Lambda/B\Lambda,\mathbb C^*)$ of its closed points. After a further base change by a multiplication isogeny, there is a finite group $\Gamma$ fitting into a central exact sequence
$$1\to K_B\to\Gamma\to K\to 1.$$

There is an $S$-torsor $\widehat{\mathcal X}_g\to X_g$ and a $\widehat B$-equivariant finite Galois lift $$\Phi:\widehat{\mathcal X}\to\widehat{\mathcal X}_g$$ of $u$ with $\mathrm{Gal}(\Phi)= \Gamma$. The induced action of $\Gamma$ on $\mathcal R$ is by $\mathcal O_A$-linear, $\Lambda$-graded algebra automorphisms. For $\mu\in\Lambda$, there are isomorphisms
\begin{equation}\label{eq: invariant-cox-sheaf}
(\mathcal R^\Gamma)_{B\mu} \cong g^*\mathcal R_\mu\otimes\mathcal L_{A,\mu}^{-1}.
\end{equation}
These isomorphisms are compatible with multiplication and with base change to closed points.
\end{lemma}

\begin{proof}
\noindent\textbf{Step 1.} We lift $u:X\to X_g$ to the characteristic spaces.

Let $p:X_g\to X$ and $p_g: X_g\to A$ be the two projections. To lift $u$, we first cancel the base factors $\mathcal L_{A,\mu}$ in the pullback formula for $f^*\mathcal L_{\mu}$. Define the target $S$-torsor by
$$\widehat{\mathcal X}_g= \mathrm{Spec}_{X_g}\left(\bigoplus_{\mu\in\Lambda}p^*\mathcal L_{\mu}\otimes p_g^*\mathcal L_{A,\mu}^{-1}\right).$$
Since $p\circ u=f$ and $p_g\circ u= \mathrm{alb}_X$, the pullback formula gives multiplicatively compatible isomorphisms
$$u^*\bigl(p^*\mathcal L_\mu\otimes p_g^*\mathcal L_{A,\mu}^{-1}\bigr)\cong \mathcal L_{B\mu}.$$
Taking the direct sum of these isomorphisms and then including the summands with degrees in $B\Lambda$ gives an $\mathcal O_X$-algebra homomorphism
$$\bigoplus_{\mu\in\Lambda}u^*\bigl(p^*\mathcal L_\mu\otimes p_g^*\mathcal L_{A,\mu}^{-1}\bigr)\xrightarrow{\sim}\bigoplus_{\mu\in\Lambda}\mathcal L_{B\mu}\hookrightarrow\bigoplus_{\lambda\in\Lambda}\mathcal L_\lambda.$$
The algebra on the left is the pullback by $u$ of the algebra defining $\widehat{\mathcal X}_g$. Taking relative spectra reverses the arrows and gives an $X$-morphism
$$\widehat{\mathcal X}\to X\times_{X_g}\widehat{\mathcal X}_g.$$
Composing with the projection to $\widehat{\mathcal X}_g$ defines
$$\Phi: \widehat{\mathcal X} \to \widehat{\mathcal X}_g.$$
Since a section of target Cox degree $\mu$ pulls back to one of source Cox degree $B\mu$, the morphism $\Phi$ is $\widehat B$-equivariant.

A decomposition $\Lambda/B\Lambda \cong \bigoplus_j\mathbb Z/n_j\mathbb Z$ identifies $K_B$ with $\prod_j\mu_{n_j}$, where $\mu_n$ is the group scheme of $n$th roots of unity. These group schemes are reduced over $\mathbb C$, so their actions can be viewed as actions of their groups of closed points. On a homogeneous section of Cox degree $\lambda$, the subgroup $K_B$ acts through the inverse of the character determined by the class of $\lambda$ in $\Lambda/B\Lambda$. The $K_B$-invariant subalgebra of $\bigoplus_{\lambda\in\Lambda}\mathcal L_\lambda$ is therefore $\bigoplus_{\mu\in\Lambda}\mathcal L_{B\mu}$, so
$$\widehat{\mathcal X}/K_B\cong\mathrm{Spec}_X\left(\bigoplus_{\mu\in\Lambda}\mathcal L_{B\mu}\right).$$
The relative spectrum on the right is naturally isomorphic to $X\times_{X_g}\widehat{\mathcal X}_g$. Thus $\Phi$ factors as the $K_B$-torsor $\widehat{\mathcal X}\to\widehat{\mathcal X}/K_B$, followed by the base change of the finite $K$-Galois morphism $u$. In particular, $\Phi$ is finite of degree $|K_B|\cdot |K|$.

\medskip

\noindent\textbf{Step 2.} We lift $K$ to a finite group preserving $\Phi$.

We first make the line bundles defining $\widehat{\mathcal X}$ invariant under $K$, and then choose lifts that preserve $\Phi$.

For $k\in K$ and $\mu\in\Lambda$, the class $B\mu$ on every fiber is represented by the restriction of $u^*p^*\mathcal L_\mu$. Then $u\circ k=u$  gives $k^*(B\mu)=B\mu$ on the fiberwise Picard lattice. Thus $k^*$ fixes the finite-index sublattice $B\Lambda\subseteq\Lambda$ and, since $\Lambda$ is torsion-free, acts trivially on $\Lambda$. Hence, for every $\lambda\in\Lambda$, there is a line bundle $\xi_{k,\lambda}$ on $A$ such that
$$k^*\mathcal L_\lambda\otimes\mathcal L_\lambda^{-1}\cong \mathrm{alb}_X^*\xi_{k,\lambda}.$$
Indeed, $k^*\mathcal L_\lambda\otimes\mathcal L_\lambda^{-1}$ is trivial on every rationally connected fiber. The cohomology and base change theorem, together with the evaluation map, shows that it is pulled back from $A$.

Fix $k\in K$ and $\lambda\in\Lambda$. Since $B\Lambda$ has finite index, choose $n>0$ and $\mu\in\Lambda$ with $n\lambda=B\mu$. Step~1 expresses $\mathcal L_{B\mu}$ as the pullback of a line bundle on $X_g$. It is therefore $K$-invariant, since $u\circ k=u$. Taking the $n$th tensor power gives
$$\mathrm{alb}_X^*\xi_{k,\lambda}^{\otimes n}\cong k^*\mathcal L_{B\mu}\otimes\mathcal L_{B\mu}^{-1}\cong\mathcal O_X.$$
The equality $(\mathrm{alb}_X)_*\mathcal O_X= \mathcal O_A$ and the projection formula show that $\mathrm{alb}_X^*:\mathrm{Pic}(A)\to\mathrm{Pic}(X)$ is injective. Hence $\xi_{k,\lambda}^{\otimes n}\cong\mathcal O_A$. The N\'{e}ron--Severi group of an abelian variety is torsion-free, so these torsion line bundles lie in $\mathrm{Pic}^0(A)$.

Choose a positive integer $m$ divisible by the orders of the finitely many line bundles $\xi_{k,\lambda}$ with $k\in K$ and $\lambda$ in the chosen basis of $\Lambda$. Since $[m]^*\xi\cong\xi^{\otimes m}$ for $\xi\in\mathrm{Pic}^0(A)$, pullback by $[m]$ makes all these line bundles trivial. Choose trivializations of the pullback line bundles for these pairs and extend them using tensor products and duals. Every element of $K$ then admits a graded lift to $\widehat{\mathcal X}$.

Let $\Theta$ be the group of all graded lifts of elements of $K$ to $\widehat{\mathcal X}$; equivalently, these are the $S$-equivariant lifts. An element of the kernel of $\Theta\to K$ is an $S$-equivariant automorphism of $\widehat{\mathcal X}$ over $X$, and hence is given by a morphism $X\to S$. Since $X$ is projective and connected, this morphism is constant. Consequently, there is an exact sequence
$$1\to S\to\Theta\to K\to 1.$$
This sequence is central because every element of $K$ acts trivially on $\Lambda$.

We now determine the subgroup of $\Theta$ consisting of lifts that preserve $\Phi$. For $\theta\in\Theta$, the maps $\Phi\circ\theta$ and $\Phi$ cover the same morphism $u:X\to X_g$. Since $\widehat{\mathcal X}_g\to X_g$ is an $S$-torsor, there is a unique regular map $r_\theta:\widehat{\mathcal X}\to S$ such that
$$\Phi(\theta z)= \Phi(z)r_\theta(z).$$
Both maps are $\widehat B$-equivariant, so $r_\theta$ is invariant under the source $S$-action. Zariski-local trivializations of the $S$-torsor $\widehat{\mathcal X}\to X$ therefore show that $r_\theta$ descends to a morphism $X\to S$. This morphism is constant because $X$ is projective and connected and $S$ is affine. Denote its value by $c(\theta)$. The defining equation and the centrality of $S$ show that $c:\Theta\to S$ is a homomorphism and that $c|_S= \widehat B$.

The lifts preserving $\Phi$ form the subgroup $\Gamma= \ker(c)$. Since $\widehat B$ is an isogeny, every lift of an element of $K$ can be multiplied by an element of $S$ to obtain a lift in $\Gamma$. Hence we obtain a central exact sequence
$$1\to K_B\to\Gamma\to K\to 1.$$
In particular,
$$|\Gamma|=|K_B|\cdot|K|= \deg\Phi.$$

\medskip

\noindent\textbf{Step 3.} We identify the quotient by $\Gamma$ and its algebra of relative functions.

We claim that the action of $\Gamma$ on $\widehat{\mathcal X}$ is faithful. Indeed, if $\gamma\in\Gamma$ acts trivially, then its image in $K$ acts trivially on $X$ and is therefore the identity. Hence $\gamma\in K_B$. Since $K_B\subseteq S$ acts freely on the $S$-torsor $\widehat{\mathcal X}\to X$, we have $\gamma = 1$. By construction, $\Phi$ is $\Gamma$-invariant. Thus
$$[\mathbb{C}(\widehat{\mathcal X}):\mathbb{C}(\widehat{\mathcal X})^\Gamma]=|\Gamma|= \deg\Phi.$$

Since $S$ is a split torus and $X$ is projective, the $S$-torsor $\widehat{\mathcal X}\to X$ has quasi-projective total space. Its quotient by the finite group $\Gamma$ therefore exists, and the quotient map is finite. The morphism $\Phi$ factors through a finite morphism $\overline\Phi:\widehat{\mathcal X}/\Gamma\to\widehat{\mathcal X}_g$. The equality of degrees above shows that $\overline\Phi$ is birational. Since $\widehat{\mathcal X}_g$ is smooth, it is normal, and hence $\overline\Phi$ is an isomorphism.

Since $\Gamma$ acts over $A$ and commutes with the characteristic torus $S$, it acts on the relative functions $\mathcal R$ by $\mathcal O_A$-linear automorphisms preserving the Cox grading. To compute the invariant algebra, we use the isomorphism $\overline\Phi:\widehat{\mathcal X}/\Gamma\cong\widehat{\mathcal X}_g$ over $A$ to identify the direct images of the two structure sheaves under the structure morphisms to $A$. On the quotient side, finite group invariants identify the relative functions with $\mathcal R^\Gamma$. On the target side, we push forward the line bundles defining $\widehat{\mathcal X}_g$ along $p_g:X_g\to A$. Flat base change and the projection formula give
$$\mathcal R^\Gamma\cong\bigoplus_{\mu\in\Lambda}g^*\mathcal R_\mu\otimes\mathcal L_{A,\mu}^{-1}.$$
The summand indexed by $\mu$ has degree $B\mu$, since pullback by $\Phi$ sends target degree $\mu$ to source degree $B\mu$. These identifications are compatible with base change to closed points. Indeed, the Reynold operator
$$e_\Gamma:= \frac{1}{|\Gamma|}\sum_{\gamma\in\Gamma}\gamma:\mathcal R\to\mathcal R$$
is an $\mathcal O_A$-linear idempotent with image $\mathcal R^\Gamma$. Thus $\mathcal R= \mathcal R^\Gamma\oplus\ker(e_\Gamma)$ as $\mathcal O_A$-modules. This decomposition is preserved by base change, and the pullback operator is the Reynolds operator on the pullback algebra. Consequently, for every closed point $y\to A$, the natural map
$$(\mathcal R^\Gamma)_y\to(\mathcal R_y)^\Gamma$$
is an isomorphism. On the target side, compatibility with base change follows from that of the graded pieces $\mathcal R_\mu$ and pullback of the line bundles $\mathcal L_{A,\mu}$. This proves \eqref{eq: invariant-cox-sheaf}.
\end{proof}

\medskip

We now study the action of $\Gamma= \mathrm{Gal}(\Phi)$ (constructed in Lemma~\ref{lem: relative-theta}) on the Cox generators of each fiber. Recall that $\mathfrak m\subseteq\mathcal R$ is the ideal generated by the pieces of nonzero Cox degree, and that $\mathcal V= \mathfrak m/\mathfrak m^2$. For a closed point $y\to A$, the algebra $\mathcal R_y$ is the Cox ring of $X_y$, and a homogeneous basis of $\mathcal V_y$ lifts to a minimal set of homogeneous algebra generators of $\mathcal R_y$. We write $\mathcal V_{y,\nu}$ for the homogeneous piece of $\mathcal V_y$ of Cox degree $\nu\in\Lambda= \mathrm{Pic}(F)$.

\begin{lemma}\label{lem: relative-reflection}
For every closed point $y\to A$, there is a $\Gamma$-equivariant $\Lambda$-graded isomorphism
$$\mathrm{Sym}(\mathcal V_y)\cong\mathcal R_y.$$
The action of $\Gamma$ on $\mathcal V_y$ is faithful and is generated by pseudoreflections (Definition~\ref{def: pseudoreflection}). For each Cox degree $\nu$ occurring in $\mathcal V_y$, let $\Gamma_\nu$ be the subgroup of $\Gamma$ generated by the pseudoreflections $\gamma\in\Gamma$ whose normal lines $\mathrm{Im}(\gamma-\mathrm{id})$ lie in $\mathcal V_{y,\nu}$. Then multiplication induces an isomorphism
$$\prod_\nu\Gamma_\nu\cong\Gamma,$$
and the invariant ring decomposes as
$$\mathrm{Sym}(\mathcal V_y)^\Gamma= \bigotimes_\nu\mathrm{Sym}(\mathcal V_{y,\nu})^{\Gamma_\nu}.$$
Here the product and tensor product range over the Cox degrees occurring in $\mathcal V_y$. Consequently, the basic invariants may be chosen block by block. A basic invariant of polynomial degree $d$ in the block of Cox degree $\nu$ has Cox degree $d\nu$.
\end{lemma}

\begin{proof}
Fix a closed point $y\to A$. We first check that the action on the fiber remains faithful. Since finite group invariants commute with base change in characteristic zero, the morphism $\Phi_y:\widehat{\mathcal X}_y\to(\widehat{\mathcal X}_g)_y$ is the quotient by $\Gamma$. Here $(\widehat{\mathcal X}_g)_y$ denotes the fiber over $y$ of the composite $\widehat{\mathcal X}_g\to X_g\to A$.

The morphism $\Phi:\widehat{\mathcal X}\to\widehat{\mathcal X}_g$ is finite between smooth schemes of the same dimension and hence is flat by miracle flatness. Its base change $\Phi_y$ over $y$ is therefore finite flat of degree $|\Gamma|$. If a nontrivial subgroup of $\Gamma$ acted trivially on $\widehat{\mathcal X}_y$, then $\Phi_y$ would have generic degree smaller than $|\Gamma|$. Hence the action on every fiberwise characteristic space is faithful.

For each Cox degree $\nu\in\Lambda$, averaging a linear section over the finite group $\Gamma$ gives a $\Gamma$-equivariant splitting of
$$0\to(\mathfrak m_y^2)_\nu\to(\mathfrak m_y)_\nu\to\mathcal V_{y,\nu}\to 0.$$
The function $\deg_B$ defined in Definition~\ref{def: relative degree} is positive on all nonzero Cox degrees occurring in $\mathcal R_y$. Thus products involve factors of smaller degree with respect to $\deg_B$, and induction shows that the chosen lifts generate $\mathcal R_y$. Multiplication gives a $\Gamma$-equivariant $\Lambda$-graded surjection $\mathrm{Sym}(\mathcal V_y)\to\mathcal R_y$. Since $X_y$ is toric, its Cox ring $\mathcal R_y$ is a polynomial ring with $\dim\mathcal V_y$ variables. The two rings therefore have the same dimension, so the kernel is a height-zero prime and hence is zero. This isomorphism also proves that the action on $\mathcal V_y$ is faithful: an element acting trivially on $\mathcal V_y$ acts trivially on $\mathcal R_y$ and hence on the characteristic space $\widehat{\mathcal X}_y$, which is a $\Gamma$-invariant open subset of $\mathrm{Spec}\mathcal R_y$.

To prove the reflection assertion, we first show that $\mathcal R_y^\Gamma$ is a polynomial ring. Lemma~\ref{lem: relative-theta} gives an isomorphism of algebras
$$\mathcal R_y^\Gamma\cong
\bigoplus_{\mu\in\Lambda}(\mathcal R_\mu)_{g(y)}
\otimes(\mathcal L_{A,\mu}^{-1})_y.$$
Choose a nonzero vector in $(\mathcal L_{A,\mu}^{-1})_y$ for each $\mu$ in a basis of $\Lambda$, and extend these choices using tensor products and duals. These choices identify the algebra on the right with $\mathcal R_{g(y)}$, which is a polynomial ring. Hence $\mathcal R_y^\Gamma$ is also a polynomial ring.

Thus $\mathrm{Sym}(\mathcal V_y)^\Gamma$ is a polynomial ring. The polynomial form of Theorem~\ref{thm: reflection-invariants}, applied over the algebraically closed residue field of $y$, shows that $\Gamma$ is generated by pseudoreflections on $\mathcal V_y$.

The characteristic torus commutes with $\Gamma$, so the Cox degree decomposition is preserved. Lemma~\ref{lem: reflection-blocks} gives the product decomposition of $\Gamma$ and the tensor-product decomposition of $\mathrm{Sym}(\mathcal V_y)^\Gamma$. Applying Theorem~\ref{thm: reflection-invariants} to each block gives basic invariants for that block. Their Cox degrees have the asserted form because every variable in a block has the same Cox degree.
\end{proof}

\medskip

The preceding lemma gives pseudoreflections on each fiber. To obtain line bundles from their normal lines, we need to use the same group elements throughout the family. The following lemma makes this possible.

\begin{lemma}\label{lem: uniform-reflections}
Fix a closed point $y_0\to A$. There is a generating set of $\Gamma$ consisting of pseudoreflections on $\mathcal V_{y_0}$ such that every chosen element remains a pseudoreflection on every fiber, with its normal line in the same Cox degree.
\end{lemma}

\begin{proof}
Choose a generating set of pseudoreflections on $\mathcal V_{y_0}$, which exists by Lemma~\ref{lem: relative-reflection}. Let $\gamma$ be one of these elements and let $n$ be its order. For every $n$-th root of unity $\xi$, define the idempotent
$$e_\xi= \frac{1}{n}\sum_{j=0}^{n-1}\xi^{-j}\gamma^j.$$
Its image is the $\xi$-eigenspace of $\gamma$. It preserves every homogeneous piece of $\mathcal V$, and its image is a direct summand. The rank of this image on each homogeneous piece is therefore locally constant. Since $A$ is connected, all eigenvalue multiplicities agree with those on $\mathcal V_{y_0}$. On that fiber, there is exactly one eigenvalue different from $1$, and its eigenspace is one-dimensional. Thus $\gamma$ remains a pseudoreflection on every fiber. The same rank argument on the homogeneous pieces shows that its normal line has the same Cox degree on every fiber.
\end{proof}

\medskip

\subsection{Cox generators and toric reconstruction}

Recall that $\mathcal R= \bigoplus_{\lambda\in\Lambda}\mathcal R_\lambda$ is the relative Cox algebra, graded by $\Lambda= \mathrm{Pic}(F)$, and that the finite group $\Gamma$ constructed in Lemma~\ref{lem: relative-theta} acts by graded $\mathcal O_A$-algebra automorphisms. The ideal $\mathfrak m$ is generated by the pieces of nonzero Cox degree, and $\mathcal V= \mathfrak m/\mathfrak m^2$ records homogeneous generators modulo products. We want to find homogeneous line subbundles of $\mathfrak m$ whose images decompose $\mathcal V$ and which freely generate $\mathcal R$. Pseudoreflections will give lifts of the nontrivial isotypic components. To lift the fixed part, we first compare the generators of $\mathcal R^\Gamma$ with those of $\mathcal R$.

In this comparison, $g:A\to A$ is the isogeny induced by $f$, and $B:\Lambda\to\Lambda$ represents pullback by the fiber maps $X_y\to X_{g(y)}$. The line bundles $\mathcal L_{A,\mu}$ on $A$ are the base factors in the multiplicatively compatible isomorphisms
$$f^*\mathcal L_\mu\cong\mathcal L_{B\mu}\otimes\mathrm{alb}_X^*\mathcal L_{A,\mu}.$$
Here $\mathcal L_\mu$ represents $\mu$ on each fiber and $\mathcal R_\mu=(\mathrm{alb}_X)_*\mathcal L_\mu$.

\begin{lemma}\label{lem: invariant-indecomposables}
Let $\mathfrak m^\Gamma$ be the ideal of invariant sections in $\mathfrak m$, equivalently the ideal of nonzero Cox degrees in $\mathcal R^\Gamma$. The indecomposable quotient of $\mathcal R^\Gamma$ is $\mathfrak m^\Gamma/(\mathfrak m^\Gamma)^2$. For every $\mu\in\Lambda$, there is an isomorphism
\begin{equation}\label{eq: invariant-indecomposables}
\bigl(\mathfrak m^\Gamma/(\mathfrak m^\Gamma)^2\bigr)_{B\mu}\cong g^*\mathcal V_\mu\otimes\mathcal L_{A,\mu}^{-1}.
\end{equation}

The natural map $\mathfrak m^\Gamma/(\mathfrak m^\Gamma)^2\to\mathcal V$ induced by $\mathcal R^\Gamma\subseteq\mathcal R$ has image $\mathcal V^\Gamma$. Both assertions are compatible with base change to closed points.
\end{lemma}

\begin{proof}
By Lemma~\ref{lem: relative-theta}, the graded algebra $\mathcal R^\Gamma$ is identified with $\bigoplus_{\mu\in\Lambda}g^*\mathcal R_\mu\otimes\mathcal L_{A,\mu}^{-1}$, where the summand indexed by $\mu$ has degree $B\mu$. Under this identification, the degree $B\mu$ part of the homogeneous ideal $\mathfrak m^\Gamma$ is $g^*\mathfrak m_{\mu}\otimes\mathcal L_{A,\mu}^{-1}$. The corresponding part of its square is $g^*(\mathfrak m^2)_\mu\otimes\mathcal L_{A,\mu}^{-1}$. Here we use the flatness of $g$ and the multiplicativity of the $\mathcal L_{A,\mu}$. Taking the quotient proves \eqref{eq: invariant-indecomposables}.

Since $(\mathfrak m^\Gamma)^2\subseteq\mathfrak m^2$, the inclusion $\mathfrak m^\Gamma\subseteq\mathfrak m$ induces a map $\mathfrak m^\Gamma/(\mathfrak m^\Gamma)^2\to\mathcal V$. Taking invariants in $0\to\mathfrak m^2\to\mathfrak m\to\mathcal V\to 0$ is exact by averaging over $\Gamma$, so $\mathfrak m^\Gamma\to\mathcal V^\Gamma$ is surjective. Thus the induced map has image $\mathcal V^\Gamma$. Local freeness and Reynolds projectors show that these constructions are compatible with base change to closed points.
\end{proof}

On a fiber, the map $\bigl(\mathfrak m^\Gamma/(\mathfrak m^\Gamma)^2\bigr)_y\to(\mathcal V_y)^\Gamma$ need not be injective: in the polynomial presentation $\mathcal R_y\cong\mathrm{Sym}(\mathcal V_y)$, a basic invariant of polynomial degree at least two has zero image in $\mathcal V_y$. By the blockwise description in Lemma~\ref{lem: relative-reflection}, a basic invariant of polynomial degree $d$ in a block of Cox degree $\nu$ has Cox degree $d\nu$. Recall that a nonzero element of $\Lambda$ is \emph{primitive} if it is not divisible by any integer greater than one. The following primitivity property will exclude these nonlinear invariants in the degrees occurring in $\mathcal V$.

\begin{lemma}\label{lem: primitive-cox-degree}
Every Cox degree occurring in $\mathcal V$ is primitive in $\Lambda$.
\end{lemma}

\begin{proof}
The degrees occurring in $\mathcal V$ are the classes of the torus-invariant prime divisors of $F$. Choose such a divisor $D_\rho$ with $[D_\rho]= \lambda$, where $\rho$ is a ray of the fan of $F$. Choose a maximal cone $\sigma$ containing $\rho$, and let $\tau$ be the facet of $\sigma$ generated by its remaining rays. The invariant curve $V(\tau)$ meets $D_\rho$ transversely at the torus fixed point corresponding to $\sigma$. Smoothness gives
$$D_\rho\cdot V(\tau)= 1.$$
If $[D_\rho]$ were divisible by an integer $d>1$ in $\mathrm{Pic}(F)$, then the intersection number $D_\rho\cdot V(\tau)$ would be divisible by $d$, a contradiction.
\end{proof}

\medskip

We now construct the generating line bundles. The preceding lemmas will identify the fixed generator classes with pullbacks of classes of smaller degree with respect to the grading in Definition~\ref{def: relative degree}.

\begin{proposition}\label{prop: relative-generators}
There is a homogeneous vector subbundle $\mathcal M\subseteq\mathfrak m$ which is a direct sum of line bundles and whose natural projection $\mathcal M\to\mathcal V= \mathfrak m/\mathfrak m^2$ is an isomorphism. Moreover, there is an isomorphism of $\Lambda$-graded algebras
\begin{equation}\label{eq: relative polynomial cox}
\mathcal R\cong\mathrm{Sym}_{\mathcal O_A}(\mathcal M).
\end{equation}
\end{proposition}

\begin{proof}
Only finitely many Cox degrees occur in $\mathcal V$. Fix a closed point $y_0\to A$ and a generating set of $\Gamma$ as in Lemma~\ref{lem: uniform-reflections}. Then each chosen element is a pseudoreflection on every fiber, with its normal line in a fixed Cox degree. Recall that the function $\deg_B$ defined in Definition~\ref{def: relative degree} is positive on nonzero effective classes and satisfies $\deg_B(B\mu) > \deg_B(\mu)$ for every such class $\mu$. We proceed by induction on $\deg_B(\lambda)$, keeping the homogeneous pieces indexed by their Cox degrees. Fix a Cox degree $\lambda$ at the current stage and write
$$\mathcal V_\lambda= \mathcal E_\lambda\oplus\mathcal W_\lambda,\qquad \mathcal E_\lambda= \mathcal V_\lambda^\Gamma.$$
Here $\mathcal W_\lambda$ is the kernel of the Reynolds projector $|\Gamma|^{-1}\sum_{\gamma\in\Gamma}\gamma$, hence the sum of the nontrivial isotypic components.

\noindent\textbf{Step 1.} We lift $\mathcal W_\lambda$ to a direct sum of line bundles in $\mathcal R_\lambda$.

We first examine the multiplicities of the nontrivial representations. For every nontrivial irreducible $\Gamma$-module $W$, put $\mathcal A_W= \mathcal{H}\!om_\Gamma(W\otimes \mathcal O_A,\mathcal V_\lambda)$. Its rank is the multiplicity of $W$ in a fiber of $\mathcal V_\lambda$. Reynolds projectors show that each $\mathcal A_W$ is locally free and that the evaluation maps give the isotypic decomposition
$$\mathcal W_\lambda= \bigoplus_W\mathcal A_W\otimes W.$$
Suppose that $\mathcal A_W\neq 0$. Since the chosen pseudoreflections generate $\Gamma$, some chosen pseudoreflection $\gamma$ acts nontrivially on $W$. By Lemma~\ref{lem: uniform-reflections}, on every fiber $y$ we have
$$1= \mathrm{rank}((\gamma-\mathrm{id})|_{\mathcal V_y})\geq\mathrm{rank}(\mathcal A_W)\,\mathrm{rank}((\gamma-\mathrm{id})|_W) > 0.$$
Both positive factors on the right are one. Thus every nonzero $\mathcal A_W$ is a line bundle, and $\gamma$ acts trivially on all other isotypic summands of $\mathcal V$.

Next, we lift the normal lines of the reflections from $\mathcal V_\lambda$ to $\mathcal R_\lambda$. For a pseudoreflection $\gamma$ in the fixed generating set whose normal line has Cox degree $\lambda$, let $\zeta\neq 1$ be its nontrivial eigenvalue. Let $\mathcal H_\gamma\subseteq\mathcal R_\lambda$ be the image of the projector onto the $\zeta$-eigenspace. This idempotent makes $\mathcal H_\gamma$ a direct summand and commutes with base change to closed points.

On a fiber over $y$, the equivariant polynomial presentation in Lemma~\ref{lem: relative-reflection} shows that $\gamma$ multiplies one variable $x_y$ of Cox degree $\lambda$ by $\zeta$ and fixes all other variables. A monomial on which $\gamma$ acts by multiplication by $\zeta$ must contain $x_y$. If it is nonlinear, its degree with respect to $\deg_B$ is strictly larger than $\deg_B(\lambda)$ because $\deg_B$ is positive on the Cox degrees of all variables. It therefore cannot have Cox degree $\lambda$. Thus the $\zeta$-eigenspace in $(\mathcal R_y)_\lambda$ is one-dimensional, spanned by $x_y$.
It follows that $\mathcal H_\gamma$ is a line bundle and maps isomorphically onto the normal eigenline in $\mathcal V_\lambda$ on every fiber.

We now lift an entire isotypic summand. Fix a nontrivial irreducible $\Gamma$-module $W$ with $\mathcal A_W\neq 0$. For every pseudoreflection $\gamma$ in the fixed generating set that acts nontrivially on $W$, choose a nonzero vector $w_\gamma\in W$ spanning $\mathrm{im}((\gamma-\mathrm{id})|_W)$. The corresponding normal line subbundle of $\mathcal V_\lambda$ is $\mathcal A_W\otimes \mathbb{C}w_\gamma$.

Choose a $\Gamma$-invariant Hermitian form on $W$. If the normal vectors $w_\gamma$ did not span $W$, their orthogonal complement would be fixed by every generating pseudoreflection: those acting nontrivially on $W$ fix this complement, and the others fix all of $W$. The nonzero complement would therefore lie in $W^\Gamma$, contradicting $W^\Gamma = 0$.

Among these pseudoreflections, choose $\gamma_1,\ldots,\gamma_{\dim W}$ such that $w_{\gamma_1},\ldots,w_{\gamma_{\dim W}}$ form a basis of $W$. Then the direct sum of the corresponding line bundles $\mathcal H_{\gamma_j}$ maps isomorphically onto $\mathcal A_W\otimes_{\mathbb{C}}W$. Repeating this for every nontrivial isotypic summand gives a direct sum of line bundles in $\mathcal R_\lambda$ mapping isomorphically onto $\mathcal W_\lambda$.

\medskip

\noindent\textbf{Step 2.} We lift the invariant subbundle $\mathcal E_\lambda$ by induction on $\deg_B(\lambda)$.

If $\mathcal E_\lambda=0$, we are done. We may assume that $\mathcal E_\lambda\neq 0$. Recall that the central subgroup $K_B\subseteq\Gamma$ is $\mathrm{Hom}(\Lambda/B\Lambda,\mathbb C^*)$. It acts on $\mathcal V_\lambda$ through the inverse of the character given by the class of $\lambda$ in $\Lambda/B\Lambda$. Since $\mathcal E_\lambda$ consists of fixed vectors, this character is trivial. Hence $\lambda\in B\Lambda$. Because $B$ is injective, there is a unique $\mu\in\Lambda$ such that $\lambda = B\mu$.

Recall that the quotient $\mathfrak m^\Gamma/(\mathfrak m^\Gamma)^2$ records the generator classes of the invariant algebra. The degree $\lambda$ part of the natural map $\mathfrak m^\Gamma/(\mathfrak m^\Gamma)^2\to\mathcal V$ has image $\mathcal E_\lambda$, so $\bigl(\mathfrak m^\Gamma/(\mathfrak m^\Gamma)^2\bigr)_\lambda\neq 0$. Equation~\eqref{eq: invariant-indecomposables} then gives $\mathcal V_\mu\neq 0$. Thus $\mu$ is a nonzero effective degree occurring in $\mathcal V$, and \eqref{eq: relative-degree-increase} gives
$$0 < \deg_B(\mu) < \deg_B(B\mu) = \deg_B(\lambda).$$

To recover all of $\mathcal E_\lambda$ from the smaller degree $\mu$, we show that the surjection $\bigl(\mathfrak m^\Gamma/(\mathfrak m^\Gamma)^2\bigr)_\lambda\to\mathcal E_\lambda$ is an isomorphism. On a fiber, Lemma~\ref{lem: relative-reflection} identifies $\mathcal R_y$ with $\mathrm{Sym}(\mathcal V_y)$ and provides blockwise basic invariants. Their classes form a basis of the indecomposable quotient of $\mathcal R_y^\Gamma$.

A basic invariant of polynomial degree one maps to its class in $\mathcal V_y^\Gamma$. One of polynomial degree $d\geq2$ lies in $\mathfrak m_y^2$ and therefore maps to zero in $\mathcal V_y$. If its block has Cox degree $\nu$, its Cox degree is $d\nu$. Consequently, on each fiber, the kernel of $\bigl(\mathfrak m^\Gamma/(\mathfrak m^\Gamma)^2\bigr)_\lambda\to\mathcal E_\lambda$ is spanned by the classes of basic invariants satisfying $\lambda=d\nu$ with $d\geq 2$.

Since $\mathcal E_\lambda\neq 0$, we have $\mathcal V_\lambda\neq 0$. Lemma~\ref{lem: primitive-cox-degree} says that $\lambda$ is primitive, so an equality $\lambda=d\nu$ with $d\geq 2$ is impossible. Thus $\bigl(\mathfrak m^\Gamma/(\mathfrak m^\Gamma)^2\bigr)_\lambda\to\mathcal E_\lambda$ is an isomorphism on every fiber. Both sheaves are locally free: this follows for $\mathcal E_\lambda$ from the Reynolds projector and for $\bigl(\mathfrak m^\Gamma/(\mathfrak m^\Gamma)^2\bigr)_\lambda$ from \eqref{eq: invariant-indecomposables}. The map commutes with base change to closed points, so it is an isomorphism of vector bundles. Combining it with \eqref{eq: invariant-indecomposables} gives
$$g^*\mathcal V_\mu\otimes\mathcal L_{A,\mu}^{-1}\cong\mathcal E_{B\mu}.$$

If $\deg_B(\lambda)$ is minimal among the values $\deg_B(\nu)$ with $\mathcal V_\nu\neq 0$, then $\mathcal E_\lambda$ must be zero. Otherwise, the preceding argument would produce a class $\mu$ with $\mathcal V_\mu\neq 0$ and $\deg_B(\mu)<\deg_B(\lambda)$, a contradiction. The construction in Step~1 therefore gives the base case of the induction.

At a later stage of the induction, suppose that $\mathcal E_\lambda\neq 0$. Since $\deg_B(\mu)<\deg_B(\lambda)$, a direct sum of line bundles $\mathcal M_\mu\subseteq\mathcal R_\mu$ has already been constructed and maps isomorphically to $\mathcal V_\mu$.

The invariant algebra isomorphism \eqref{eq: invariant-cox-sheaf} gives the inclusions
$$g^*\mathcal M_\mu\otimes\mathcal L_{A,\mu}^{-1}\hookrightarrow g^*\mathcal R_\mu\otimes\mathcal L_{A,\mu}^{-1}\cong(\mathcal R^\Gamma)_{B\mu}\hookrightarrow\mathcal R_{B\mu}.$$
By the identification $g^*\mathcal V_\mu\otimes\mathcal L_{A,\mu}^{-1}\cong\mathcal E_{B\mu}$, the image maps isomorphically onto $\mathcal E_{B\mu}$. It is a direct sum of line bundles and gives the required lift of the fixed part.

\medskip

\noindent\textbf{Step 3.} We define $\mathcal M$ as the direct sum
of these lifts and prove that it freely generates $\mathcal R$.

Taking the direct sum over all degrees gives a direct sum of line bundles $\mathcal M\subseteq\mathfrak m$ which maps isomorphically to $\mathcal V$. In fact, the image of any homogeneous splitting of $\mathfrak m\to\mathcal V$ freely generates $\mathcal R$. To prove this, consider the graded multiplication homomorphism
$$\mathrm{Sym}_{\mathcal O_A}(\mathcal M)\to\mathcal R.$$
The map is surjective by induction on $\deg_B(\lambda)$ for sections of Cox degree $\lambda$, with $\deg_B$ as in Definition~\ref{def: relative degree}. Indeed, subtracting the chosen lift of the indecomposable class of a homogeneous local section leaves an element of $\mathfrak m^2$. This remainder is a sum of products, and every factor has Cox degree $\nu$ with $0 < \deg_B(\nu) < \deg_B(\lambda)$.

To check that this surjection is an isomorphism globally, fix $d\geq 0$ and consider the sum of the Cox graded pieces with $\deg_B(\lambda) = d$ on each side. Since there are finitely many Cox variable degrees and $\deg_B$ is positive on each of them, only finitely many monomials contribute to these pieces. Surjectivity shows that only the Cox degrees of these monomials can occur in $\bigoplus_{\deg_B(\lambda)=d}\mathcal R_\lambda$. Thus the pieces on both sides are finite direct sums of vector bundles.

On every fiber, the algebra map is a surjection between polynomial rings of the same dimension. Its kernel is a prime ideal of height zero and is therefore zero. Hence, for each $d$, the induced map on the pieces of degree $d$ with respect to $\deg_B$ is an isomorphism on every fiber. Since these pieces are vector bundles, this map is an isomorphism over $A$. Taking the direct sum over $d$ proves \eqref{eq: relative polynomial cox}.
\end{proof}

We now recover the toric fibration from these line bundles. Local frames give diagonal changes of Cox coordinates, which descend to automorphisms of $F$ induced by its dense torus.

\begin{proposition}\label{prop: cox-reconstruction}
There are a principal $T$-bundle $P\to A$ and an isomorphism over $A$
$$X\cong P\times^T F,$$
where $T$ is the dense torus of the fixed toric fiber $F$.
\end{proposition}

\begin{proof}
Write $\mathcal M= \bigoplus_{i=1}^e\mathcal M_i$, where $e= \mathrm{rank}(\mathcal V)$ is the number of Cox variables of $F$, and let $\lambda_i\in\Lambda= \mathrm{Pic}(F)$ be the Cox degree of $\mathcal M_i$. Put $T_0=(\mathbf{G}_m)^e$, the torus acting by independent rescaling of the coordinates. Let $M_F$ be the character lattice of the dense torus of $F$, and index the torus invariant prime divisors of $F$ so that their classes are $\lambda_1,\ldots,\lambda_e$. The standard toric divisor sequence is
$$0\to M_F\to\mathbb{Z}^e\to\Lambda \to 0.$$
The map $\mathbb{Z}^e\to\Lambda$ sends the $i$-th standard basis vector to $\lambda_i$. Dually,
$$1\to S\to T_0\to T\to 1.$$
Here $S= \mathrm{Hom}(\Lambda,\mathbf{G}_m)$ is the characteristic torus and $T= \mathrm{Hom}(M_F,\mathbf{G}_m)$ is the dense torus of $F$. This sequence explains how diagonal changes of Cox coordinates descend, modulo $S$, to automorphisms of $F$ induced by elements of $T$.

Let $\mathcal P_0\to A$ be the product of the frame bundles of the line bundles $\mathcal M_i$, with its standard right $T_0$-action. Let $T_0$ act on $\mathbf{A}^e$ by the inverse coordinatewise action
$$(t_1,\ldots,t_e)\cdot(z_1,\ldots,z_e)=(t_1^{-1}z_1,\ldots,t_e^{-1}z_e).$$
The inverse action occurs because $\mathrm{Spec}_A\mathrm{Sym}_{\mathcal O_A}(\mathcal M)$ is the total space of $\mathcal M^\vee$. With this action, Equation~\eqref{eq: relative polynomial cox} gives
$$\mathrm{Spec}_A\mathcal R\cong \mathcal P_0\times^{T_0}\mathbf{A}^e.$$

We must also recover the characteristic torsor $\widehat{\mathcal X}\to X$ inside this affine bundle. Recall that $U_F\subseteq\mathrm{Spec}\mathrm{Cox}(F)\cong \mathbf A^e$ is the Cox characteristic space of $F$, and that $H_F$ is the ample class induced by the restrictions of $H$ to the fibers. The open subset $U_F$ is the semistable locus for $H_F$:
$$U_F= \bigcup_{\substack{m > 0\\ s\in\mathrm{Cox}(F)_{mH_F}}}D(s)\subseteq\mathbf{A}^e.$$
Every $\Lambda$-graded automorphism of the Cox ring preserves each vector space $\mathrm{Cox}(F)_{m H_F}$ and therefore preserves $U_F$. The relative Cox construction gives a canonical open immersion
$$\widehat{\mathcal X}\hookrightarrow\mathrm{Spec}_A\mathcal R.$$
Choose an \'{e}tale cover giving the compatible local trivializations in Proposition~\ref{prop: relative-cox-construction} and trivializing the line bundles $\mathcal M_i$. The Cox trivialization identifies this open immersion with the product of $U_F\hookrightarrow\mathbf{A}^e$ and the base. The chosen frames of the $\mathcal M_i$, together with \eqref{eq: relative polynomial cox}, give a second trivialization of $\mathcal R$ as a $\Lambda$-graded polynomial algebra. The two trivializations differ by a $\Lambda$-graded automorphism, which preserves $U_F$. Thus the second trivialization also identifies $\widehat{\mathcal X}$ with the product family with fiber $U_F$.

On overlaps, the changes of frames lie in $T_0$ and preserve $U_F$. The local descriptions therefore glue to give
$$\widehat{\mathcal X}\cong\mathcal P_0\times^{T_0}U_F.$$

Replacing $s$ by $s^{-1}$ in the standard Cox action does not change its orbits. Thus the inverse $S$-action on $U_F$ still has geometric quotient $U_F/S=F$ \cite[Theorem~2.1]{Cox95}. The induced action of $T=T_0/S$ on $F$ is given by $(t,z) \mapsto t^{-1}\cdot z$, where $\cdot$ is the usual torus action. We will use this action below. Since $\widehat{\mathcal X}\to X$ is an $S$-torsor, taking the quotient by $S$ gives
$$X\cong\mathcal P_0\times^{T_0}(U_F/S).$$
Hence $X\cong\mathcal P_0\times^{T_0}F$. Set $P:= \mathcal P_0\times^{T_0}T= \mathcal P_0/S$. The exact torus sequence shows that $P\to A$ is a principal $T$-bundle, and
$$\mathcal P_0\times^{T_0}F\cong P\times^T F.$$
This completes the proof.
\end{proof}

\medskip

\subsection{Equivariant generators and descent}

Recall that the relative Cox algebra is graded by $\Lambda= \mathrm{Pic}(F)$, with pieces $\mathcal R_\lambda=(\mathrm{alb}_X)_*\mathcal L_\lambda$. Proposition~\ref{prop: relative-generators} gives homogeneous line bundles $\mathcal M_i\subseteq\mathcal R_{\lambda_i}$ that freely generate this algebra. Their evaluation maps define the coordinate divisors in the toric presentation of Proposition~\ref{prop: cox-reconstruction}. To descend the union of these divisors through a finite quotient, we will choose the generating lines so that the group permutes them.

We first describe how an automorphism acts on these lines. Let $\delta_X: X\to X$ be an automorphism inducing an automorphism $\delta_A:A\to A$, and let $\rho_\delta:\Lambda\to\Lambda$ be its pullback on the fiberwise Picard lattice. For each $\lambda\in\Lambda$, the line bundle $\delta_X^*\mathcal L_\lambda\otimes\mathcal L_{\rho_\delta(\lambda)}^{-1}$ is trivial on every fiber. As in Proposition~\ref{prop: relative-cox-pullback}, it is therefore pulled back from a line bundle $\mathcal L_{A,\delta,\lambda}$ on $A$. Choosing the resulting isomorphisms on a basis of $\Lambda$ and extending by tensor products and duals gives
$$\delta_X^*\mathcal L_\lambda\cong\mathcal L_{\rho_\delta(\lambda)}\otimes\mathrm{alb}_X^*\mathcal L_{A,\delta,\lambda}.$$
Taking direct images gives multiplicatively compatible isomorphisms
\begin{equation}\label{eq: semigraded-action}
\phi_{\delta,\lambda}:
\delta_A^*\mathcal R_\lambda
\xrightarrow{\sim}
\mathcal R_{\rho_\delta(\lambda)}
\otimes\mathcal L_{A,\delta,\lambda}.
\end{equation}
For a group of such automorphisms, we say that the isomorphisms \eqref{eq: semigraded-action} permute homogeneous line subbundles $\mathcal M_i\subseteq\mathcal R_{\lambda_i}$ if, for every group element $\delta$ and every $i$, the isomorphism in degree $\lambda_i$ sends $\delta_A^*\mathcal M_i$ onto $\mathcal M_j\otimes\mathcal L_{A,\delta,\lambda_i}$ for some $j$ with $\lambda_j= \rho_\delta(\lambda_i)$. This condition includes both the change of Cox degree and the line bundle pulled back from the base.

In the following, we also write $\delta$ for $\delta_X$ when no confusion is possible.

\begin{proposition}\label{prop: equivariant relative generators}
Let a finite group $G$ act on $X$ and $A$ so that $\mathrm{alb}_X:X\to A$ is $G$-equivariant. Assume the following.
\begin{enumerate}[label=\textup{(\arabic*)},leftmargin=2.5em]
\item $G$ centralizes the subgroup $K\subseteq\mathrm{Gal}(f)$ preserving every fiber of $\mathrm{alb}_X$.
\item $G$ preserves the ample class $H_F$ induced by $H$ on the fibers.
\item There is a homomorphism $\beta:G\to G$ satisfying
$$f\circ\delta= \beta(\delta)\circ f \ \ \ \text{for}\ \ \ \delta\in G.$$
\end{enumerate}
Then the bundle $\mathcal M$ in Proposition~\ref{prop: relative-generators} may be chosen so that the isomorphisms \eqref{eq: semigraded-action} induced by $G$ permute its homogeneous line summands.
\end{proposition}

\begin{proof}
The isomorphisms \eqref{eq: semigraded-action} preserve the ideal $\mathfrak m$ of nonzero Cox degrees and induce isomorphisms on the bundle $\mathcal V= \mathfrak m/\mathfrak m^2$ of generator classes modulo products. We first check their compatibility with the Galois action and with pullback by $f$. We then decompose $\mathcal V$ into lines permuted by the isomorphisms induced by $G$, and lift those lines to generators by averaging a homogeneous splitting of $\mathfrak m\to\mathcal V$.

\noindent\textbf{Step 1.} We check compatibility of the $G$-action with the grading, the reflection action, and the invariant algebra.

We first show that $\deg_B$ is preserved by $G$. On the base change $X_g = X\times_{A,g}A$, where $g:A\to A$ is the isogeny induced by $f$, the element $\delta$ induces
$$\delta_g(x,b)=(\beta(\delta)(x),\delta_A(b)),$$
because $g\circ\delta_A= \beta(\delta)_A\circ g$. The relative quotient $u=(f,\mathrm{alb}_X):X\to X_g$ satisfies $u\circ\delta= \delta_g\circ u$. Its fiberwise Picard pullback $B:\Lambda\to\Lambda$ therefore satisfies
\begin{equation}\label{eq: equivariant-degree-map}
\rho_\delta\circ B = B\circ\rho_{\beta(\delta)}.
\end{equation}
It follows that $B^{-j}\rho_\delta= \rho_{\beta^j(\delta)}B^{-j}$ for every $j\geq1$. Recall from Definition~\ref{def: relative degree} that
$$\deg_B(\lambda)=N\sum_{j\geq 1}(B^{-j}\lambda\cdot H_F^{\dim F-1}),$$
where $N$ is the fixed positive integer clearing the denominators of $(B-\mathrm{id})^{-1}$. Since every element of $G$ preserves $H_F$, replacing $\lambda$ by $\rho_\delta(\lambda)$ leaves each intersection number in this sum unchanged. Hence
$$\deg_B(\rho_\delta(\lambda))= \deg_B(\lambda).$$
Thus the grading defined by $\deg_B$ is $G$-invariant. This argument only uses positive iterates of $\beta$ and does not require $\beta$ to be invertible.

We next compare the $G$-action with the reflection group $\Gamma$. Recall that $\Phi:\widehat{\mathcal X}\to\widehat{\mathcal X}_g$ is the $\Gamma$-Galois lift of $u:X\to X_g$ from Lemma~\ref{lem: relative-theta}, and that $S= \mathrm{Hom}(\Lambda,\mathbf G_m)$ is the characteristic torus. After an \'{e}tale localization $U\to A$ trivializing the line bundles on the base that occur in the pullback formulas, choose lifts $\widetilde\delta$ and $\widetilde{\delta_g}$ to these $S$-torsors. These lifts respect the Cox grading with degree maps $\rho_\delta$ and $\rho_{\beta(\delta)}$, respectively. Equation~\eqref{eq: equivariant-degree-map} shows that $\Phi\circ\widetilde\delta$ and $\widetilde{\delta_g}\circ\Phi$ are equivariant with respect to the same homomorphism of tori. They also cover the same morphism. Write $a:\widehat{\mathcal X}_U\to U$ and $b:\widehat{\mathcal X}_{g,U}\to U$ for the natural projections, so $b\circ\Phi=a$. Their ratio is the unique morphism $r:\widehat{\mathcal X}_U\to S$ satisfying
$$(\Phi\circ\widetilde\delta)(p)=r(p)\cdot(\widetilde{\delta_g}\circ\Phi)(p).$$
The common equivariance gives $r(s\cdot p)=r(p)$. Thus $r$ descends through the source $S$-torsor to $X_U=X\times_AU$, and then to $U$ by flat base change and $(\mathrm{alb}_X)_*\mathcal O_X=\mathcal O_A$, since $S$ is affine. Hence $r=c\circ a$ for a morphism $c:U\to S$. Define
$$\widetilde{\delta_g}^{'}(q):=c(b(q))\cdot\widetilde{\delta_g}(q).$$
By replacing $\widetilde{\delta_g}$ with $\widetilde{\delta_g}^{'}$, we have
$$\Phi\circ\widetilde\delta= \widetilde{\delta_g}\circ\Phi.$$
Thus conjugation by $\widetilde{\delta}$ normalizes $\Gamma= \mathrm{Gal}(\Phi)$. On overlaps, two source lifts differ by an $S$-valued section on the base. Since $S$ centralizes $\Gamma$, the induced conjugation action is independent of the lift and hence glues.

For $\gamma\in\Gamma$, put $\gamma'= \widetilde\delta\gamma\widetilde\delta^{-1}$. Recall that $K_B= \ker(\widehat B)\subseteq S$, where $\widehat B:S\to S$ is dual to $B$. In the central extension $1\to K_B\to\Gamma\to K\to 1$, the images of $\gamma'$ and $\gamma$ in $K$ agree because $G$ centralizes $K$. Hence $\gamma'\gamma^{-1}\in K_B$. Since $K_B\subseteq S$ acts by scalars on each homogeneous piece, the stabilizer in $G$ of a Cox degree centralizes the projective image of $\Gamma$ on that piece.

Finally, we check that the invariant algebra description is compatible with these actions. The line bundles $\mathcal L_{A,\mu}$ are defined by
$$f^*\mathcal L_\mu\cong\mathcal L_{B\mu}\otimes\mathrm{alb}_X^*\mathcal L_{A,\mu},$$
and the $\mathcal L_{A,\delta,\lambda}$ play the same role for pullback by $\delta$. Comparing the two pullbacks of $\mathcal L_\mu$ under $f\circ\delta= \beta(\delta)\circ f$ gives
$$\delta_A^*\mathcal L_{A,\mu}\otimes\mathcal L_{A,\delta,B\mu}\cong\mathcal L_{A,\rho_{\beta(\delta)}(\mu)}\otimes g^*\mathcal L_{A,\beta(\delta),\mu}.$$
Together with the equivariance of $\Phi$, this comparison makes the invariant algebra isomorphism
$$(\mathcal R^\Gamma)_{B\mu}\cong g^*\mathcal R_\mu\otimes\mathcal L_{A,\mu}^{-1}$$
compatible with the source action of $\delta$ and the target action of $\beta(\delta)$. Recall that the quotient $\mathfrak m^\Gamma/(\mathfrak m^\Gamma)^2$ records the generator classes of the invariant algebra. Passing to these quotients gives the same compatibility for $\bigl(\mathfrak m^\Gamma/(\mathfrak m^\Gamma)^2\bigr)_{B\mu}\cong g^*\mathcal V_\mu\otimes\mathcal L_{A,\mu}^{-1}$. The natural maps $\bigl(\mathfrak m^\Gamma/(\mathfrak m^\Gamma)^2\bigr)_{B\mu}\to\mathcal V_{B\mu}^\Gamma$ are compatible as well, since they come from the inclusion $\mathcal R^\Gamma\subseteq\mathcal R$.

\medskip

\noindent\textbf{Step 2.} We decompose $\mathcal V$ into line bundles permuted by $G$.

As in Proposition~\ref{prop: relative-generators}, use the Reynolds projector to write
$$\mathcal V_\lambda= \mathcal E_\lambda\oplus\mathcal W_\lambda,\qquad \mathcal E_\lambda= \mathcal V_\lambda^\Gamma,$$
where $\mathcal W_\lambda$ is the sum of the nontrivial isotypic components. Since the lifts of $G$ normalize $\Gamma$, they preserve its Reynolds projector. Thus both summands are carried to the corresponding summands in degree $\rho_{\delta}(\lambda)$ by \eqref{eq: semigraded-action}.

We first decompose the bundles $\mathcal W_\lambda$. By Step~1 of Proposition~\ref{prop: relative-generators}, each nontrivial isotypic component has the form $\mathcal A_W\otimes W$, where $\mathcal A_W$ is a line bundle and $W$ is a fixed nontrivial irreducible representation. The image of $\Gamma$ in $\mathrm{GL}(W)$ is a reflection group with no fixed vectors. Let $G_{\lambda, W}$ be the stabilizer of this component. Under the canonical identification
$$\mathbf{P}(\mathcal A_W\otimes_{\mathbb{C}}W)\cong A\times\mathbf{P}(W),$$
the line bundle factors in \eqref{eq: semigraded-action} disappear after projectivization. Thus, for each element of $G_{\lambda,W}$, the induced isomorphisms between projective fibers are described by a morphism $A\to\mathrm{PGL}(W)$. This morphism is constant because $A$ is projective and $\mathrm{PGL}(W)$ is affine. By Step~1, the resulting elements of $\mathrm{PGL}(W)$ centralize the projective image of the reflection group on $W$.

Lemma~\ref{lem: reflection-centralizer}, applied to the inverse image of this projective group in $\mathrm{GL}(W)$, gives a decomposition of $W$ into lines permuted by $G_{\lambda,W}$. Tensoring with $\mathcal A_W$ gives line subbundles. Choose this decomposition on one component in every $G$-orbit and transport it by \eqref{eq: semigraded-action}. Two choices of transport differ by an element of the component stabilizer, which permutes the chosen lines. They therefore give the same unordered set of line subbundles. We thus obtain decompositions of all $\mathcal W_\lambda$ into line bundles permuted by the isomorphisms induced by $G$.

We now decompose the fixed parts by induction on $\deg_B(\lambda)$, with $\deg_B$ as in Definition~\ref{def: relative degree}. By Step~1, $\deg_B$ is $G$-invariant. If $\mathcal E_\lambda\neq 0$, Step~2 of Proposition~\ref{prop: relative-generators} gives a degree $\mu$ occurring in $\mathcal V$ with
$$\lambda = B\mu,\qquad \mathcal E_{B\mu}\cong g^*\mathcal V_\mu\otimes\mathcal L_{A,\mu}^{-1},\qquad \deg_B(\mu)< deg_B(\lambda).$$
In particular, the fixed part vanishes in every Cox degree minimizing $\deg_B$ among those occurring in $\mathcal V$. At each later stage, transport the already chosen line decomposition of $\mathcal V_\mu$ through this isomorphism to obtain one of $\mathcal E_{B\mu}$. The compatibility proved in Step~1 shows that $\delta$ carries these lines to those obtained from $\mathcal V_{\rho_{\beta(\delta)}(\mu)}$. By \eqref{eq: equivariant-degree-map}, these are precisely the lines chosen in $\mathcal E_{\rho_\delta(B\mu)}$. This proves equivariance of the induction. Combining the decompositions of the fixed and nontrivial parts gives
$$\mathcal V= \bigoplus_{i=1}^e\mathcal N_i,$$
where $e= \mathrm{rank}(\mathcal V)$ and the $\mathcal N_i$ are homogeneous line bundles permuted by the isomorphisms induced by $G$.

\medskip

\noindent\textbf{Step 3.} We choose the Cox generators equivariantly.

It remains to lift these line subbundles to $\mathfrak m$ compatibly with the isomorphisms induced by $G$. Let $s_0:\mathcal V\to\mathfrak m$ be a homogeneous splitting given by Proposition~\ref{prop: relative-generators}. We make this splitting equivariant by averaging. Denote the isomorphisms induced by \eqref{eq: semigraded-action} on $\mathfrak m$ and $\mathcal V$ by $\phi^{\mathfrak m}_{\delta,\lambda}$ and $\phi^{\mathcal V}_{\delta,\lambda}$. For a homogeneous splitting $s$, define
$$(\delta * s)_{\rho_\delta(\lambda)}= \phi^{\mathfrak m}_{\delta,\lambda}\circ\delta_A^*s_\lambda\circ(\phi^{\mathcal V}_{\delta,\lambda})^{-1}.$$
Here the common factor $\mathcal L_{A,\delta,\lambda}$ cancels because tensoring both the source and the target by the same line bundle does not change their Hom bundle.
For $\delta,\epsilon\in G$, the composite of their chosen local lifts differs from the chosen local lift of $\delta\epsilon$ by an $S$-valued section. On the piece of Cox degree $\lambda$, this section acts by the same invertible function on $\mathfrak m_\lambda$ and $\mathcal V_\lambda$. It therefore cancels in the conjugation defining $\delta * s$, so
$$\epsilon *(\delta * s)=(\delta\epsilon)* s.$$
The order is reversed because the operations are defined by pullback. Thus they give an affine right action of $G$ on the space of homogeneous splittings of $\mathfrak m\to\mathcal V$.

Every $\delta * s_0$ is a right inverse to $\mathfrak m\to\mathcal V$, so their average
$$s = \frac{1}{|G|}\sum_{\delta\in G}\delta * s_0$$
is again a homogeneous splitting and is fixed by this action. Put $\mathcal M_i=s(\mathcal N_i)$ and $\mathcal M= \bigoplus_i\mathcal M_i$. Equivariance of $s$ shows that the isomorphisms induced by $G$ permute the $\mathcal M_i$. Since their direct sum maps isomorphically onto $\mathcal V$, Step~3 of the proof of Proposition~\ref{prop: relative-generators} shows that multiplication induces an isomorphism
$\mathrm{Sym}_{\mathcal O_A}(\mathcal M)\cong\mathcal R$.
\end{proof}

\begin{corollary}\label{cor: toric-boundary-descent}
Assume the hypotheses of Proposition~\ref{prop: equivariant relative generators}, and suppose that $G$ acts freely on $A$. Put $C=A/G$ and $Y=X/G$. Then there is a reduced divisor $D\subseteq Y$ such that $Y\to C$ is a toric fibration in the sense of Definition~\ref{def: toric fibration}, with relative toric boundary $D$. Hence $(Y,D)\to C$ is also a toric fibration in the sense of Meng and Zhong. If $D_A$ is the pullback of $D$ to $X$, then
$$(X,D_A)\cong\bigl(P\times^T F,P\times^T(F\setminus T)\bigr)$$
for a principal $T$-bundle $P\to A$, where $T$ is the dense torus of $F$.
\end{corollary}

\begin{proof}
Write $\mathcal M= \bigoplus_{i=1}^e\mathcal M_i$ as in Proposition~\ref{prop: equivariant relative generators}, and let $\lambda_i$ be the Cox degree of $\mathcal M_i$. Since $\mathcal M_i\subseteq\mathcal R_{\lambda_i}=(\mathrm{alb}_X)_*\mathcal L_{\lambda_i}$, evaluation gives a morphism of line bundles
$$\mathrm{alb}_X^*\mathcal M_i\to\mathcal L_{\lambda_i}.$$
In a local Cox presentation, this map is given by the coordinate corresponding to $\mathcal M_i$. Its zero divisor $D_i$ is therefore the corresponding coordinate divisor. Put $D_A= \sum_iD_i$. The construction in Proposition~\ref{prop: cox-reconstruction} identifies
$$(X,D_A)\cong\bigl(P\times^T F,P\times^T D_F\bigr),$$
where $D_F=F\setminus T$ is the toric boundary. Thus $D_A$ is a reduced relative simple normal crossing divisor.

The isomorphisms induced by $G$ permute the summands $\mathcal M_i$. Tensor factors pulled back from $A$ and scalar factors do not change their zero divisors, so $G$ permutes the divisors $D_i$. The action on $X$ is free because the action on $A$ is free. Hence $D_A$ descends through the finite \'{e}tale Galois morphism $X\to Y$ to a reduced divisor $D\subseteq Y$.

Since $A\to C$ is a $G$-torsor and $X\to A$ is equivariant, the quotient square is Cartesian:
$$X\cong Y\times_C A.$$
The finite quotient $Y=X/G$ is projective, so $Y\to C$ is projective. Smoothness and relative normal crossings descend along the finite \'{e}tale cover $A\to C$. After the base change $A\to C$, the family is $\bigl(P\times^T F,P\times^T D_F\bigr)$. Hence every fiber of $(Y,D)\to C$ is isomorphic to $(F,D_F)$. By \cite[Proposition~4.3.1]{AWZ21}, this gives a toric fibration in the sense of Definition~\ref{def: toric fibration}. Its \'{e}tale local product presentations give local product presentations in the analytic topology, so it is also a toric fibration in the sense of \cite[Definition~2.8]{MZ24}.
\end{proof}

\section{Proofs of the main results}

In this section, we finish the proof of our main results. We first prove Proposition~\ref{prop: smooth-mrc}.

\begin{proof}[Proof of Proposition~\ref{prop: smooth-mrc}]
Proposition~\ref{prop: smooth albanese} proves the assertion about the Albanese morphism.

Lemma~\ref{lem: invariant cover} gives a connected finite \'{e}tale Galois cover $\tau_0:Y\to X$ and an int-amplified lift $\widetilde f:Y\to Y$ making the following square Cartesian:
$$
\begin{tikzcd}
Y \arrow[r, "\widetilde f"] \arrow[d, "\tau_0"'] & Y \arrow[d, "\tau_0"]\\
X \arrow[r, "f"'] & X
\end{tikzcd}
$$
Moreover, we have
$$q(Y)= \widetilde q(Y)= \widetilde q(X).$$
Put $A_0= \mathrm{Alb}(Y)$, and let $\mathrm{alb}_Y:Y\to A_0$ be the Albanese morphism. Lemma~\ref{lem: maximal irregularity mrc} identifies $\mathrm{alb}_Y$ with the special MRC fibration and shows that its fibers are irreducible and rationally connected. Proposition~\ref{prop: smooth albanese} shows that $\mathrm{alb}_Y$ is smooth.

Put $G_0= \mathrm{Gal}(\tau_0)$. Albanese functoriality gives an action of $G_0$ on $A_0$ for which $\mathrm{alb}_Y$ is equivariant. The action of $G_0$ on $Y$ is free because $\tau_0$ is \'{e}tale. Lemma~\ref{lem: free action on base} therefore shows that the action on $A_0$ is free. Set $C=A_0/G_0$ and denote the quotient map by $\pi_0:A_0\to C$. Then $C$ is a smooth projective $Q$-abelian variety, while $\pi_0$ is finite, \'{e}tale, and Galois. Moreover,
$$\dim C= \dim A_0 = q(Y)= \widetilde q(X).$$

The $G_0$-invariant morphism $\pi_0\circ \mathrm{alb}_Y$ descends to a morphism $h:X\to C$. The natural diagram
\begin{equation}\label{eq: mrc-quotient-square}
\begin{tikzcd}
Y \arrow[r, "\mathrm{alb}_Y"] \arrow[d, "\tau_0"'] & A_0 \arrow[d, "\pi_0"]\\
X \arrow[r, "h"'] & C
\end{tikzcd}
\end{equation}
is Cartesian. Indeed, the induced morphism $Y\to X\times_C A_0$ is a morphism between finite \'{e}tale covers of $X$, both of degree $|G_0|$. If two closed points $y_1,y_2\in Y$ have the same image, then $y_2= \gamma(y_1)$ for some $\gamma\in G_0$ and $\mathrm{alb}_Y(y_2)= \mathrm{alb}_Y(y_1)$. Thus $\gamma$ fixes $\mathrm{alb}_Y(y_1)$, so the freeness of the action on $A_0$ gives $\gamma = 1$. The morphism is therefore injective on every closed fiber. Since the two covers have the same degree, it is an isomorphism.

Since $X$ is projective, the morphism $h$ is projective. It follows from \eqref{eq: mrc-quotient-square} by \'{e}tale descent that $h$ is smooth and surjective with connected rationally connected fibers. The variety $C$ contains no rational curves. Indeed, after pulling back $\pi_0$ along a morphism $\mathbf{P}^1\to C$, every connected component of the resulting finite \'{e}tale cover of $\mathbf{P}^1$ maps isomorphically to $\mathbf{P}^1$. The morphism therefore lifts to $A_0$ and is constant. Consequently, every rational curve on $X$ lies in a fiber of $h$, while every fiber is rationally connected. Thus $h$ represents the maximal rationally connected fibration of $X$.
\end{proof}

\medskip

\begin{proof}[Proof of Theorem~\ref{thm: main}]
\noindent\textbf{Step 1.} We construct the smooth MRC fibration and record the commuting Galois actions.

Lemma~\ref{lem: invariant cover} gives a connected finite \'{e}tale Galois cover $\tau_0:Y\to X$ and a lift $\widetilde f:Y\to Y$ making the following square Cartesian:
\begin{equation}\label{eq: main-cartesian-lift}
\begin{tikzcd}
Y \arrow[r, "\widetilde f"] \arrow[d, "\tau_0"'] & Y \arrow[d, "\tau_0"]\\
X \arrow[r, "f"'] & X
\end{tikzcd}
\end{equation}
Since $f$ is Galois and int-amplified, so is $\widetilde f$, and
$$q(Y)= \widetilde q(Y).$$
Fix an ample Cartier divisor $H_X$ such that $f^*H_X-H_X$ is ample, and put $H_Y= \tau_0^*H_X$. Then $\widetilde f^*H_Y-H_Y$ is ample, and $H_Y$ is invariant under $G_0:= \mathrm{Gal}(\tau_0)$. We use $H_Y$ in the fiberwise and relative Cox constructions below.

Apply Proposition~\ref{prop: albanese-structure}(1) to $(Y,\widetilde f)$, and choose the Albanese morphism $\mathrm{alb}_Y:Y\to A_0:= \mathrm{Alb}(Y)$ so that the induced endomorphism is an int-amplified isogeny $g_0:A_0\to A_0$. Using the construction in the proof of Proposition~\ref{prop: smooth-mrc}, we obtain the smooth $Q$-abelian quotient $C=A_0/G_0$ and a morphism $h:X\to C$ representing the maximal rationally connected fibration. With $\pi_0:A_0\to C$ the quotient map, the square \eqref{eq: mrc-quotient-square} is Cartesian.

Proposition~\ref{prop: isotrivial-toric}, applied to $(Y,\widetilde f)$, shows that $\mathrm{alb}_Y$ is \'{e}tale locally trivial with a fixed smooth projective toric fiber $F$.

Let $K\subseteq\mathrm{Gal}(\widetilde f)$ be the subgroup preserving every fiber of $\mathrm{alb}_Y$, as in Proposition~\ref{prop: albanese-structure}(2). Lemma~\ref{lem: cartesian galois actions} gives an automorphism $\alpha:G_0\to G_0$ such that
$$\widetilde f\circ\gamma= \alpha(\gamma)\circ\widetilde f$$
for every $\gamma\in G_0$, and it shows that every elements in $G_0$ commute with those in $\mathrm{Gal}(\widetilde f)$. In particular, the elements in $G_0$ commute with the elements in $K$, and the induced affine actions on $A_0$ satisfy
$$g_0\circ\gamma= \alpha(\gamma)\circ g_0.$$

\medskip

\noindent\textbf{Step 2.} We construct a finite \'{e}tale cover of $X$ whose Albanese morphism is a split toric fibration and descend the toric boundary to $X$.

If $\dim F=0$, then $\mathrm{alb}_Y$ is an isomorphism, and the assertions of this step hold with $A=A_0$ and $\pi= \pi_0$. For the remainder of this step, assume that $\dim F > 0$.

Starting from $(Y,\widetilde f,H_Y)$, perform the finitely many base changes by multiplication isogenies required for the relative Cox construction of Section~\ref{sec: relative-cox}. Denote the composite of these isogenies by
$$[m]:A\to A_0,$$
where $A$ is a second copy of $A_0$. We use the same notation $g_0:A\to A$ for the corresponding isogeny. It commutes with $[m]$. Put $\pi= \pi_0\circ[m]:A\to C$. For $\gamma\in G_0$, write
$$\gamma(x)= \ell_\gamma(x)+c_\gamma.$$
By Lemma~\ref{lem: multiplication-galois-cover} and the description in its proof, $\pi$ is a connected finite \'{e}tale Galois cover with group
$$G= \{(\gamma,b):\gamma\in G_0,\ b\in A,\ [m]b=c_\gamma\}.$$
Set
$$Z=X\times_C A\cong Y\times_{A_0,[m]}A,$$
and let $p_Z:Z\to A$ be the projection. Since $(\mathrm{alb}_Y)_*\mathcal O_Y= \mathcal O_{A_0}$, flat base change gives $(p_Z)_*\mathcal O_Z= \mathcal O_A$. Thus $Z$ is connected, and $Z\to X$ is a connected finite \'{e}tale Galois cover with group $G$.

Under the identification $Z\cong Y\times_{A_0,[m]}A$, the actions of $G$ and $K$ are
$$(\gamma,b)\cdot(y,x)=(\gamma(y),\ell_\gamma(x)+b),\qquad k\cdot(y,x)=(k(y),x).$$
Since every elements in $G_0$ commute with those in $K$, these actions commute as well.

The endomorphism $\widetilde f$ lifts to
$$f_A:Z\to Z,\qquad f_A(y,x)=(\widetilde f(y),g_0(x)).$$
Together with $\widetilde f\circ\gamma= \alpha(\gamma)\circ\widetilde f$, the construction in the proof of Lemma~\ref{lem: multiplication-galois-cover} gives a homomorphism
$$\beta:G\to G,\qquad \beta(\gamma,b)=(\alpha(\gamma),g_0(b)),$$
such that $f_A\circ\delta= \beta(\delta)\circ f_A$ for every $\delta\in G$. We use only that $\beta$ is a homomorphism. It need not be an automorphism.

Let $H_Z$ be the pullback of $H_Y$ to $Z$. Both $H_Z$ and $f_A^*H_Z-H_Z$ are ample, so $f_A$ is int-amplified. It is finite and surjective, as in the base change discussion at the start of Section~\ref{sec: relative-cox}. Moreover, $H_Z$ is $G$-invariant, so $G$ preserves its fiberwise class $H_F$.

Since $p_Z$ is the base change of $\mathrm{alb}_Y$, it has connected rationally connected fibers. Every morphism from $Z$ to an abelian variety is constant on these fibers and hence factors uniquely through $p_Z$ by the rigidity lemma \cite[Lemma~1.15]{Deba01}. Thus $p_Z$ is the Albanese morphism, after choosing compatible base points.

By the base change argument at the start of Section~\ref{sec: relative-cox}, the relative quotient by $K$ is
$$Z/K\cong Z\times_{A,g_0}A.$$
We can therefore apply Proposition~\ref{prop: equivariant relative generators} to the Albanese morphism $p_Z:Z\to A$ and its relative Cox algebra, with the $G$-action and the homomorphism $\beta$ constructed above. Since $G$ acts freely on $A$, that proposition and Corollary~\ref{cor: toric-boundary-descent} give a $G$-invariant reduced divisor $D_A\subseteq Z$, a principal $T$-bundle $P\to A$, and a reduced divisor $D\subseteq X$ such that, writing $\partial F=F\setminus T$,
$$(Z,D_A)\cong\bigl(P\times^T F,P\times^T\partial F\bigr),\qquad D_A=(Z\to X)^*D.$$
Moreover, $(X,D)\to C$ is a toric fibration. In particular, all its fibers are isomorphic to $(F,\partial F)$.

\medskip

\noindent\textbf{Step 3.} We descend the endomorphism and determine its Galois group.

By Step 1, the isogeny $g_0:A_0\to A_0$ satisfies $g_0\circ\gamma= \alpha(\gamma)\circ g_0$ for every $\gamma\in G_0$, where $\alpha$ is an automorphism of $G_0$. Hence $g_0$ descends to a surjective endomorphism $f_C:C\to C$, giving the commutative diagram
$$
\begin{tikzcd}
A_0 \arrow[r, "g_0"] \arrow[d, "\pi_0"'] & A_0 \arrow[d, "\pi_0"]\\
C \arrow[r, "f_C"'] & C
\end{tikzcd}
$$
The identities defining $h$ and the lift $\widetilde f$ give $h\circ f\circ\tau_0= \pi_0\circ g_0\circ\mathrm{alb}_Y=f_C\circ h\circ\tau_0$. Since $\tau_0$ is surjective, we obtain $h\circ f=f_C\circ h$. The surjectivity of $h$ gives the uniqueness of $f_C$. If $C$ is a point, the remaining assertions are immediate, so assume that $\dim C>0$.

By \cite[Lemma~3.4]{Meng20}, the endomorphism $f_C$ is int-amplified and hence finite. If $f$ is $q$-polarized, then \cite[Theorem~3.11]{MZ18} shows that $f_C$ is $q$-polarized. Since $\pi_0$ is surjective and \'{e}tale and $f_C\circ\pi_0= \pi_0\circ g_0$ is \'{e}tale, \'{e}tale locality on the source shows that $f_C$ is \'{e}tale.

By Proposition~\ref{prop: albanese-structure}(2), the action of $\mathrm{Gal}(\widetilde f)$ on $A_0$ consists of all translations by points of $\ker(g_0)$. Since every elements in $G_0$ commute with those in $\mathrm{Gal}(\widetilde f)$ on $Y$ and $\mathrm{alb}_Y$ is surjective, its action on $A_0$ commutes with all these translations. They therefore descend to an action of $\ker(g_0)$ on $C=A_0/G_0$. For $k\in\ker(g_0)$, let $t_k$ denote translation by $k$. Then $g_0\circ t_k=g_0$ shows that the descended action is by automorphisms over $f_C$.

This action is faithful. Indeed, if $t_k$ induces the identity on $C$, then $t_k= \gamma$ for some $\gamma\in\mathrm{Gal}(\pi_0)=G_0$. Hence
$$g_0=g_0\circ t_k=g_0\circ\gamma= \alpha(\gamma)\circ g_0.$$
Since $g_0$ is surjective, $\alpha(\gamma)=1$. As $\alpha$ is an automorphism, $\gamma=1$ and $k=0$.

Finally, the commutative diagram gives $\deg f_C= \deg g_0=|\ker(g_0)|$. Thus the faithful group of automorphisms over $f_C$ has order $\deg f_C$, so $f_C$ is Galois and
$$\mathrm{Gal}(f_C)\cong\ker(g_0).$$
In particular, $\mathrm{Gal}(f_C)$ is abelian. This proves the theorem.
\end{proof}

\begin{proof}[Proof of Corollary~\ref{cor: albanese toric cover}]
Take $\widetilde X=X\times_C A$ for the cover $A\to C$ in Theorem~\ref{thm: main}, and let $\mathrm{alb}_{\widetilde X}:\widetilde X\to A$ be the projection. Since $h$ is proper with geometrically connected fibers, $h_*\mathcal O_X= \mathcal O_C$. Flat base change gives $(\mathrm{alb}_{\widetilde X})_*\mathcal O_{\widetilde X}= \mathcal O_A$, so $\widetilde X$ is connected. The morphism $\widetilde X\to X$ is the base change of the finite \'{e}tale Galois cover $A\to C$ and is therefore finite, \'{e}tale, and Galois. The theorem identifies $\mathrm{alb}_{\widetilde X}$ with the Albanese morphism of $\widetilde X$ and gives an isomorphism $\widetilde X\cong P\times^T F$ over $A$. Hence $\mathrm{alb}_{\widetilde X}$ is a split toric fibration.
\end{proof}

As an application, we show Theorem~\ref{thm:log-cy}.

\begin{proof}[Proof of Theorem~\ref{thm:log-cy}]
    By Theorem~\ref{thm: main} and Corollary~\ref{cor: albanese toric cover}, there exists a finite \'etale cover $\tau: \widetilde{X}\to X$ from a smooth variety $\widetilde{X}$ such that the Albanese map $\mathrm{alb}_{\widetilde{X}}: \widetilde{X}\to A:=\mathrm{Alb}(\widetilde{X})$ is a split toric fibration. In particular, there exists a reduced divisor $D$ on $X$ such that $$ (\widetilde X,\widetilde D:=\tau^*D)\cong(P\times^T F,P\times^T\partial F) $$ over $A$,
where $F$ is a smooth projective toric variety with dense torus $T$, $\partial F=F\setminus T$, and $P\to A$ is a principal $T$-bundle.

A torus invariant rational top form on $F$ has divisor $-\partial F$ \cite[Section~8.1]{CLS11}. The transition maps of $P\times^T F$ preserve this form when viewed as a relative form, so
$$ K_{\widetilde X/A}+\widetilde D\sim0. $$
Since $A$ is abelian, $K_A\sim0$, and hence $K_{\widetilde X}+\widetilde D\sim0$.

Moreover, $\partial F$ has simple normal crossings, so the local product presentations show that $\widetilde D$ also has simple normal crossings. Thus $(\widetilde X,\widetilde D)$ is log canonical, and the same holds for $(X,D)$ because $\tau$ is surjective and \'{e}tale. The canonical divisor formula gives
$$ \tau^*(K_X+D)\sim K_{\widetilde X}+\widetilde D\sim0. $$
Pushing forward by $\tau$, we obtain $(\deg\tau)(K_X+D)\sim0$. Therefore $K_X+D\sim_{\mathbb Q}0$, and $(X,D)$ is a log Calabi--Yau pair.
\end{proof}

\medskip

Motivated by the quotient constructions in \cite{MY21} and \cite[Example~3.10]{Yosh21b}, we provide an example showing that the smoothness assumption in Theorem~\ref{thm: rationally connected} cannot be removed.

\begin{example}\label{ex: klt-counterexample}
Let $E= \mathbb{C}/(\mathbb{Z}+i\mathbb{Z})$ be the elliptic curve with complex multiplication by $i$, and put $Y= \mathbf{P}^1\times E$. Consider the involution
$$\iota([x:y],a)=([y:x],-a),$$
and let $\pi:Y\to X=Y/\langle\iota\rangle$ be the quotient. The fixed locus of $\iota$ is the product of the set of fixed points of $[x:y]\mapsto[y:x]$ and $E[2]$; these sets have two and four points, respectively. At each fixed point, the differential of $\iota$ acts as $-\mathrm{id}$ on the tangent space. Thus $\pi$ is quasi-\'{e}tale, and $X$ has eight quotient singularities of type $A_1:= \mathbb{C}^2/\frac{1}{2}(1,1)$. In particular, $X$ is a $\mathbb{Q}$-factorial canonical surface.

Projection to the second factor induces a morphism $X\to E/\langle-1\rangle\cong\mathbf{P}^1$. Its general fiber is a smooth rational curve, and the image of $\{[1:1]\}\times E$ is a section. The generic fiber therefore has a $\mathbb{C}(\mathbf{P}^1)$-point and is isomorphic to $\mathbf{P}^1$ over $\mathbb{C}(\mathbf{P}^1)$. Hence $\mathbb{C}(X)$ is a purely transcendental extension of $\mathbb{C}$ of transcendence degree two, so $X$ is rational and in particular rationally connected.

The surface $X$ is not of Fano type. Indeed, if it is, then the quasi-\'{e}tale cover $Y$ is of Fano type by \cite[Proposition~6.1]{Yosh21b}. A variety of Fano type is rationally connected \cite{Zhang06}, whereas $Y$ maps onto the elliptic curve $E$, a contradiction. Since every projective toric variety is of Fano type \cite[Example~11.4.26]{CLS11}, the surface $X$ above is not toric.

Define $\varphi:\mathbf{P}^1\to\mathbf{P}^1$ by $\varphi([x:y])=[x^2:y^2]$ and $\beta:E\to E$ by $\beta(a)=(1+i)a$. Then $\widetilde f:= \varphi\times\beta$ commutes with $\iota$ and descends to a finite endomorphism $f:X\to X$ of degree four. The Galois group of $\widetilde f$ is
$$\mathrm{Gal}(\widetilde f)= \langle\delta\rangle\times\ker\beta\cong(\mathbb{Z}/2\mathbb{Z})^2,$$
where $\delta([x:y])=[-x:y]$ and $\ker\beta$ acts by translations on $E$. This group commutes with $\iota$, since $\delta$ commutes with the coordinate swap and $\ker\beta\subseteq E[2]$. Its intersection with $\langle\iota\rangle$ is trivial, so its descended action over $f$ is faithful. Since its order equals $\deg f=4$, the endomorphism $f$ is Galois.

The product of the quotient maps for $[x:y]\mapsto[y:x]$ on $\mathbf{P}^1$ and $a\mapsto-a$ on $E$ factors through a finite morphism $\nu:X\to\mathbf{P}^1\times\mathbf{P}^1$. The maps $\varphi$ and $\beta$ descend to degree-two endomorphisms of the respective quotient curves. Hence the line bundle $\mathcal L:= \nu^*\mathcal O_{\mathbf{P}^1\times\mathbf{P}^1}(1,1)$ is ample and satisfies
$$f^*\mathcal L\cong\mathcal L^{\otimes2}.$$
Thus $X$ is a rationally connected $\mathbb{Q}$-factorial canonical surface with a $2$-polarized Galois endomorphism, but $X$ is not toric.
\end{example}

\end{document}